\documentclass[10,reqno]{amsart}
\usepackage{amssymb}
\usepackage{amsmath}
\usepackage{amsfonts}
\usepackage[mathscr]{eucal}
\usepackage{graphicx}
\usepackage{multicol}
\usepackage{caption2}
\let\phi\varphi

\def\om {\omega }

\def\pd {\partial }

\def\bN{{\Bbb N}}

\newtheorem{deff}{Definition}[section]
\newtheorem{cor}{Corollary}[section]
\newtheorem{lem}{Lemma}[section]

\newtheorem{theorem}{Theorem}[section]
\newtheorem{prop}{Proposition}[section]
\newtheorem{condition}{Condition}[section]

\makeatletter
\@addtoreset{equation}{section}
\@addtoreset{theorem}{section}

 \makeatletter\renewcommand{\@@and}{and} \makeatother

\title[Non-local stabilization of NPE] 
      {Nonlocal stabilization by starting control of the normal equation generated from Helmholtz system}

\author{A.\, V.\, Fursikov}
\address{Department of Mechanics \& Mathematics, Moscow State University, Moscow
Russia 119991}
\author{L.\, S.\, Shatina}
\address{Department of Mechanics \& Mathematics, Moscow State University, Moscow}


\begin{document}

\def\OldComma{,}
    \catcode`\,=13
        \def,{%
            \ifmmode%
            \OldComma\discretionary{}{}{}%
            \else%
            \OldComma%
        \fi%
        }%
\begin{abstract}
We consider the semilinear parabolic equation
of normal type connected with the 3D Helmholtz equation with periodic boundary
condition.
The problem of stabilization to zero of the solution for normal parabolic equation with
arbitrary initial condition by starting control is studied. This
problem is reduced to establishing three inequalities connected with starting control, one of
which has been proved in \cite{F5}, \cite{FSh}. The proof for the other two is
given here.
\end{abstract}

\maketitle

\keywords
\section{Introduction}

This work is devoted to the construction of  theory of the nonlocal
stabilization by starting control for the equations of
normal type connected with thee-dimensional Helmholtz system,
describing the curl $\omega $ of the velocity vector field $v$ for
the viscous incompressible fluid flow. As well-known, velocity
vector field $v$ is described with Navier-Stokes system. Up to now
there exists extensive literature on the local stabilization of
Navier-Stokes system in the neighborhood of a stationary point
(see for example, \cite{F6}, \cite{BLT}, \cite{R1},  \cite{FG} ,
as well as literature listed in the review \cite{FK})  but
construction of its  nonlocal analog has not been started yet.
Note that for some equations of fluid dynamics there are certain
nonlocal stabilization results: for Burgers equation where exact
formula of its solution was used (see \cite{K}), and for Euler
equations (see \cite{C1},\cite{C2}) where the construction is
based on such properties of its solutions which Navier-Stokes
system does not possess. We have to note also that nonlocal exact
controllability of the Navier-Stokes system by distributed control
supported in sub domain has been proved in \cite{CF} for 2D case
and in \cite{FI} for 3D case. Since settings of exact
controllability and stabilization problems are related in some
sense, this give us the hope that nonlocal stabilization problem
can be solved.

So our  the main goal is to construct nonlocal stabilization
theory by feedback control for Navier-Stokes system, and, as we
hope, the problem we solve in this work is some essential part in
achieving of this goal. As was mentioned above in this work we
have deal with semilinear normal parabolic equations (NPE). Such
equations have been introduced and studied in \cite{F1}-\cite{F4}
to understand better the dynamical structure for hydrodynamical
equations of Navier-Stokes type and first of all for Helmholtz
equations.

Let us explain how NPE arise and recall their definition. As
well-known, the existence proof of weak solutions $v$ to
Navier-Stokes equations is based on energy estimate for fluid
velocity $v$. However, similar existence proof of a strong solution
(proven to be unique) is impeded because the solution of Helmholtz
equation $\omega~=~\mbox{curl}v$ does not satisfy the energy
estimate (see details below, in section \ref{s2}). \footnote{In
two-dimensional case $\omega $ satisfies the energy estimate,
which allowed V.I. Yudovitch \cite{Yud} to prove the existence of
a strong solution even for Euler equations of ideal incompressible
liquid.} The latter is due to the fact that the image $B(\omega )$
of the nonlinear operator $B$, generated by non-linear members of
Helmholtz equation, is not orthogonal to vector $\omega$, i.e. it
contains component $\Phi (\omega )\omega$, collinear to $\omega$
(here $\Phi (\omega)$ is a certain functional).

If the non-linear members of Helmholtz system that define
$B(\omega )$  are substituted by members defining $\Phi (\omega)
\omega$, then the resulting system of equations is called (by
definition) a normal parabolic equation (NPE), corresponding to
Helmholtz system.

Note that everywhere in this paper we consider  the solutions
$y(t,x)$ of NPE that satisfy periodic boundary conditions on
spatial variables $x=(x_1,x_2,x_3)$. In other words we look for
solution $y(t,x)$ of NPE for $x\in \mathbb{T}^3$ where
$\mathbb{T}^3=(\mathbb{R}/2\pi \mathbb{Z})^3$ is 3D torus. In
\cite{F1}-\cite{F4} the structure of dynamics generated by NPE
corresponding to Burgers equation as well as to 3D Helmholtz
system was described. In particular it was established that the solution $y(t,x)$ of NPE
either tends to zero or to infinity as $t\to \infty$, or blows
up, i.e. $\| y(t,\cdot)\| \to \infty$ as $t\to t_0\ne \infty$ depending
 on the initial condition.
This result makes  the following setting of non-local
stabilization problem by starting control for NPE corresponding to
3D Helmholtz system reasonable:

Let $0<a_j<b_j<2\pi , j=1,2,3$ be fixed. Given divergence free
initial condition $y_0(x)\in L_2(\mathbb{T}^3)$ for NPE connected
with 3D Helmholtz system (NPEH), find divergence free starting
control $u_0(x)\in L_2(\mathbb{T}^3)$ supported in
$[a_1,b_1]\times [a_2,b_2]\times [a_3,b_3]\subset \mathbb{T}^3$ such
that the solution $y(t,x)$ of NPEH with initial condition
$y_0+u_0$ satisfies the inequality
\begin{equation}\label{1.0}
      \| y(t,\cdot)\|_{L_2(\mathbb{T}^3)}\le \alpha \|
      y_o+u_0\|_{L_2(\mathbb{T}^3)}e^{-t}\quad \forall t>0
\end{equation}
with some $\alpha >1$.

The problem formulated above was solved in this paper.
 Namely, we proved that the NPEH with arbitrary
initial condition $y_0$ can be stabilized by starting control in
the form
\begin{equation}\label{1.01}
u(x) =Fy_0+ \lambda u_0(x),
\end{equation}
where $Fy_0$ is a certain feedback control with feedback operator
$F$ constructed by some technic of local stabilization theory (see
\cite{F6}, \cite{FG}), $\lambda$ is a constant, depending on
$y_0$, and $u_0$ is a universal function, depending only on a
given arbitrary parallelepiped $[a_1,b_1]\times [a_2,b_2]\times
[a_3,b_3]\subset \mathbb{T}^3$, which contains the support of
control $u_0$. The following estimate is the key one for the proof
of the stabilization result:
\begin{equation}\label{1.1}
   \int_{\mathbb{T}^3}((\mathbf{S}(t,{x};{u}),\nabla)\operatorname{rot}^{-1}\mathbf{S}(t,{x};{u}),\mathbf{S}(t,{x};{u}))dx>
   \beta e^{-18t} \qquad \forall \; t\ge 0
\end{equation}
where $S(t,x;u_0)$ is the solution of the Stokes equation with
initial condition $u_0$  and $\beta
>0$ is some constant.

We should note that the proof of the estimate (\ref{1.1}) is very
complicated, and the largest part of the paper is devoted just to this
proof. Before investigation of nonlocal stabilization of NPEH we
have studied analogous problem for NPE connected with 1D Burgers
equation (see \cite{F5},\cite{FSh}). The most difficult part of
that works was to prove some analog of the bound \eqref{1.1}. It
is complicated but essentially easier than \eqref{1.1} because it
is an estimate for a one-dimensional integral. The first essential
step to the proof of inequality \eqref{1.1} was reducing this
inequality to several similar estimates for
one-dimensional integrals. To prove these bounds for
one-dimensional integrals we used some development of the technic
which had been worked out in \cite{F5}, \cite{FSh}.

In the  section \ref{s1} we remind the definitions and some facts
concerning NPE connected with 3D Helmholtz system, section
\ref{s2} is devoted to formulation of the main stabilization
result and result connected with the estimate \eqref{1.1}.
Besides, in subsection \ref{s2.3} we construct feedback operator
$F$ from \eqref{1.01}, and at last in subsection \ref{s2.4} we
derive the main nonlocal stabilization result from the bound
\eqref{1.1}. The rest part of the paper i.e. sections
\ref{s3}-\ref{s6} are devoted to the proof of the estimate
(\eqref{1.1}).

In conclusion let us note that we hope to use results obtained in
this paper for constructing the  nonlocal stabilization
theory of Navier-Stokes and Helmholtz systems by impulse, or
distributed, or boundary controls. We expect that these results and
technics of local stabilization theory should be enough to do so.

\section{Semilinear parabolic equation of normal type}\label{s1}

In this section we recall basic information  on parabolic
equations of normal type corresponding to 3D Navier-Stokes system:
their derivation, explicit formula for solutions, theorem on
existence and uniqueness of solution for normal parabolic
equations, the structure of their dynamics. These results have
been obtained in \cite{F1}-\cite{F4}

\subsection{Navier-Stokes equations}\label{s1.1}

Let us consider 3D Navier-Stokes system
\begin{equation}\label{NavierStokes}
      \partial_t {v}(t,{x})- \Delta {v}(t,{x})+({v},\nabla){v}+\nabla{p}(t,{x})=0, \, \operatorname{div} {v}=0,
\end{equation}
with periodic boundary conditions
\begin{equation}\label{NavierStokes_boundcond}
   {v}(t,...,x_i,...)={v}(t,...,x_i+2\pi,...), \, i=1,2,3
\end{equation}
and initial condition
\begin{equation}\label{NavierStokes_incond}
  {v}(t,{x})|_{t=0}={v}_0({x})
\end{equation}
where $t\in \mathbb{R}_{+}$, ${x}=(x_1,x_2,x_3)\in\mathbb{R}^3$, ${v}(t,{x})=(v_1,v_2,v_3)$ is the velocity vector field of fluid flow, $\nabla p$ is the gradient of pressure, $\Delta$ is the Laplace operator, ${({v},\nabla){v}=\sum_{j=1}^{3}v_j\partial_{x_j}v}$. Periodic boundary conditions \eqref{NavierStokes_boundcond} mean that Navier-Stokes eqautions \eqref{NavierStokes} and initial conditions \eqref{NavierStokes_incond} are defined on torus $\mathbb{T}^3=(\mathbb{R}/2\pi\mathbb{Z})^3$.

For each $m\in \mathbb{Z}_+=\{j\in \mathbb{Z}:j\geq 0\}$ we define the space
\begin{equation}\label{phase_space}
V^m=V^m(\mathbb{T}^3)=\{v(x)\in(H^m(\mathbb{T})^3)^3:\mbox{div}v=0,\int_{\mathbb{T}^3}v(x)dx=0\}
\end{equation}
where $H^m(\mathbb{T}^3)$ is the Sobolev space.

It is well-known, that the non-linear term $(v,\nabla)v$ in problem \eqref{NavierStokes}-\eqref{NavierStokes_incond} satisfies relation
$$\int_{\mathbb{T}^3}(v(t,x),\nabla)v(t,x)\cdot v(t,x)dx=0.$$

Therefore, multiplying \eqref{NavierStokes} scalarly by $v$ in $L_2(\mathbb{T}^3)$, integrating by parts by $x$,
and then integrating by $t$, we obtain the well-known energy estimate
\begin{equation}\label{en_est}
\int_{\mathbb{T}^3}|v(t,x)|^2dx+2\int_{0}^t\int_{\mathbb{T}^3}|\nabla_x v(\tau, x)|^2dx d\tau\leq\int_{\mathbb{T}^3}|v_0(x)|^2dx,
\end{equation}
which allows to prove the existence of weak solution for
\eqref{NavierStokes}-\eqref{NavierStokes_incond}. But, as is
well-known, scalar multiplication of \eqref{NavierStokes} by $v$
in $V^1(\mathbb{T}^3)$ does not result into an analog of estimate
\eqref{en_est}. Nevertheless, expression of such kind will be
useful for us. More exactly, we will consider the scalar product
in $V^0$ of Helmholtz equations by its unknown vector field (which
is equivalent).


\subsection{Helmholtz equations}
Using problem
\eqref{NavierStokes}-\eqref{NavierStokes_incond} for fluid
velocity $v$, let us derive the similar problem for the curl of velocity
\begin{equation}\label{curl}
{\omega}(t,x)=\operatorname{curl}v(t,x) = (\partial_{x_2}v_3 - \partial_{x_3}v_1,\partial_{x_3}v_1 - \partial_{x_1}v_3,\partial_{x_1}v_2 - \partial_{x_2}v_1)
\end{equation}
from it.

It is well-known from vector analysis, that
\begin{equation}\label{vectan1}
(v,\nabla)v = \omega\times v+\nabla\frac{|v|^2}{2},
\end{equation}
\begin{equation}\label{vectan2}
\operatorname{curl}(\omega\times v)=(v,\nabla)\omega-(\omega,\nabla)v, \text{ if } \operatorname{div}v=0,\, \operatorname{div}\omega=0.
\end{equation}
where $\omega\times v = (\omega_2v_3 - \omega_3 v_1,\omega_3v_1 -
\omega_1v_3,\omega_1v_2 - \omega_2v_1)$ is the vector product of
$\omega$ and $v$, and $|v|^2=v_1^2+v_2^2+v_3^2$. Substituting
\eqref{vectan1} into \eqref{NavierStokes} and applying curl
operator to both sides of the obtained equation, taking into
account \eqref{curl}, \eqref{vectan2} and formula
$\operatorname{curl}\nabla F=0$, we obtain the Helmholtz equations
\begin{equation}\label{Helmholtz}
\partial_t\omega(t,x)-\Delta \omega+(v,\nabla)\omega-(\omega,\nabla)v=0
\end{equation}
with initial conditions
\begin{equation}\label{Helmholtz_incond}
\omega(t,x)|_{t=0}=\omega_0(x):=\operatorname{curl}v_0(x),
\end{equation}
and periodic boundary condition.


\subsection{Derivation of Normal Parabolic Equations (NPE)}\label{s1.2}
Using decomposition into Fourier series
\begin{equation}\label{Fourier_decomp}
v(x)=\sum_{k\in\mathbb{Z}^3}\hat{v}(k)e^{i(k,x)}, \, \, \hat{v}(k)=(2\pi)^{-3}\int_{\mathbb{T}^3}v(x)e^{-i(k,x)}dx,
\end{equation}
where $(k,x)=k_1\cdot x_1+k_2\cdot x_2+k_3\cdot x_3$, $k=(k_1,k_2,k_3)$, and the well-known formula $\mbox{curl}\,\mbox{curl}\,v = -\Delta v$, if $\mbox{div }\,v =0$, we see that inverse operator to curl is well-defined on space $V^m$ and is given by formula
\begin{equation}\label{rotinv}
\mbox{curl}^{-1}\omega(x)=i\sum_{k\in\mathbb{Z}^3}\frac{k\times\hat{\omega}(k)}{|k|^2}e^{i(k,x)}.
\end{equation}
Therefore, operator $\mbox{curl}:V^1\mapsto V^0$ realizes isomorphism of the spaces, thus, a sphere in $V^1$ for
 \eqref{NavierStokes}-\eqref{NavierStokes_incond} is equivalent to a sphere in $V^0$
 for the problem \eqref{Helmholtz}-\eqref{Helmholtz_incond}.

Let us denote the non-linear term in Helmholtz system by $B$:
\begin{equation}\label{B_term}
B(\omega)=(v,\nabla)\omega-(\omega,\nabla)v,
\end{equation}
where $v$ can be expressed in terms of $\omega$ using
\eqref{rotinv}.

Multiplying \eqref{B_term} scalarly by $\omega=(\omega_1, \omega_2, \omega_3)$ and integrating by parts, we get expression
\begin{equation}\label{B_term_mult}
(B(\omega),\omega)_{V_0}=-\int_{\mathbb{T}^3}\sum_{j,k=1}^{3}\omega_j\partial_jv_k\omega_kdx,
\end{equation}
that, generally speaking, is not zero. Hence, energy estimate for
solutions of 3-D Helmholtz system is not fulfilled. In other
words, operator $B$ allows decomposition
\begin{equation}\label{B_term_decomp}
B(\omega) = B_n(\omega)+B_{\tau}(\omega),
\end{equation}
where vector $B_n(\omega)$ is orthogonal to sphere
$\Sigma(\|\omega\|_{V^0})=\{u\in V^0:\|u\|_{V^0}=\|\omega\|_{V^0}
\}$ at the point $\omega$, and vector $B_{\tau}$ is tangent to
$\Sigma(\|\omega\|_{V^0})$ at $\omega$. In general, both terms in
\eqref{B_term_decomp} are not equal to zero. Since the presence of $B_n$,
and not of $B_{\tau}$, prevents the fulfillments of the energy
estimate, it is plausible that the just $B_n$ generates the possible
singularities in the solution. Therefore, it seems reasonable to
omit the $B_{\tau}$ term in Helmholtz system and study the system \eqref{Helmholtz}
with non-linear operator $B(\omega)$ replaced with $B_n(\omega)$ first.
We will call the obtained system the system of normal parabolic
equations (NPE system).

Let us derive the NPE system corresponding to \eqref{Helmholtz}-\eqref{Helmholtz_incond}.

Since summand $(v,\nabla)\omega$ in \eqref{B_term} is  tangential
to vector $\omega$, the normal part of operator $B$ is defined by
the summand $(\omega,\nabla)v$. We shall seek for it in the form
$\Phi(\omega)\omega$, where $\Phi$ is the unknown functional,
which can be found from equation
\begin{equation}\label{Phi_eq}
\int_{\mathbb{T}^3}\Phi(\omega)\omega(x)\cdot\omega(x)dx=\int_{\mathbb{T}^3}(\omega(x),\nabla)v(x)\cdot \omega(x)dx.
\end{equation}
According to \eqref{Phi_eq},
\begin{equation}\label{Phi_def}
\Phi(\omega) = \left\{\begin{array}{rl}
        \int_{\mathbb T^3}(\omega(x),\nabla)\operatorname{curl}^{-1}\omega(x)\cdot \omega(x)dx/\int_{\mathbb T^3}|\omega(x)|^2dx, & \omega\neq 0,\\
        0, & \omega \equiv 0.
            \end{array}\right.
\end{equation}
where $\operatorname{curl}^{-1}\omega(x)$ is defined in \eqref{rotinv}.

Thus, we arrive at the following system of normal parabolic equations corresponding to Helmholtz equations \eqref{Helmholtz}:
\begin{equation}\label{npe}
\partial_t \omega(t,x) - \Delta \omega - \Phi(\omega)\omega = 0,\, \operatorname{div}\omega = 0,
\end{equation}
\begin{equation}\label{NPE_boundcond}
   {\omega}(t,...,x_i,...)={\omega}(t,...,x_i+2\pi,...), \, i=1,2,3
\end{equation}
where $\Phi$ is the functional defined in \eqref{Phi_def}

Further we study problem \eqref{npe}, \eqref{NPE_boundcond} with
initial condition \eqref{Helmholtz_incond}.


\subsection{Explicit formula for solution of NPE}\label{s2.2}

In this subsection we remind the explicit formula for NPE
solution.

\begin{lem}\label{exp_sol_lem}
Let $\mathbf{S}(t,x;\omega_0)$ be the solution of the following Stokes system with periodic boundary conditions:

\begin{eqnarray}
&\partial_t z - \Delta z=0,\, \operatorname{div}z = 0;\label{3Dheat_eq}\\
& z(t,...,x_i+2\pi,...) = z(t,x),\,\,\,\, i=1,2,3;\label{3DS_boundary_cond}\\
&z(0,x)=\omega_0,\label{3Dheat_in_cond}
\end{eqnarray}
i.e. $\mathbf{S}(t,x;\omega_0)=z(t,x)$. \footnote{Note that
because of periodic boundary conditions the Stokes system should
not contain the pressure term $\nabla p$} Then the solution of
problem \eqref{npe} with periodic boundary conditions and initial
condition \eqref{Helmholtz_incond} has the form
\begin{equation}\label{exp_sol}
   \omega(t,x;\omega_0)=\frac{\mathbf{S}(t,x;\omega_0)}{1-\int_0^t\Phi (\mathbf{S}(\tau,x;\omega_0))d\tau}
\end{equation}
\end{lem}

One can see the proof of this Lemma in \cite{F2}, \cite{F4}. 

\subsection{Unique solvability  of NPE}

Let $Q_T=(0,T)\times \mathbb{T}^3,\;T>0$ or $T=\infty $. The
following space of solutions for NPE will be used:
$$
    V^{1,2(-1)}(Q_T)=L_2(0,T;V^1)\cap H^1(0,T;V^{-1})
$$
We look for solutions $\omega (t,x;\omega_0)$ satisfying

\begin{condition}\label{ex cond} If initial condition $\omega_0\in V^0\setminus
\{ 0\}$ and solution $\omega (t,x;\omega_0)\in V^{1,2(-1)}(Q_T)$
then $ \omega(t,\cdot ,\omega_0)\ne 0\; \forall t\in [0,T]$
\end{condition}

\begin{theorem} For each $\omega_0 \in V^0$ there exists $T>0$ such
that there exists unique solution $\omega (t,x;\omega_0)\in
V^{1,2(-1)}(Q_T)$ of the problem
\eqref{npe},\eqref{NPE_boundcond}, \eqref{Helmholtz_incond}
satisfying Condition \eqref{ex cond}
\end{theorem}

\begin{theorem} The solution $\omega (t,x;\omega_0)\in
V^{1,2(-1)}(Q_T)$ of the problem
\eqref{npe},\eqref{NPE_boundcond}, \eqref{Helmholtz_incond}
depends continuously on initial condition $\omega_0\in V^0$.
\end{theorem}

One can see the proof of these Theorems in \cite{F4}.

\subsection{Structure of dynamical flow for NPE}

We will use $V^0(\mathbb{T}^3)\equiv V^0$ as the phase space for
problem \eqref{npe},\eqref{NPE_boundcond},
\eqref{Helmholtz_incond}.

\begin{deff}\label{M-} The set $M_-\subset V^0$ of $\omega_0$, such that the
corresponding solution $\omega(t,x;\omega_0)$ of problem
\eqref{npe},\eqref{NPE_boundcond}, \eqref{Helmholtz_incond}
satisfies inequality
$$
      \|\omega(t,\cdot ;\om_0)\|_0\le \alpha \|\om_0\|_0e^{-t/2}\qquad \forall
      t>0
$$
is called the set of stability. Here $\alpha >1$ is a fixed number
depending on $\| \omega_0\|_0$.
\end{deff}

\begin{deff}\label{M+} The set $M_+\subset V^0$ of $\om_0$, such that the
corresponding solution $\om(t,x;\om_0)$ exists only on a finite
time interval $t\in (0,t_0)$, and blows up at $t=t_0$ is called
the set of explosions.
\end{deff}

\begin{deff}\label{Mg} The set $M_g\subset V^0$ of $\om_0$, such that the corresponding
solution $\om(t,x;\om_0)$ exists for time $t\in \mathbb{R}_+$, and
$\|\om(t,x;\om_0)\|_{0}\to \infty$ as $t\to \infty$ is called the
set of growing.
\end{deff}

\begin{lem}(see \cite{F4}) Sets $M_-,M_+,M_g$ are not empty, and $M_-\cup M_+\cup
M_g=V^0$
\end{lem}
\subsection{On a geometrical structure of phase space}

Let define the following subsets of unit sphere: $\Sigma =\{ v\in
V^0:\; \| v\|_0=1\}$ in the phase space $V^0$:
$$
  A_-(t)=\{v\in \Sigma : \int_0^t\Phi (S(\tau , v))d\tau \le 0\},
  \quad A_-=\cap_{t\ge 0}A_-(t),
$$
$$
 B_+=\Sigma \setminus A_-\equiv
 \{ v\in \Sigma : \; \exists t_0>0 \; \int_0^{t_0}\Phi (S(\tau ,v))d\tau >0\},
$$
$$
   \partial B_+=\{ v\in \Sigma :\; \forall t>0 \int_0^{t}\Phi (S(\tau ,v))d\tau \le 0 \quad
 \mbox{and}\; \exists t_0>0: \; \int_0^{t_0}\Phi (S(\tau ,v))d\tau =0 \}
$$
We introduce the following function on sphere $\Sigma $:
\begin{equation}\label{func sphere}
   B_+\ni v\to b(v)=\max_{t\ge 0} \int_0^{t}\Phi (S(\tau ,v))d\tau
\end{equation}
Evidently, $b(v)>0$ и $b(v)\to 0$ as $v\to \partial B_+$. Let
define the map $\Gamma (v)$:
\begin{equation}\label{map}
   B_+\ni v\to \Gamma (v)=\frac{1}{b(v)}v\in V^0
\end{equation}
It is clear that $\| \Gamma (v)\|_0\to \infty $ as $v\to \partial
B_+$. The set $\Gamma(B_+)$ divides $V^0$ into two parts:
$$
   V^0_-=\{ v\in V^0:\; [0,v]\cap \Gamma(B_+)=\emptyset \},
$$
$$
   V^0_+=\{ v\in V^0:\; [0,v)\cap \Gamma(B_+)\ne \emptyset \}
$$
Let $B_+=B_{+,f}\cup B_{+,\infty}$ where
$$
   B_{+,f}=\{ v\in B_+:\; \mbox{max in \eqref{func sphere} is achived at}\; t<\infty \}
$$
$$
   B_{+,\infty}=\{ v\in B_+:\; \mbox{max in \eqref{func sphere} is not achived at}\; t<\infty \}
$$
\begin{theorem}(see \cite{F4}) $M_-=V^0_-,\; M_+=V^0_+\cup B_{+,f},\;
M_g=B_{+,\infty}$
\end{theorem}


\section{Stabilization of solution for NPE by starting control}\label{s2}


\subsection{Formulation of the main result on stabilization}

We consider semilinear parabolic equations \eqref{npe}:
\begin{equation}\label{NPE_1}
    \pd_t{y}(t,{x})-\Delta {y}(t,{x})-\Phi ({y}){y}=0
\end{equation}
 with periodic boundary condition
\begin{equation}\label{npe_bound}
   {y}(t,...x_i+2\pi,... )={y}(t,{x}),\, i=1,2,3
\end{equation}
and initial condition
\begin{equation}\label{npe_in_contr}
    {y}(t,{x})|_{t=0}={y}_0({x})+{u}_0({x}).
\end{equation}
Here $\Phi$ is the functional defined in \eqref{Phi_def}, ${y}_0(
{x})\in V^0(\mathbb{T}^3)$ is an arbitrary given initial datum and
${u}_0({x})\in V^0(\mathbb{T}^3)$ is a control. Phase space $V^0$
is defined in \eqref{phase_space}.

We assume that ${u}_0({x})$ is supported on
$[a_1,b_1]\times[a_2,b_2]\times[a_3,b_3]\subset \mathbb{T}^3 =
(\mathbb{R}/2\pi\mathbb{Z})^3$:
\begin{equation}\label{supp_v}
    \mbox{supp}\,{u}_0\subset [a_1,b_1]\times[a_2,b_2]\times[a_3,b_3]
\end{equation}

Our goal is to find for every given ${y}_0(x)\in
V^0(\mathbb{T}^3)$ a control ${u}_0\in V^0(\mathbb{T}^3)$
satisfying \eqref{supp_v} such that there exists unique solution
${y}(t,{x};{y}_0+{u}_0)$ of \eqref{NPE_1}-\eqref{npe_in_contr} and
this solution satisfies the estimate
\begin{equation}\label{stab_est}
    \| {y}(t,\cdot ;{y}_0+{u}_0)\|_0 \le \alpha \|{y}_0+{u}_0\|_0e^{-t}\quad \forall t>0
\end{equation}
with a certain $\alpha >1$.

By Definition \ref{M-} of the set of stability $M_-$ inclusion
$y_0\in M_-$ implies estimate \eqref{stab_est} with $u_0=0$.
Therefore the formulated problem is reach of content only if
$y_0\in V^0\setminus M_-=M_+\cup M_g$.

The following main theorem holds:

\begin{theorem}\label{stabiliz_theorem}
Let ${y}_0\in V^0\setminus M_-$ be given. Then there exists a
control ${u}_0\in V^0$ satisfying \eqref{supp_v} such that there
exists a unique solution ${y}(t,{x};{y}_0+{u}_0)$ of
\eqref{NPE_1}-\eqref{npe_in_contr}, and this solution satisfies
bound \eqref{stab_est} with a certain $\alpha >1$.
\end{theorem}

The rest part of the paper is devoted to the proof of this theorem.


\subsection{Formulation of the main preliminary result}

To rewrite condition \eqref{supp_v} in more convenient form, let us
first perform the change of variables in
\eqref{NPE_1}-\eqref{npe_in_contr}:
$$\tilde{x}_i=x_i - \frac{a_i+b_i}{2}, i=1,2,3$$
and denote

\begin{equation}\label{newvar}
\begin{split}
&\tilde {{y}}(t,\tilde {{x}})= {{y}}\left(t,\tilde {x}_1+\frac{a_1+b_1}{2},\tilde {x}_2+\frac{a_2+b_2}{2},\tilde {x_3}+\frac{a_3+b_3}{2}\right),\\
&\tilde {{y}}_0(\tilde {{x}})= {{y}}_0\left(\tilde {x}_1+\frac{a_1+b_1}{2},\tilde {x}_2+\frac{a_2+b_2}{2},\tilde {x}_3+\frac{a_3+b_3}{2}\right),\\
&\tilde {{u}_0}(\tilde {{x}})= {{u}_0}\left(\tilde
{x}_1+\frac{a_1+b_1}{2},\tilde {x}_2+\frac{a_2+b_2}{2},\tilde
{x}_3+\frac{a_3+b_3}{2}\right).
\end{split}
\end{equation}

Then substituting \eqref{newvar} into relations \eqref{NPE_1}-\eqref{npe_in_contr}, \eqref{stab_est} and omitting the tilde sign leaves these relations unchanged, while inclusion \eqref{supp_v} transforms into
\begin{equation}\label{supp_v_p}
\mbox{supp}\,{u}_0\subset
[-\rho_1,\rho_1]\times[-\rho_2,\rho_2]\times[-\rho_3,\rho_3]
\end{equation}
where $\rho_i = \dfrac{b_i-a_i}{2}\in(0,\pi)$, $i=1,2,3$.

Below we consider stabilization problem \eqref{NPE_1}-\eqref{npe_in_contr}, \eqref{stab_est}
with condition \eqref{supp_v_p} instead of \eqref{supp_v}.

We look for a starting control $u_0(x)$ in a form
\begin{equation}\label{StartContr}
  u_0(x)=u_1(x)-\lambda u(x)
\end{equation}
where the component $u_1(x)$ and the constant $\lambda >0$ will be
defined later and the main component $u(x)$ is defined as follows.
For given $\rho_1,\rho_2,\rho_3\in(0,\pi)$ we choose $p\in \mathbb
N$ such that
\begin{equation}\label{p_def}
\frac{\pi}{p}\le \rho_i, \,i=1,2,3,
\end{equation}
and denote by $\chi_{\frac{\pi}{p}}(\alpha )$ the characteristic
function of interval $(-\frac{\pi}{p},\frac{\pi}{p})$:
\begin{equation}\label{def_chi}
\chi_{\frac{\pi}{p}}(\alpha) = \left\{\begin{array}{rl}
        1, & |\alpha | \leq \frac{\pi}{p},\\
        0, & \frac{\pi}{p}<|\alpha | \leq \pi.
            \end{array}\right.
\end{equation}
Then we set
\begin{equation}\label{u_def}
{u}({x}) = \operatorname{curl}\operatorname{curl}(\chi_{\frac{\pi}{p}}(x_1)\chi_{\frac{\pi}{p}}(x_2)\chi_{\frac{\pi}{p}}(x_3)w(px_1,px_2,px_3),0,0),
\end{equation}
where
\begin{equation}\label{def_w}
w(x_1,x_2,x_3) =\sum_{\genfrac{}{}{0pt}{}{i,j,k=1}{i<j,k\neq i,j}}^{3}a_k(1+\cos x_k)(\sin x_i+\frac{1}{2}\sin 2x_i)(\sin x_j+\frac{1}{2}\sin 2x_j),
\end{equation}
$a_1,a_2,a_3\in \mathbb{R}$.

\begin{prop}\label{prop_u} The vector field $u(x)$ defined in
\eqref{p_def}-\eqref{def_w} possesses the following properties:
\begin{equation}\label{main_est_a}
   {u}({x})\in V^0(\mathbb{T}^3),\qquad \mbox{supp}\,u\, \subset ([-\rho
   ,\rho])^3
\end{equation}
\end{prop}

\begin{proof}
For each $j=1,2,3$ function $w(x_1,x_2,x_3)$ defined in
\eqref{def_w}
 and $\partial_jw$ equal to zero at $x_j=\pm \pi$. That is why
using notations $\mathbf{\chi}_{\frac{\pi}{p}}(x)=
\chi_{\frac{\pi}{p}}(x_1)\chi_{\frac{\pi}{p}}(x_2)\chi_{\frac{\pi}{p}}(x_3),\,
w(px)=w(px_1,px_2,px_3)$ we get
\begin{equation}\label{def_cw}
\mbox{curl}(\mathbf{\chi}_{\frac{\pi}{p}}(x)(w(px),0,0))=p\mathbf
{\chi}_{\frac{\pi}{p}}(x)(0,\partial_3w(px),-\partial_2w(px))\in
(H^1(\mathbb{T}^3))^3
\end{equation}
\begin{equation}\label{def_ccw}
  u(x)=p^2{\mathbf
{\chi}}_{\frac{\pi}{p}}(x)(-\partial_{22}w(px)-\partial_{33}w(px),\partial_{12}w(px),-\partial_{13}w(px))\in
(H^0(\mathbb{T}^3))^3
\end{equation}
Applying to vector field \eqref{def_ccw} operator \mbox{div} and
performing  direct calculations in the space of distributions we get
that $\mbox{div}\, u(x)=0$. Hence, $u(x)\in V^0(\mathbb{T}^3)$.
The second inclusion in \eqref{main_est_a} is evident.
\end{proof}

Let consider the boundary value problem for the system of three
heat equations
\begin{equation}\label{heat_eq_u}
    \pd_t{\mathbf{S}}(t,{x};u)-\Delta {\mathbf{S}}(t,{x};u)=0,\qquad  \mathbf{S}(t,{x})|_{t=0}={u}({x})
\end{equation}
with periodic boundary condition. (Since by Proposition
\ref{prop_u} $\mbox{div}u(x)=0$ we get that $\mbox{div}\,
S(t,x;u)=0$ for $t>0$, and therefore system \eqref{heat_eq_u} in
fact is equal to the Stokes system.)

The following theorem is true:
\begin{theorem}\label{main_est_th}
For each $\rho :=\pi /p \in (0,\pi )$ the function ${u}({x})$
defined in \eqref{u_def} by a natural number $p$ satisfying
\eqref{p_def} and characteristic function \eqref{def_chi},
satisfies the estimate:
\begin{equation}\label{main_est_b}
\int_{T^3}((\mathbf{S}(t,{x};{u}),\nabla)\operatorname{curl}^{-1}\mathbf{S}(t,{x};{u}),\mathbf{S}(t,{x};{u}))dx>\beta
e^{-18t} \qquad \forall \; t\ge 0
\end{equation}
with a positive constant $\beta$.
\end{theorem}

The proof of Theorem \ref{main_est_th} in fact is the main content
of this paper, and it will be given further.


\subsection{Intermediate Control}\label{s2.3}

To avoid certain difficulties with the proof of Theorem
\ref{stabiliz_theorem}, we have to include additional control that
eliminates some Fourier coefficients in given initial condition
$y_0$ of our stabilization problem. We will use the techniques
developed in local stabilization theory (see \cite{F6}, \cite{FG}
and references therein).

Let us consider the following decomposition of the phase  space:
$V^0=V_+\oplus V_-$, where
\begin{equation}\label{V+def}
V_+=\{v\in V^0: v(x)=\sum_{0<|k|^2<18}v_ke^{i({k},{x})}, v_k\in
\mathbb{C}^3,r\cdot v_k=0, v_{-k}=\overline{v_k}\},
\end{equation}
where, recall $\overline v_k$ means complex conjugation of $v_k$,
$k=(k_1,k_2,k_3), |k|^2=k_1^2+k_2^2+k_3^2$, and
\begin{equation}\label{V-def}
V_-=V^0\ominus V_+.
\end{equation}

\begin{theorem}\label{intermediate_control}
There exists a linear feedback operator $F$,
\begin{equation}\label{F_operator}
F: V^0(\mathbb{T}^3)\mapsto V^1_{00}(\Omega) := \{ {y}({x})\in V^1(\mathbb{T}^3):
\mbox{supp}\;{y}\subset \Omega\},
\end{equation}
where $\Omega =\{x\in ([-\rho ,\rho])^3: |x|^2\le \rho^2\}\subset
\mathbb{T}^3$, $\rho =\pi/p, p\in \mathbb{N}\setminus \{1\}$, such
that for every $y\in V^0$
\begin{equation}\label{F_operator1}
     y+Fy\in V_-\quad,
\end{equation}
where $V_-$ is the subset of $V^0$ defined in
\eqref{V+def}-\eqref{V-def}
\end{theorem}

\begin{proof}
For $f(x)\in L_2(\Omega )$ let us consider the Poisson problem
\begin{equation}\label{op_eq}
    -\Delta v(x)=f(x),\; x\in \mbox{int}\{\Omega\} ,\quad v|_{\partial \Omega}=0,
\end{equation}
where $\mbox{int}\{ \Omega \}$ is the interior of the set
$\Omega$. Define operator $R_\Omega$ by the formula
\begin{equation}\label{R_operator}
      (R_\Omega f)(x)=v(x),\; \mbox{for}\; x\in \Omega ,\quad (R_\Omega f)(x)=0,\; \mbox{for} \; x\in \mathbb{T}^3\setminus
      \Omega
\end{equation}

Using notations
$$
    {g}=(g(k)\in \mathbb{C},\; 0<|k|^2<18,\; g(-k)=\overline{g(k)}),
$$
$$
    {c(g)}=(c_j(g)\in \mathbb{C},\; 0<|j|^2<18,\; c_{-j}(g)=\overline{c_j(g)}),
$$
where overline means complex conjugation, and
$$
     m_{k,j}=(2\pi)^{-3}\int_\Omega (R_\Omega e^{i(j,\cdot )})(x)e^{-i(k,x)}dx,\quad M=\|
     m_{k,j}\|_{0<|k|^2<18,),0<|j|^2<18},
$$
 we consider the following system of linear algebraic
equations with respect to $c_j(g), 0<|j|^2<18$:
\begin{equation}\label{F_constr2}
     Mc(g)=g
\end{equation}
Let prove that this system is uniquely solvable. We show first
that matrix $M$ is complex symmetric:
\begin{equation*}
\begin{split}
 &m_{k,j}=\int_\Omega (R_\Omega e^{i(j,\cdot )})(x)\overline{(-\Delta)(R_\Omega e^{i(k,\cdot))}(x)}dx=\\
 & \int_\Omega (-\Delta)(R_\Omega e^{i(j,\cdot )})(x)(R_\Omega e^{i(-k,\cdot))(x)}dx=\overline{m}_{j,k}
\end{split}
\end{equation*}
 Next we show that  $M$ is positive definite.
Indeed, for each vector ${c=(c_j, 0<|j|^2<18)}$
$$
    <Mc,c>=\sum_{j,k}m_{j,k}c_j\overline{c}_{k}=
    \sum_{j,k}c_j\overline{c}_{k}\int_\Omega (\nabla(R_\Omega e^{i(j,\cdot )})(x))\cdot(\nabla \overline{(R_\Omega
   e^{i(k,\cdot))(x)}})dx=
$$
$$
\int_\Omega|\sum_{j}((R_\Omega je^{i(j,\cdot
)})(x)c_j)|^2dx\ge 0
$$

Equality here can be attained only if $\sum_{j}((R_\Omega
je^{i(j,\cdot)})(x)c_j)=0\; \forall x\in \Omega$, or, as it
follows from definition \eqref{op_eq}, \eqref{R_operator} of
operator $R_\Omega$, only if $$\sum_{j} je^{i(j,\cdot)}(x)c_j=0\;\,
\forall {x}\in \Omega$$ The last equality implies that ${c}_j=0$
for all ${j}: 0<|{j}|^2<18$.

Thus, we have proved that $\det{M}\ne 0$. This means that we can
find ${c}(g)$ from equation \eqref{F_constr2}.

Let $y(x)\in V^0$ and $\hat{y}(k)$ be Fourier coefficients of $y$.
We look for operator \eqref{F_operator}, \eqref{F_operator1} in
the form
\begin{equation}\label{F_op_def}
      Fy=\sum_{0<|j|^2<18}c_j(y)R_\Omega e^{i(j,\cdot )}
\end{equation}
where operator $R_\Omega$ is defined in \eqref{op_eq},
\eqref{R_operator}, and $c_j(y)=(c^1_j(y),c^2_j(y),c^3_j(y))\in
\mathbb{C}^3$, satisfying $c_{-j}(y)=\overline{c_j(y)}$,  $c_j\cdot
j=0$, are defined from the system
\begin{equation}\label{F_constr1}
-\hat {{y}}({k})=(2\pi)^{-3}\sum_{0<|j|^2<18}c_j({y})\int_\Omega
(R_\Omega e^{i({j},\cdot )})({x})e^{-i(k,x)}dx,\, \forall k:
0<|k|^2<18
\end{equation}
Indeed, for each $m=1,2,3$ system \eqref{F_constr1} can be
rewritten in the form \eqref{F_constr2} with $g=(-g^m(k),
0<|k|<18)$. Since there exists unique solution of system
\eqref{F_constr2}, solution $\{c_j(y)=(c^m_j(y), m=1,2,3),
0<|j|<18\}$ of \eqref{F_constr1} is well defined. Moreover, it is
easy to see that that $c_{-j}(y)=\overline{c_j(y)}$ and $c_j\cdot
j=0$. Hence, we have proved that operator \eqref{F_op_def},
\eqref{F_constr1} satisfies \eqref{F_operator}and
\eqref{F_operator1}.
\end{proof}

We define component $u_1(x)$ from \eqref{StartContr} as follows:
\begin{equation}\label{StartContr1}
    u_1(x)=(Fy_0)(x)
\end{equation}
where $F$ is feedback operator \eqref{F_operator},
\eqref{F_operator1}.


\subsection{ Proof of the stabilization result}\label{s2.4}
In this subsection we prove Theorem \ref{stabiliz_theorem} using
Theorems \ref{main_est_th}, \ref{intermediate_control}. We take
control \eqref{StartContr} as a desired one where vector-functions
$u_1(x)$, $u(x)$ are defined in \eqref{StartContr1} and in
\eqref{u_def}, \eqref{def_w} correspondingly, and $\lambda \gg 1$
is a parameter.

In virtue of explicit formula \eqref{exp_sol} for solution of NPE,
in order to prove the desired result it is enough to choose
parameter $\lambda$ in such way that the function
\begin{equation}\label{denominator}
    1-\int_0^t\Phi (\mathbf{S}(\tau ,\cdot ;y_0+u_1-\lambda u))d\tau
\end{equation}
for each $t>0$ is bounded from below by a positive constant
independent of $t$. For this aim we estimate the function $-\Phi
(\mathbf{S}(t,\cdot ;z_0-\lambda u))$ where
$z_0=y_0+u_1=y_0+Fy_0$. In virtue of Theorem
\ref{intermediate_control} $z_0\in V_-$, where $V_-$ is the subset
of the phase space defined in \eqref{V-def}, i.e. Fourier
coefficients of $z_0$ satisfy the condition
\begin{equation}\label{y_k_zero}
\hat{z}_0(k)=0\, \text{ for } \, |k|^2=k_1^2+k_2^2+k_3^2<18.
\end{equation}
Let us denote the nominator of functional $\Phi$ defined in
\eqref{Phi_def} as follows:
\begin{equation}\label{Psi}
  \Psi(y_1,y_2,y_3)=\int_{T^3}((y_1,\nabla)\mbox{curl}^{-1}y_2,y_3)dx,
  \quad
  \Psi(y)=\Psi(y,y,y)
\end{equation}
and prove the estimate
\begin{equation}\label{psi_est}
-\Psi(\mathbf{S}(t,\cdot;z_0-\lambda u)) > c_1\lambda^3
e^{-18t}\quad \mbox{for } \lambda\gg 1,
\end{equation}
where $c_1$ is some positive constant.

According to Theorem \ref{main_est_th},
\begin{equation}\label{psi_u_est}
\Psi(\mathbf{S}(t,\cdot;u)) \geq \beta e^{-18t}, \, \beta >0.
\end{equation}

From definition \eqref{Psi} of $\Psi$,
\begin{equation}\label{psi_est_proof1}
\begin{split}
&-\Psi (\mathbf{S}(t,z_0-\lambda u))=\lambda^3\Psi(\mathbf{S}(t,u))- \lambda^2(\Psi(\mathbf{S}(t,u),\mathbf{S}(t,u),\mathbf{S}(t,z_0))+\\
&\Psi(\mathbf{S}(t,u),\mathbf{S}(t,z_0),\mathbf{S}(t,u))+\Psi(\mathbf{S}(t,z_0),\mathbf{S}(t,u),\mathbf{S}(t,u)))+\\
&\lambda(\Psi(\mathbf{S}(t,u),\mathbf{S}(t,z_0),\mathbf{S}(t,z_0))+\Psi(\mathbf{S}(t,z_0),\mathbf{S}(t,u),\mathbf{S}(t,z_0))+\\
&\Psi(\mathbf{S}(t,z_0),\mathbf{S}(t,z_0),\mathbf{S}(t,u)))-\Psi(\mathbf{S}(t,z_0))
\end{split}
\end{equation}

In virtue of Sobolev embedding theorem and definition \eqref{Psi} we get
\begin{equation}
|\Psi (y_1,y_2,y_3)|\leq\| y_1\|_{L_3(\mathbb{T}^3)}\|\nabla
     \mbox{curl}^{-1}y_2\|_{L_3}\|y_3\|_{L_3}\leq c_2\|y_1\|_{V^{1/2}}\|y_2\|_{V^{1/2}}\|y_3\|_{V^{1/2}}
\end{equation}
Using this inequality and \eqref{psi_u_est}, \eqref{psi_est_proof1} we get
\begin{equation}\label{psi_est_proof2}
\begin{split}
&-\Psi(\mathbf{S}(t,z_0-\lambda u))>\beta \lambda^3e^{-18t}-\\
&c_2\left(\lambda^2\|\mathbf{S}(t,u)\|^2_{V^{1/2}}\|
\mathbf{S}(t,z_0)\|_{V^{1/2}}+\lambda
 \|\mathbf{S}(t,u)\|_{V^{1/2}}\| \mathbf{S}(t,z_0)\|^2_{V^{1/2}}+\|\mathbf{S}(t,z_0)\|^3_{V^{1/2}}\right)
\end{split}
\end{equation}

According to the definition \eqref{u_def}, \eqref{def_w} of
function $u$, all the Fourier coefficients of $u$ corresponding to
$k$ with less than two non-zero components are equal to zero, so
\begin{equation}\label{S_u_est}
\|\mathbf{S}(t,u)\|_{V^{1/2}}\le c_3e^{-2t}
\end{equation}

In virtue of \eqref{y_k_zero},
\begin{equation}\label{S_y_est}
\|\mathbf{S}(t,z_0)\|_{V^{1/2}}\le c_4e^{-18t}.
\end{equation}
Relations \eqref{psi_est_proof2}, \eqref{S_u_est}, \eqref{S_y_est} imply
\begin{equation}
-\Psi(\mathbf{S}(t,z_0-\lambda u))>
c\lambda^3e^{-18t}-c_2(\lambda^2e^{-22t}+\lambda
e^{-38t}+e^{-54t}),
\end{equation}
which completes the proof of estimate \eqref{psi_est}

The denominator of $\Phi (\mathbf{S}(\tau ,\cdot ;z_0-\lambda u))$
is positive, i.e.
\begin{equation}\label{denom_est}
\int_{\mathbb{T}^3}|\mathbf{S}(t,z_0-\lambda u)|^2dx>0
\end{equation}
Bounds \eqref{psi_est},\eqref{denom_est} imply that the function
\eqref{denominator} is more than 1, which completes the proof of
Theorem \ref{stabiliz_theorem}.


\section{Bound \eqref{main_est_b} reduction to estimation of some functionals on scalar functions of one variable}\label{s3}
The rest part of the paper is devoted to the proof of Theorem
\ref{main_est_th}. In this section we reduce the bound
\eqref{main_est_b} for the functional with respect to vector
fields on three space variables to estimate of three functionals
with respect to scalar functions on one space dimension.

\subsection{Preliminaries}\label{s1Dcase}
Let us recall that the function $u$ is defined as follows (see
\eqref{u_def}):
\begin{equation}\label{u_def_recall}
{u}({x}) = \mbox{curl}\,\mbox{curl}\left(\chi_{\frac{\pi}{p}}(x_1)\chi_{\frac{\pi}{p}}(x_2)\chi_{\frac{\pi}{p}}(x_3)w(px_1,px_2,px_3),0,0\right),
\end{equation}
where
\begin{equation}\label{def_w_recall}
\begin{split}
&w(x_1,x_2,x_3) =\sum_{\genfrac{}{}{0pt}{}{i,j,k=1}{i<j,k\neq i,j}}^{3}a_k(1+\cos x_k)(\sin x_i+\frac{1}{2}\sin 2x_i)(\sin x_j+\frac{1}{2}\sin 2x_j), \\
&\hphantom{w(x_1,x_2,x_3) =\sum_{\genfrac{}{}{0pt}{}{i,j,k=1}{i<j,k\neq i,j}}^{3}a_k(1+\cos x_k)}a_1, a_2, a_3\in \mathbb{R}.
\end{split}
\end{equation}
and $\chi_{\frac{\pi}{p}}(x)$ is the characteristic function of interval $[-\pi/p;\pi/p]$, $p\in \mathbb{N}$, defined in \eqref{def_chi}.

Our goal is to prove estimate
\begin{equation}\label{main_est}
\int_{T^3}((\mathbf{S}(t,{x};{u}),\nabla)\operatorname{curl}^{-1}\mathbf{S}(t,{x};{u}),\mathbf{S}(t,{x};{u}))dx>\beta e^{-18t} \qquad \forall \; t\ge 0
\end{equation}
with a positive constant $\beta$, where $\mathbf{S}(t,x;u)$ is the
solution of the system of three heat equations
\begin{equation}\label{heat_eq_u_1}
    \pd_t\mathbf{S}(t,{x};{u})-\Delta \mathbf{S}(t,{x};{u})=0,\qquad  \mathbf{S}(t,{x};{u})|_{t=0}={u}({x})
\end{equation}
with periodic boundary conditions.

As is well-known, the solution of system \eqref{heat_eq_u_1} is given by
\begin{equation}\label{S_fourier}
    \mathbf{S}(t,{x};u) = \sum_{k\in\mathbb{Z}^3\setminus\{0\}}\hat{u}(k)e^{i(k,x)}e^{-|k|^2t},
\end{equation}
where $\hat{u}(k)$ are the Fourier coefficients of function $u$:
\begin{equation}\label{u_fourier}
    u(x) = \sum_{k\in\mathbb{Z}^3\setminus\{0\}}\hat{u}(k)e^{i(k,x)}.
\end{equation}

It obviously follows from \eqref{S_fourier}, \eqref{u_fourier}, that
\begin{equation}\label{S_vect_der}
    \partial_{x_i}\mathbf{S}(t,{x};u) = \mathbf{S}(t,{x};\partial_{x_i}u).
\end{equation}

Let us denote by $S(t,x;\phi(\xi))$ the solution of scalar heat
equation with initial condition $S(t,x;\phi
(\xi))|_{t=0}=\phi(x)$, where $\phi$ is a periodic scalar function
of ${x=(x_1,x_2,x_3)\in \mathbb{R}^3}$.
Further we will need the following lemmas:

\begin{lem}\label{S_split}
Let $S(t,x;f_1(\xi_1)f_2(\xi_2)f_3(\xi_3))$ be the solution of the heat equation
\begin{equation}\label{heat_split_lem}
\partial_t S-\Delta S=0
\end{equation}
with periodic boundary condition
\begin{equation}\label{heat_split_lem_bound}
S(t,...,x_i+2\pi,...)=S(t,x;...), i=1,2,3
\end{equation}
and periodic initial condition
\begin{equation}\label{heat_split_lem_init}
S(t,x;f_1f_2f_3)|_{t=0}=f_1(x_1)f_2(x_2)f_3(x_3)
\end{equation}
Then $$S(t,x;f_1(\xi_1)f_2(\xi_2)f_3(\xi_3)) =
S(t,x_1;f_1)S(t,x_2;f_2)S(t,x_3;f_3)$$, where $S(t,x_i;f_i)$ is the
solution of problem
\begin{equation*}
\partial_t S-\partial_{x_ix_i}S=0,\,\, S(t,x_i+2\pi ;f_i)=S(t,x_i;f_i),\,\, S(t,x_i;f_i)|_{t=0}=f_i(x_i)
\end{equation*}
\end{lem}
\begin{proof}
The statement of lemma is true because solution of problem
\eqref{heat_split_lem} - \eqref{heat_split_lem_init} is unique,
and straightforward calculations show that the function
$${S(t,x_1;f_1)S(t,x_2;f_2)S(t,x_3;f_3)}$$ satisfies
\eqref{heat_split_lem} - \eqref{heat_split_lem_init}.
\end{proof}

\begin{lem}\label{S_der}
Let $S(t,x;\chi_{\frac{\pi}{p}}\phi(p\xi))$ be the solution of the
one dimensional heat equation
\begin{equation}\label{heat_1D}
\partial_t S-\partial_{xx}S=0
\end{equation}
with periodic boundary condition
\begin{equation}\label{heat_1D_bound}
S(t,x+2\pi)=S(t,x)
\end{equation}
and initial condition
\begin{equation}\label{heat_1D_init}
S(t,x)|_{t=0}=\chi_{\frac{\pi}{p}}\phi(px)
\end{equation}
where $\phi (x+2\pi)=\phi (x)$, and $\chi_{\frac{\pi}{p}}$ is the
characteristic function of interval ${[-\frac{\pi}{p},\frac{\pi}{p}]}$, defined in
\eqref{def_chi}, $p\in \mathbb{N}$. Then

i) For all  $t\geq 0$,
\begin{equation}\label{Int_S_const}
\int_{-\pi}^{\pi} S(t,x;\chi_{\frac{\pi}{p}}\phi(p\xi))dx = \int_{-\pi}^{\pi} S(0,x;\chi_{\frac{\pi}{p}}\phi(p\xi))dx = const;
\end{equation}

ii) If function $\phi(x)\in C^1[-\pi,\pi]$, and $\phi(\pm\pi)=0$, then
\begin{equation}\label{S_x}
\partial_x S(t,x;\chi_{\frac{\pi}{p}}\phi(p\xi)) = pS(t,x;\chi_{\frac{\pi}{p}}\phi'(p\xi)),
\end{equation}

iii) If function $\phi(x)\in C^{2}[-\pi,\pi]$, $\phi(\pm\pi) = 0$ and $\phi'(\pm\pi) = 0$ then
\begin{equation}\label{S_t}
\partial_t S(t,x;\chi_{\frac{\pi}{p}}\phi(p\xi)) = p^2S(t,x;\chi_{\frac{\pi}{p}}\phi''(p\xi)).
\end{equation}
\end{lem}
\begin{proof}
As is well known, the solution of the \eqref{heat_1D} with initial condition \eqref{heat_1D_init} and boundary condition \eqref{heat_1D_bound} is given by formula
\begin{equation}\label{S_via_G}
S(t,x; \chi_{\frac{\pi}{p}}\phi(p\xi)) = \int_{-{\pi}}^{{\pi}}G(t,x-\xi)\chi_{\frac{\pi}{p}}\phi(p\xi)d\xi,
\end{equation}
where $G(t,y)$ is the Green function, defined as follows:
\begin{equation}\label{def_G}
G(t,y)=\sum_{k\in \mathbb Z}\frac{1}{2\sqrt{\pi t}}e^{-\frac{(y+2\pi k)^2}{4t}}.
\end{equation}

i) Since $S(t,x; \phi(p\xi))$ is the solution of the heat equation \eqref{heat_1D} with boundary condition \eqref{heat_1D_bound}, the following equality is true:
\begin{equation}\label{Int_S_t}
\begin{split}
&\partial_t\int_{-\pi}^{\pi}S(t,x;\chi_{\frac{\pi}{p}} \phi(p\xi))dx = \int_{-\pi}^{\pi}\partial_t S(t,x;\chi_{\frac{\pi}{p}} \phi(p\xi))dx = \\
&\int_{-\pi}^{\pi}\partial_{xx} S(t,x;\chi_{\frac{\pi}{p}} \phi(p\xi))dx=\partial_x S(t,x;\chi_{\frac{\pi}{p}} \phi(p\xi))\bigg\vert_{x=-\pi}^{\pi} = 0.
\end{split}
\end{equation}
Therefore, $\int_{-\pi}^{\pi}S(t,x;\chi_{\frac{\pi}{p}} \phi(p\xi))dx = const$

ii) Differentiating \eqref{S_via_G} with respect to $x$ and then integrating the right side by
parts, taking into account that $\phi(\pm\pi) = 0$, we
get
\begin{equation*}
\begin{split}
&\partial_x S(t,x;\chi_{\frac{\pi}{p}}\phi(p\xi))=\int_{-\frac{\pi}{p}}^{\frac{\pi}{p}}\partial_{x}G(t,x-\xi)\phi(p\xi)d\xi=\\
&-\int_{-\frac{\pi}{p}}^{\frac{\pi}{p}}\partial_{\xi}G(t,x-\xi)\phi(p\xi)d\xi= p\int_{-\frac{\pi}{p}}^{\frac{\pi}{p}}G(t,x-\xi)\phi'(p\xi)d\xi=pS(t,x;\phi'(p\xi)).
\end{split}
\end{equation*}

iii) Equations \eqref{heat_1D} and \eqref{S_via_G} after
integrating by parts twice with consideration of $\phi(\pm\pi) =
0$ and $\phi'(\pm\pi)=0$ imply the following relations:
\begin{equation*}
\begin{split}
&\partial_t S(t,x;\chi_{\frac{\pi}{p}}\phi(p\xi))=\partial_{xx}S(t,x;\chi_{\frac{\pi}{p}}\phi (p\xi))=\int_{-\frac{\pi}{p}}^{\frac{\pi}{p}}\partial_{xx}G(t,x-\xi)\phi(p\xi)d\xi=\\
&\int_{-\frac{\pi}{p}}^{\frac{\pi}{p}}\partial_{\xi\xi}G(t,x-\xi)\phi(p\xi)d\xi= p^2\int_{-\frac{\pi}{p}}^{\frac{\pi}{p}}G(t,x-\xi)\phi''(p\xi)d\xi=p^2S(t,x;\phi''(p\xi)).
\end{split}
\end{equation*}
\end{proof}

\subsection{Some transformation in left side of \eqref{main_est}}
Now, let us return to the expression in the left hand side of
\eqref{main_est}. First of all we express there solution
$\mathbf{S}(t,{x};{u})$ of system of equations via scalar heat
equation solutions of the form
${S(t,x;\boldsymbol{\chi}_{\frac{\pi}{p}}\partial_{ij}w(p\xi))}$
where $w$ is the function \eqref{def_w_recall}:
\begin{lem}\label{transform1}
The following relation is true:
\begin{equation}\label{S_w}
\begin{split}
&\int_{T^3}((\mathbf{S}(t,{x};{u}),\nabla)\operatorname{curl}^{-1}\mathbf{S}(t,{x};{u}),\mathbf{S}(t,{x};{u}))dx=\\
&p^6\int_{T^3}\left(S^2(t,x;\boldsymbol{\chi}_{\frac{\pi}{p}}\partial_{12}w(p\xi))-S^2(t,x;\boldsymbol{\chi}_{\frac{\pi}{p}}\partial_{13}w(p\xi)\right)S(t,x;\boldsymbol{\chi}_{\frac{\pi}{p}}\partial_{23}w(p\xi))+\\
&S(t,x;\boldsymbol{\chi}_{\frac{\pi}{p}}\partial_{12}w(p\xi))S(t,x;\boldsymbol{\chi}_{\frac{\pi}{p}}\partial_{13}w(p\xi))\cdot\\
&\hphantom{S(t,x;\boldsymbol{\chi}_{\frac{\pi}{p}}\partial_{12}w(p\xi))}\left(S(t,x;\boldsymbol{\chi}_{\frac{\pi}{p}}\partial_{33}w(p\xi))-S(t,x;\boldsymbol{\chi}_{\frac{\pi}{p}}\partial_{22}w(p\xi))\right)dx,
\end{split}
\end{equation}
where $x=(x_1,x_2,x_3)$,
$\boldsymbol{\chi}_{\frac{\pi}{p}}=\chi_{\frac{\pi}{p}}(\xi_1){\chi}_{\frac{\pi}{p}}(\xi_2){\chi}_{\frac{\pi}{p}}(\xi_3)$,
$w$ is the function  \eqref{def_w_recall}.
\end{lem}
\begin{proof}
Comparing \eqref{def_cw} with \eqref{u_def} we see that
\eqref{def_cw} can be rewritten as follows:
\begin{equation}\label{def_cw1}
(\mbox{curl}^{-1}u)(x)=p\mathbf{\chi}_{\frac{\pi}{p}}(x)(0,\partial_3w(px),-\partial_2w(px)):=(v_1(x),v_2(x),v_3(x))
\end{equation}
Rewrite also \eqref{def_ccw} by such a way:
\begin{equation}\label{def_ccw1}
  u(x)=p^2{\mathbf
{\chi}}_{\frac{\pi}{p}}(-\partial_{22}w(px)-\partial_{33}w(px),\partial_{12}w(px),-\partial_{13}w(px)):=(u_1(x),u_2(x),u_3(x))
\end{equation}
Note that
$$
\operatorname{curl}^{-1}\mathbf{S}(t,{x};{u})=
\operatorname{curl}^{-1}\mathbf{S}(t,{x};\operatorname{curl}\operatorname{curl}^{-1}{u})
=\mathbf{S}(t,{x};\operatorname{curl}^{-1}{u})
$$
Taking this into account, we substitute  expression for $u$ and
$\operatorname{curl}^{-1}{u}$ from \eqref{def_ccw1} and
\eqref{def_cw1} into the left part of
\eqref{main_est}. Performing some transformations we get
$$
\int_{T^3}((\mathbf{S}(\cdot;{u}),\nabla)\operatorname{curl}^{-1}\mathbf{S}(\cdot;{u}),\mathbf{S}(\cdot;{u}))dx=
\int_{T^3}((\mathbf{S}(\cdot;{u}),\nabla)\mathbf{S}(\cdot;\operatorname{curl}^{-1}{u}),\mathbf{S}(\cdot;{u}))dx=
$$
$$
\int_{T^3}\sum_{j=1}^3S(\cdot;u_j)\sum_{k=2}^3S(\cdot;\partial_jv_k),S(\cdot;u_k))dx=
p^6\int_{T^3}[S(\cdot;u_1)(S(\cdot;\boldsymbol{\chi}_{\frac{\pi}{p}}\partial_{13}w)
S(\cdot;\boldsymbol{\chi}_{\frac{\pi}{p}}\partial_{12}w)-
$$
$$
S(\cdot;\boldsymbol{\chi}_{\frac{\pi}{p}}\partial_{12}w)S(\cdot;\boldsymbol{\chi}_{\frac{\pi}{p}}\partial_{13}w))+
S^2(\cdot;u_2)S(\cdot;\partial_{2}v_2)+S^2(\cdot;u_3)S(\cdot;\partial_{3}v_3)+
$$
$$
S(\cdot;u_2)S(\cdot;u_3)(S(\cdot;\partial_2v_3)+S(\cdot;\partial_3v_2))]dx=$$
$$
p^6\int_{T^3}S(\cdot;\boldsymbol{\chi}_{\frac{\pi}{p}}\partial_{23}w)\left(S^2(\cdot;\boldsymbol{\chi}_{\frac{\pi}{p}}\partial_{12}w)-
S^2(\cdot;\boldsymbol{\chi}_{\frac{\pi}{p}}\partial_{13}w)\right)+
$$
$$
S(\cdot;\boldsymbol{\chi}_{\frac{\pi}{p}}\partial_{12}w)S(\cdot;\boldsymbol{\chi}_{\frac{\pi}{p}}\partial_{13}w)
\left(S(\cdot;\boldsymbol{\chi}_{\frac{\pi}{p}}\partial_{33}w)-S(\cdot;\boldsymbol{\chi}_{\frac{\pi}{p}}
\partial_{22}w)\right)dx.
$$
\end{proof}

Our next step  is to replace function $w$ in \eqref{S_w} by the sum
of terms it consists of. Let prepare this step.
Differentiating function $w$, defined in \eqref{def_w_recall}, we
obtain

\begin{equation}\label{w_ij}
\begin{split}
\partial_{ij}w = &-a_i\sin x_i(\cos x_j+\cos 2x_j)(\sin x_k+\frac{1}{2}\sin 2x_k)-\\
&a_j\sin x_j(\cos x_i+\cos 2x_i)(\sin x_k+\frac{1}{2}\sin 2x_k)+\\
&a_k(1+\cos x_k)(\cos x_i+\cos 2x_i)(\cos x_j+\cos 2x_j), \, i<j,k\neq i,j,
\end{split}
\end{equation}
\begin{equation}\label{w_ii}
\begin{split}
\partial_{ii}w = &-a_i\cos x_i(\sin x_j+\frac{1}{2}\sin 2x_j)(\sin x_k+\frac{1}{2}\sin 2x_k)-\\
&a_j(1+\cos x_j)(\sin x_i+2\sin 2x_i)(\sin x_k+\frac{1}{2}\sin 2x_k)-\\
&a_k(1+\cos x_k)(\sin x_i+2\sin 2x_i)(\sin x_j+\frac{1}{2}\sin 2x_j), \, i\neq j,k, j<k.
\end{split}
\end{equation}

Let us introduce the following notation:

\begin{equation}\label{A_def}
\begin{split}
&A_{ijk}:=A(x_i,x_j,x_k)=  \\
&S(t,x_i,x_j,x_k;a_k\boldsymbol{\chi}_{\frac{\pi}{p}}(1+\cos p\xi_k)(\cos p\xi_i+\cos 2p\xi_i)
(\cos p\xi_j+\cos 2p\xi_j));
\end{split}
\end{equation}
\begin{equation}\label{B_def}
\begin{split}
&B_{ijk}:=B(x_i,x_j,x_k)=\\
&S(t,x_i,x_j,x_k;-a_i\boldsymbol{\chi}_{\frac{\pi}{p}}\sin p\xi_i(\cos p\xi_j+\cos 2p\xi_j)
(\sin p\xi_k+\frac{1}{2}\sin 2p\xi_k));
\end{split}
\end{equation}
\begin{equation}\label{C_def}
\begin{split}
&C_{ijk}:=C(x_i,x_j,x_k)=\\
&S(t,x_i,x_j,x_k;-a_i\boldsymbol{\chi}_{\frac{\pi}{p}}\cos p\xi_i(\sin p\xi_j+\frac{1}{2}\sin 2p\xi_j)
(\sin p\xi_k+\frac{1}{2}\sin 2\xi_k));
\end{split}
\end{equation}
\begin{equation}\label{D_def}
\begin{split}
&D_{ijk}:=D(x_i,x_j,x_k)=\\
&S(t,x_i,x_j,x_k;-a_j\boldsymbol{\chi}_{\frac{\pi}{p}}(1+\cos
p\xi_j)(\sin p\xi_i+2\sin 2p\xi_i)(\sin p\xi_k+\frac{1}{2}\sin
2p\xi_k));
\end{split}
\end{equation}

Then, denoting $x=(x_i,x_j,x_k)$ we get
\begin{equation}\label{S_ij}
S(t,x;\boldsymbol{\chi}_{\frac{\pi}{p}}\partial_{ij}w) =
A_{ijk}+B_{ijk}+B_{jik}, i,j,k = 1,2,3, i<j, k\neq i,j;
\end{equation}
\begin{equation}\label{S_ii}
S(t,x;\boldsymbol{\chi}_{\frac{\pi}{p}}\partial_{ii}w) =
C_{ijk}+D_{ijk}+D_{ikj}, i=2,3, j,k = 1,2,3, i\neq j,k, j<k.
\end{equation}

Now we are in position to prove the following
\begin{lem}\label{transform2}
The following relation holds:
\begin{equation}\label{S_w_summands}
\begin{split}
\int_{T^3}&((\mathbf{S}(t,{x};{u}),\nabla)\operatorname{curl}^{-1}\mathbf{S}(t,{x};{u}),\mathbf{S}(t,{x};{u}))dx= \\
p^6\int_{T^3}&\left(({A_{231}A_{123}^2}+{2B_{321}A_{123}B_{123}})-({A_{231}A_{132}^2}+{2B_{231}A_{132}B_{132}})\right)dx+\\
p^6\int_{T^3}&\left({B_{123}A_{132}D_{321}}+{B_{213}A_{132}D_{312}}+{B_{213}B_{132}D_{321}}\right)dx- \\
p^6\int_{T^3}&\left({A_{123}B_{132}D_{231}}+{A_{123}B_{312}D_{213}}+{B_{123}B_{312}D_{231}}\right)dx,
\end{split}
\end{equation}
where notations inserted above are used.
\end{lem}

\begin{proof}
Using \eqref{S_ij} let us consider the term
$\int_{T^3}S(t,x;\boldsymbol{\chi}_{\frac{\pi}{p}}\partial_{23}w)
(S(t,x;\boldsymbol{\chi}_{\frac{\pi}{p}}\partial_{12}w))^2dx$ in
\eqref{S_w}:
\begin{equation}\label{w23w12}
\begin{split}
&\int_{T^3}S(t,x;\boldsymbol{\chi}_{\frac{\pi}{p}}\partial_{23}w)
(S(t,x;\boldsymbol{\chi}_{\frac{\pi}{p}}\partial_{12}w))^2dx=\\
&\int_{T^3}\left(A_{231}+B_{231}+B_{321}\right)\left(A_{123}^2+B_{123}^2+B_{213}^2+\right.\\
&\left.\hphantom{A_{231}+B_{231}+B_{321}A_{123}^2}2A_{123}B_{123}+2A_{123}B_{213}+2B_{123}B_{213}\right)dx 
\end{split}
\end{equation}
Multiplying terms in \eqref{w23w12} we get two kinds of summands:
the first one are the functions even on all variables and the
second one are the functions that are odd on some variables.
Indeed, since $A_{123}=A(x_1,x_2,x_3)$ is even as function on all
variables and ${B_{123}=B(x_1,x_2,x_3)}$ is an even function on
$x_2$ and an odd one on $x_1$ and $x_3$ we get that for instance
the term $A_{231}A_{123}^2$ is even on all variables, but the term
$2A_{231}A_{123}B_{123}$ is odd on $x_1$ and $x_3$. Evidently,
integral of an odd function over symmetrical interval is equal to
zero, and that is why the integrals of all summands in
\eqref{w23w12} that are odd on at lest one variable will
disappear. Therefore, performing multiplication in \eqref{w23w12} and
leaving even summands only, we get
\begin{equation}\label{w23w12_fin}
\begin{split}
&\int_{T^3}S(t,x;\boldsymbol{\chi}_{\frac{\pi}{p}}\partial_{23}w)
S^2(t,x;\boldsymbol{\chi}_{\frac{\pi}{p}}\partial_{12}w)dx=\\
&\int_{T^3}\left({A_{231}A_{123}^2}+{A_{231}B_{123}^2}+{A_{231}B_{213}^2}+2B_{231}B_{123}B_{213}
+{2B_{321}A_{123}B_{123}}\right)dx.
\end{split}
\end{equation}
Moreover, since, according to Lemma \ref{S_split} and relation
\eqref{S_x} from lemma \ref{S_der},
\begin{equation}\label{eq41}
\begin{split}
&\int_{-\pi}^{\pi}S(t,x;\chi_{\frac{\pi}{p}}\cdot(\cos p\xi+\cos
2p\xi))
S^2(t,x;\chi_{\frac{\pi}{p}}\cdot(\sin p\xi+\frac{1}{2}\sin 2p\xi))dx = \\
&\frac{1}{3}\int_{-\pi}^{\pi}d\left(S^3(t,x;\chi_{\frac{\pi}{p}}\cdot
(\sin p\xi + \frac{1}{2}\sin 2p\xi))\right) = 0,
\end{split}
\end{equation}
all summands in \eqref{w23w12_fin} containing such a multiplier
are equal to zero. In other words
\begin{equation}\label{zero_terms1}
\begin{split}
\int_{\mathbb{T}_3}A_{ijk}B_{ikj}^2dx =
\int_{\mathbb{T}_3}&A_{ijk}B^2_{kij}dx =
\int_{\mathbb{T}_3}A_{ijk}B_{jki}^2dx = \int_{\mathbb{T}_3}A_{ijk}B^2_{kji}dx=\\
&\int_{\mathbb{T}_3}B_{ijk}B_{ikj}B_{kij}dx=0,
\end{split}
\end{equation}
Omitting all terms belonging to
\eqref{zero_terms1} in \eqref{w23w12_fin}  we get the final result:
\begin{equation}\label{w23w12_fin1}
\begin{split}
&\int_{T^3}S(t,x;\boldsymbol{\chi}_{\frac{\pi}{p}}\partial_{23}w)
S^2(t,x;\boldsymbol{\chi}_{\frac{\pi}{p}}\partial_{12}w)dx=\\
&\int_{T^3}\left({A_{231}A_{123}^2}+{2B_{321}A_{123}B_{123}}\right)dx.
\end{split}
\end{equation}

Note that relation \eqref{eq41} also implies the following
relations similar to \eqref{zero_terms1}:
\begin{equation}\label{zero_terms2}
\begin{split}
\int_{\mathbb{T}_3}A_{ijk}B_{ikj}C_{kij}dx
=&\int_{\mathbb{T}_3}A_{ijk}B_{kij}D_{kij}dx=
\int_{\mathbb{T}_3}B_{ijk}B_{ikj}D_{jik}dx=\\
&\int_{\mathbb{T}_3}B_{ijk}B_{jki}C_{ijk}dx=0.
\end{split}
\end{equation}

Similarly we transform all other summands in \eqref{S_w}: we
rewrite them using \eqref{S_ij}, \eqref{S_ii}, perform
multiplication and leave only summands which are even on all variables,
and after that omit summands from \eqref{zero_terms1},
\eqref{zero_terms2}. As the result we get:
\begin{equation}\label{w23w13_fin}
\begin{split}
&\int_{T^3}S(t,x;\boldsymbol{\chi}_{\frac{\pi}{p}}\partial_{23}w)
S^2(t,x;\boldsymbol{\chi}_{\frac{\pi}{p}}\partial_{13}w)dx=\\
&\int_{T^3}\left({A_{231}A_{132}^2}+{A_{231}B_{132}^2}+{A_{231}B_{312}^2}+
{2B_{231}A_{132}B_{132}}+{2B_{321}B_{132}B_{312}}\right)dx=\\
&\int_{T^3}\left({A_{231}A_{132}^2}+{2B_{231}A_{132}B_{132}}\right)dx
\end{split}
\end{equation}
\begin{equation}\label{w12w13w33_fin}
\begin{split}
&\int_{T^3}S(t,x;\boldsymbol{\chi}_{\frac{\pi}{p}}\partial_{12}w)
S(t,x;\partial_{13}w)S(t,x;\boldsymbol{\chi}_{\frac{\pi}{p}}\partial_{33}w)wdx=\\
&\int_{T^3}\left({A_{123}B_{132}C_{312}}+{{A_{123}B_{312}D_{312}}+
B_{123}A_{132}D_{321}}+{B_{123}B_{132}D_{312}}+\right.\\
&\left.\hphantom{\int_{T^3}}{{B_{123}B_{312}C_{312}}+B_{213}A_{132}D_{312}}+{B_{213}B_{132}D_{321}}\right)dx=\\
&\int_{T^3}\left({B_{123}A_{132}D_{321}}+{B_{213}A_{132}D_{312}}+{B_{213}B_{132}D_{321}}\right)dx
\end{split}
\end{equation}
\begin{equation}\label{w12w13w22_fin}
\begin{split}
&\int_{T^3}S(t,x;\boldsymbol{\chi}_{\frac{\pi}{p}}\partial_{12}w)
S(t,x;\boldsymbol{\chi}_{\frac{\pi}{p}}\partial_{13}w)S(t,x;\partial_{22}w)dx=\\
&\int_{T^3}\left({A_{123}B_{132}D_{213}}+{{A_{123}B_{312}D_{213}}+
B_{123}A_{132}C_{213}}+{B_{123}B_{132}D_{213}}+\right.\\
&\left.\hphantom{\int_{T^3}}{{B_{123}B_{312}D_{231}}+B_{213}A_{132}D_{213}}+{B_{213}B_{132}C_{213}}\right)dx=\\
&\int_{T^3}\left({A_{123}B_{132}D_{231}}+{A_{123}B_{312}D_{213}}+{B_{123}B_{312}D_{231}}\right)dx
\end{split}
\end{equation}
Summing up \eqref{w23w12_fin1}, \eqref{w23w13_fin},
\eqref{w12w13w33_fin} and \eqref{w12w13w22_fin}, we get
\eqref{S_w_summands}
\end{proof}

\subsection{The final reduction}\label{Finals1Dcase}
Let now express the functional from \eqref{main_est} via three
functionals on scalar functions of one variable. These
functionals are as follows:
\begin{equation}\label{J_1}
J_1(t)=\int_{-\pi}^{\pi}S^2(t,x;\chi_{\frac{\pi}{p}}\cdot(1+\cos p\xi))S(t,x;\chi_{\frac{\pi}{p}}\cdot(\cos p\xi+\cos 2p\xi))dx
\end{equation}

\begin{equation}\label{J_2}
J_2(t)=\int_{-\pi}^{\pi}S(t,x;\chi_{\frac{\pi}{p}}\cdot(1+\cos p\xi))S^2(t,x;\chi_{\frac{\pi}{p}}\cdot(\cos p\xi+\cos 2p\xi))dx
\end{equation}

\begin{equation}\label{J_3}
J_3(t)=\int_{-\pi}^{\pi}S^3(t,x;\chi_{\frac{\pi}{p}}\cdot(\cos p\xi+\cos 2p\xi))dx
\end{equation}

\begin{equation}\label{J_4}
\begin{split}
&J_4(t)=\int_{-\pi}^{\pi}S(t,x;\chi_{\frac{\pi}{p}}\cdot(\cos p\xi+\cos 2p\xi))\cdot\\
&\hphantom{S(t,x;\chi_{\frac{\pi}{p}}\cdot(\cos p\xi+\cos 2p\xi))}S(t,x;\chi_{\frac{\pi}{p}}\cdot(\sin p\xi))S(t,x;\chi_{\frac{\pi}{p}}\cdot(\sin p\xi+\frac{1}{2}\sin 2p\xi))dx
\end{split}
\end{equation}
where $S(t,x;\chi_{\frac{\pi}{p}}\phi(p\xi))$ is the solution of
heat equation \eqref{heat_1D} with periodic boundary condition and
initial condition $S(t,x)|_{t=0}=\chi_{\frac{\pi}{p}}\phi(p\xi)$.

\begin{theorem}\label{final_expression}
The functional from \eqref{main_est} can be represented as follows
\begin{equation}\label{expression_final}
\begin{split}
\int_{T^3}&((\mathbf{S}(t,{x};{u}),\nabla)\operatorname{curl}^{-1}\mathbf{S}(t,{x};{u}),\mathbf{S}(t,{x};{u}))dx=\\
&\frac{5}{4}(a_3^2a_1-a_2^2a_1)J_1(t)J_3(t)\left(J_2(t)+J_4(t)\right),
\end{split}
\end{equation}
where functionals $J_1,J_2,J_3,J_4$ are defined in \eqref{J_1}-\eqref{J_4} and $a_1,a_2,a_3$ are constants from
\eqref{def_w_recall}.
\end{theorem}
\begin{proof}
To prove theorem we have to transform terms from right side of
\eqref{S_w_summands}. Using \eqref{A_def},Lemma \ref{S_split} and
\eqref{J_1} -\eqref{J_4} we can transform the first term as
follows:
\begin{equation}\label{term1}
\begin{split}
 &\int_{\mathbb{T}^3}A_{231}A^2_{123}dx=\\
&\int_{\mathbb{T}^3}S(t,x_2,x_3,x_1;a_1\boldsymbol{\chi}_{\frac{\pi}{p}}
 (1+\cos p\xi_1)(\cos p\xi_2+\cos 2p\xi_2)(\cos p\xi_3+\cos2p\xi_3))\\
&S^2(t,x_1,x_2,x_3;a_3\boldsymbol{\chi}_{\frac{\pi}{p}}(1+\cos
p\xi_3)(\cos p\xi_1+\cos 2p\xi_1) (\cos p\xi_2+\cos 2p\xi_2))dx=\\
&a_1a^2_3\int_{-\pi}^\pi S(t,x_1;{\chi}_{\frac{\pi}{p}}(1+\cos
p\xi_1))S^2(t,x_1;{\chi}_{\frac{\pi}{p}}(\cos p\xi_1+\cos
2p\xi_1))dx_1\cdot \\
&\hphantom{a_1a^2_3}\int_{-pi}^\pi S^3(t,x_2;{\chi}_{\frac{\pi}{p}}(\cos p\xi_2+\cos
2p\xi_2))dx_2\cdot\\
&\hphantom{a_1a^2_3}\int_{-pi}^\pi S(t,x_3){\chi}_{\frac{\pi}{p}}(\cos
p\xi_3+\cos2p\xi_3))S^2(t,x_3;{\chi}_{\frac{\pi}{p}}(1+\cos
p\xi_3))=\\
&a_1a^2_3J_1(t)J_2(t)J_3(t)
\end{split}
\end{equation}
Similarly, using  \eqref{A_def}-\eqref{D_def}, Lemma \ref{S_split}
and \eqref{J_1}-\eqref{J_4} we get the following equalities
\begin{equation}\label{terms2-10}
\begin{split}
&\int_{\mathbb{T}^3}B_{321}A_{123}B_{123}dx=\frac12a_1a_3^2J_1J_3J_4,
\int_{\mathbb{T}^3}A_{231}A^2_{132}dx=a_1a_2^2J_1J_2J_3, \\
&\int_{\mathbb{T}^3}B_{231}A_{132}B_{132}dx=\frac12a_1a_2^2J_1J_3J_4,
\int_{\mathbb{T}^3}B_{123}A_{132}D_{321}dx=\frac12a_1a_2^2J_1J_3J_4,\\
&\int_{\mathbb{T}^3}B_{213}A_{132}D_{312}dx=\frac14a_1a_2^2J_1J_2J_3,
\int_{\mathbb{T}^3}B_{213}B_{132}D_{321}dx=-\frac14a_1a_2^2J_1J_3J_4,\\
&\int_{\mathbb{T}^3}A_{123}B_{132}D_{231}dx=\frac12a_1a_3^2J_1J_3J_4,
\int_{\mathbb{T}^3}A_{123}B_{312}D_{213}dx=\frac14a_1a_3^2J_1J_2J_3,\\
&\int_{\mathbb{T}^3}B_{123}B_{312}D_{231}dx=-\frac14a_1a_3^2J_1J_3J_4,
\end{split}
\end{equation}
Substituting \eqref{term1}, \eqref{terms2-10} into
\eqref{S_w_summands}, we get \label{expression_final}.
\end{proof}

\section{Estimate for $J_1(t)$}\label{s4}


Previously in \cite{F5}, \cite{FSh} it was established, that

\begin{equation}\label{J_3_est}
J_3(t) > C_3\cdot e^{-6t},
\end{equation}
where $C_3$ is a positive constant.

Our goal is to prove analogous estimates for $J_1(t)$ and
$J_2(t)+J_4(t)$.

In this section we shall prove the following

\begin{theorem}\label{thJ_1est}
For any $t\geq 0$ function $J_1(t)$ defined in \eqref{J_1}, satisfies the following inequality:

\begin{equation}\label{J1est}
J_1(t)\geq C _1 e^{-6t},
\end{equation}
where $C_1$ is some positive constant.
\end{theorem}

Let us divide the proof into several steps.

\subsection{Fourier decomposition of $J_1(t)$}\label{s4.1}
The solutions of the one dimensional heat equation
 $$\partial_t S-\partial_{xx}S=0$$
with periodic boundary condition and initial conditions $S(t,x)|_{t=0}=\chi_{\frac{\pi}{p}}\cdot(1+\cos p\xi)$ and $S(t,x)|_{t=0}=\chi_{\frac{\pi}{p}}\cdot(\cos p\xi+\cos 2p\xi)$ can be represented as

\begin{equation}\label{S_1cos}
S(t, x;\chi_{\frac{\pi}{p}}\cdot(1+\cos p\xi))=\frac{1}{p}+\sum_{k=1}^{\infty}c(k)\cos kx e^{-k^2t};
\end{equation}

\begin{equation}\label{S_coscos2}
S(t, x;\chi_{\frac{\pi}{p}}\cdot(\cos p\xi+\cos 2p\xi))=\sum_{k=1}^{\infty}d(k)\cos kx e^{-k^2t},
\end{equation}
where

\begin{equation}\label{c_k}
c(k)=\left\{\begin{array}{rl}
        &\dfrac{2p^2 \sin\frac{\pi k}{p}}{\pi k(p^2-k^2)},  k\neq p,\\
        &\dfrac{1}{p}, k = p,
            \end{array}\right.
\end{equation}
and
\begin{equation}\label{d_k}
d(k)=\left\{\begin{array}{rl}
        &\dfrac{6p^2k \sin\frac{\pi k}{p}}{\pi(p^2-k^2)(4p^2-k^2)}, k\neq p, k\neq 2p,\\
        &\dfrac{1}{p}, k = p, 2p
            \end{array}\right.
\end{equation}
are the Fourier coefficients of $\chi_{\frac{\pi}{p}}\cdot(1+\cos p\xi)$ and $\chi_{\frac{\pi}{p}}\cdot(\cos p\xi+\cos 2p\xi)$, where $\chi_{\frac{\pi}{p}}=\chi_{\frac{\pi}{p}}(\xi)$ was defined in \eqref{def_chi}, correspondingly.

Therefore, in virtue of \eqref{J_1} functional $J_1(t)$ is equal to
\begin{equation*}
\begin{split}
J_1(t)&=\int_{-\pi}^{\pi}\left(\frac{1}{p}+\sum_{k=1}^{+\infty}c(k)\cos kx\cdot e^{-k^2t}\right)^2\cdot\left(\sum_{l=1}^{+\infty}d(l)\cos lx\cdot e^{-l^2t}\right)dx=\\
&\int_{-\pi}^{\pi}\left(\frac{1}{p^2}\sum_{l=1}^{+\infty}d(l)\cos lx\cdot e^{-l^2t}\right.+\\
&\hphantom{\frac{1}{4}\sum_{k,m,l=1}^{+\infty}}\frac{1}{p}\sum_{m,l=1}^{+\infty}c(m)d(l)(\cos (m+l)x+\cos(m-l)x) e^{-(m^2+l^2)t}+\\
&\hphantom{\frac{1}{4}\sum_{k,m,l=1}^{+\infty}}\frac{1}{4}\sum_{k,m,l=1}^{+\infty}c(k)c(m)d(l)(\cos (k+m+l)x+\cos(k-m-l)x+\\
&\left.\hphantom{\frac{1}{4}\sum_{k,m,l=1}^{+\infty}}\cos(k+m-l)x+\cos(k-m+l)x)e^{-(k^2+m^2+l^2)t}\right)dx
\end{split}
\end{equation*}

Since
\begin{equation*}
\int_{-\pi}^{\pi}\cos nxdx=\left\{\begin{array}{rl}
        &0, n\neq 0,\\
        &2\pi, n = 0,
            \end{array}\right.
\end{equation*}
we get that
\begin{equation}\label{J1_series}
J_1(t)=\frac{2\pi}{p}J_{10}(t)+\frac{\pi}{2}(2J_{11}(t)+J_{12}(t)),
\end{equation}
where
\begin{equation}\label{J10_series}
J_{10}(t)=\sum_{m=1}^{\infty}c(m)d(m)e^{-2m^2t},
\end{equation}
\begin{equation}\label{J11_series}
J_{11}(t)=\sum_{m,l=1}^{\infty}c(m)d(l)c(m+l)e^{-2(m^2+l^2+ml)t}=:\sum_{m,l=1}^{\infty}F_{11}(m,l;t),
\end{equation}
and
\begin{equation}\label{J12_series}
J_{12}(t)=\sum_{m,l=1}^{\infty}c(m)c(l)d(m+l)e^{-2(m^2+l^2+ml)t}=:\sum_{m,l=1}^{\infty}F_{12}(m,l;t).
\end{equation}
Below we will prove below the positiveness of functions
\eqref{J10_series} - \eqref{J12_series}.


\subsection{Positiveness of $J_{10}$}

\begin{lem}\label{J1series_sum1pos}
The following inequality is true:
\begin{equation}\label{J1ineq1}
\sum_{m=1}^{\infty}c(m)d(m)e^{-2m^2t}>0.
\end{equation}
\end{lem}
\begin{proof}
Let us write sum \eqref{J1ineq1} as
\begin{equation}\label{J1ineq11}
\begin{split}
&\sum_{m=1}^{\infty}c(m)d(m)e^{-2m^2t}=\sum_{m=1}^{p-1}c(p-m)d(p-m)e^{-2(p-m)^2t}+\frac{1}{p^2}e^{-2p^2t}+\\
&\sum_{m=p+1}^{2p-1}c(m)d(m)e^{-2m^2t}+\sum_{a=2}^{+\infty}\sum_{m=1}^{p-1}c(ap+m)d(ap+m)e^{-2(ap+m)^2t}.
\end{split}
\end{equation}
According to \eqref{c_k}-\eqref{d_k}, all summands but the last one in \eqref{J1ineq11} are positive.

Let us consider the following relation:
\begin{equation}\label{J10rel1}
\begin{split}
&\frac{|c(ap+m)d(ap+m)|}{c(p-m)d(p-m)}<\\
&\frac{m^2(2p-m)^2(p+m)(3p-m)}{((a-1)p+m)^2((a+1)p+m)^2((a-2)p+m)((a+2)p+m)}
\end{split}
\end{equation}

For $a=2$ the right hand side of \eqref{J10rel1} turns into
\begin{equation}\label{J10rel2}
\frac{m(2p-m)^2(3p-m)}{(p+m)(3p+m)^2(4p+m)}=\frac{x}{1+x}\cdot\frac{(2-x)^2(3-x)}{(3+x)^2(4+x)},
\end{equation}
where $x=m/p$, $x\in(0,1)$.

It is easy to see, that the first multiple in the right hand side of \eqref{J10rel2} is an increasing function of $x$, and the second is a decreasing function of $x$, therefore,
\begin{equation}\label{a_eq_2}
\frac{m(2p-m)^2(3p-m)}{(p+m)(3p+m)^2(4p+m)}<\frac{1}{2}\cdot\frac{4\cdot 3}{9\cdot 4}=\frac{1}{6}.
\end{equation}

For $a\geq 3$ we get
\begin{equation}\label{J10rel3}
\begin{split}
&\frac{m^2(2p-m)^2(p+m)(3p-m)}{((a-1)p+m)^2((a+1)p+m)^2((a-2)p+m)((a+2)p+m)}=\\
&\frac{x^2(1+x)}{(a-1+x)^2(a-2+x)}\cdot\frac{(2-x)^2(3-x)}{(a+1+x)^2(a+2+x)},
\end{split}
\end{equation}
where $x=m/p$, $x\in(0,1)$. Due to monotonicity,
\begin{equation}\label{J10rel4}
\frac{x^2(1+x)}{(a-1+x)^2(a-2+x)}<\frac{2}{a^2(a-1)},\,\, \frac{(2-x)^2(3-x)}{(a+1+x)^2(a+2+x)}<\frac{12}{(a+1)^2(a+2)}.
\end{equation}
Relations \eqref{J10rel1}-\eqref{J10rel4} and inequality
$(a-1)(a+2)>a^2$ for $a\ge 3$ imply
\begin{equation}\label{5.17}
\begin{split}
&\frac{|\sum_{a=2}^{+\infty}c(ap+m)d(ap+m)|}{c(p-m)d(p-m)}<
\frac{1}{6}+\sum_{a=3}^{+\infty}\frac{24}{(a-1)a^2(a+1)^2(a+2)}<\\
&\frac{1}{6}+\sum_{a=3}^{+\infty}\frac{24}{a^6}<\frac{1}{6}+\int_2^\infty
\frac{24}{x^6}dx=\frac{1}{6}+\frac{3}{20}<1.
\end{split}
\end{equation}
Hence \eqref{J1ineq1} follows from \eqref{5.17}.
\end{proof}

\subsection{Sign distribution in $J_{11}$ and $J_{12}$}

Let us determine how the signs of the summands in $J_{11}(t)$ and
$J_{12}(t)$ are distributed. We will need the following lemma, that
is a direct corollary of definitions \eqref{c_k} and \eqref{d_k} of functions
$c(k)$ and $d(k)$:
\begin{lem}\label{cd_signs}
Let $c(k), d(k), k\in \mathbb{N}$ be the functions, defined in
\eqref{c_k}, \eqref{d_k} correspondingly. Then
\begin{equation*}
\begin{split}
   i&)\;  c(k)=0\quad \text {\rm for}\quad k=pl \quad\text{\rm where}\quad l\in \mathbb{N}\setminus\{1\},\\
   \hphantom{i)\;} &d(k)=0 \quad \text {\rm for}\quad k=pl \quad\text{\rm where}\quad l\in \mathbb{N}\setminus\{1,2\},\\
   ii&)\;  c(k)>0\quad \text{\rm for}\quad k\in (0,2p)\cup \{ \cup_{l\in \mathbb{N}}((2l+1)p,(2l+2)p)\\
   &d(k)>0\quad \text{\rm for}\quad k\in (0,3p)\cup\{\cup_{l\in \mathbb{N}}((2l+2)p,(2l+3)p)\}\\
   iii&)\;  c(k)<0\quad \text{\rm for} \quad k \in \cup_{l\in \mathbb{N}}(2lp,(2l+1)p)\\
   &d(k)<0\quad \text{\rm for}\quad k\in\cup_{l\in \mathbb{N}}((2l+1)p,(2l+2)p)
\end{split}
\end{equation*}
\end{lem}

The next two statements follow directly from Lemma \ref{cd_signs}:

\begin{lem}\label{lem_cdsigns1}
The signs of function $c(m)d(l)c(m+l)$ from \eqref{J11_series} are distributed as follows
(see Figure \ref{J11fig}):

i) For all $(m,l)\in(0,2p)\times(0,2p)$, $m+l<2p$ $c(m)d(l)c(m+l)>0$;

ii) In every square $${\{(m,l)\in (pa, p(a+1))\times
(pb,p(b+1))\}},$$ where ${a\in \mathbb{N}}$, ${b\in \{0\}\cup(\mathbb{N}\setminus\{1\})}$ or $a=0, b=1$
\begin{equation}\label{J11reg1}
     {\rm sign}\,c(m)d(l)c(m+l)=
       \begin{cases}
   +, & {\rm if}\; m+l<p(a+b+1) \\
   -, & {\rm if}\; m+l>p(a+b+1) \\
 \end{cases}
\end{equation}

iii) In every set $\{ (m,l)\in (pa, p(a+1))\times(pb,p(b+1))\}$, where $a=0, b\in \mathbb{N}\setminus\{1\}$ or $a\in \mathbb{N}, b = 1$,
\begin{equation}\label{J11reg2}
     {\rm sign}\,c(m)d(l)c(m+l)=
       \begin{cases}
   -, & {\rm if}\; m+l<p(a+b+1) \\
   +, & {\rm if}\; m+l>p(a+b+1) \\
 \end{cases}
\end{equation}
\end{lem}

\begin{lem}\label{lem_cdsigns2}
The signs of function $c(m)c(l)d(m+l)$ from \eqref{J12_series} are distributed as follows (see Figure \ref{J12fig}):


i) In every square $\{(k,l)\in (pa, p(a+1))\times
(pb,p(b+1))\}$, where $a,b\in \mathbb{N}$
\begin{equation}\label{sq_ab_g2}
     {\rm sign}c(m)c(l)d(m+l)=
       \begin{cases}
   +, & {\rm if}\; m+l<p(a+b+1) \\
   -, & {\rm if}\; m+l>p(a+b+1) \\
 \end{cases}
\end{equation}
ii) In every set $\{(m,l)\in (pa, p(a+1))\times
(0,p)\cup(0,p)\times (pa,p(a+1))\}$, where $a\in \mathbb{N}, a\geq2$,
\begin{equation}\label{2.2}
     {\rm sign}c(m)c(l)d(m+l)=
       \begin{cases}
   -, & {\rm if}\; m+l<p(a+1) \\
   +, & {\rm if}\; m+l>p(a+1) \\
 \end{cases}
\end{equation}


iii) For all $(k,l)\in (0,2p)\times (0,2p), k+l<3p$
\begin{equation}\label{2.4}
     c(m)c(l)d(m+l)>0.
     \end{equation}
\end{lem}

\begin{figure}[h!]
\begin{multicols}{2}
\hfill
\includegraphics[scale=0.7]{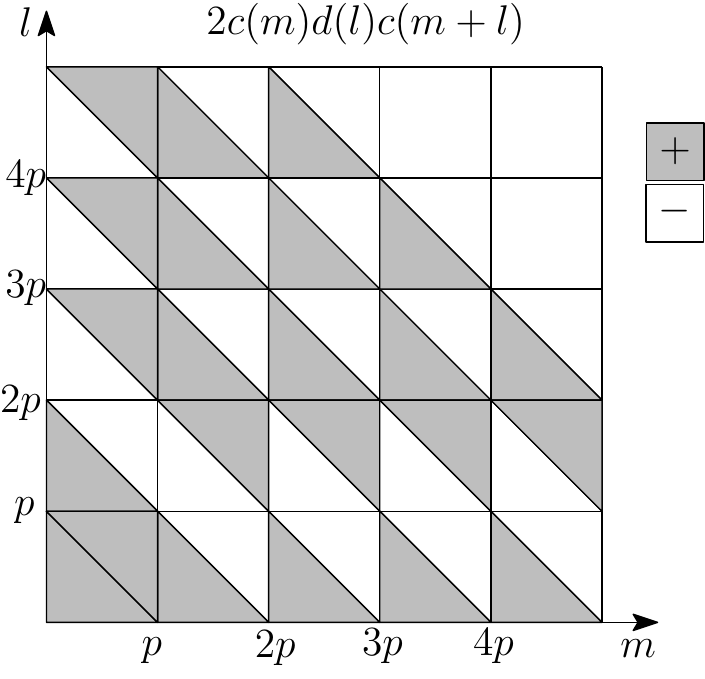}
\hfill
\caption{Signs of $c(m)d(l)c(m+l)$}
\label{J11fig}
\hfill
\includegraphics[scale=0.7]{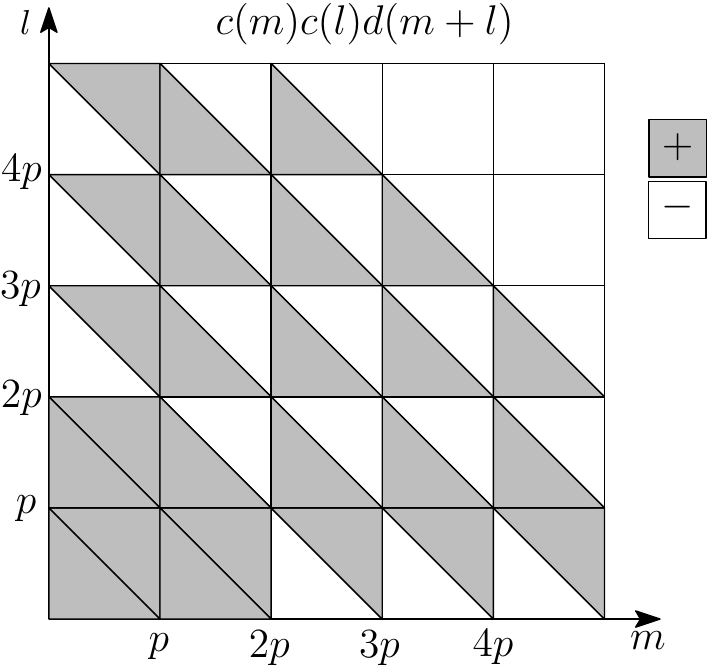}
\hfill
\caption{Signs of $c(m)c(l)d(m+l)$}\label{J12fig}
\end{multicols}
\end{figure}




\subsection{Positiveness of certain part of sum $J_{12}$ }\label{s4.4}
Let us denote the set of all points
from $\bN \times \bN$ belonging to the triangle with vertices at points
${(m_1,l_1)}$, ${(m_2,l_2)}$, ${(m_3,l_3)}$  by $\{(m_1,l_1),(m_2,l_2), (m_3,l_3)\}$, and use notation $$J_{11}(\{(m_1,l_1),(m_2,l_2), (m_3,l_3)\};t)$$, $$J_{12}(\{(m_1,l_1),(m_2,l_2), (m_3,l_3)\};t)$$ for part of the sums $J_{11}$, $J_{12}$, corresponding to this set.

First let us prove the following lemma:
\begin{lem}\label{J_12poslem1}
The following inequality is true:
\begin{equation}\label{J_12pos1}
\sum_{m,l=p+1}^{\infty}F_{12}(m,l;t) > 0,
\end{equation}
where $F_{12}(m,l;t)$ was defined in \eqref{J12_series}.
\end{lem}
\begin{proof}

Let us consider sets $\{(pa,p(a+1))\times(pb,p(b+1))\}, a, b\in \mathbb{N}$.
According to Lemma \ref{lem_cdsigns2}, all summands in triangle $\{(pa,pb),(pa,p(b+1)),(p(a+1),pb)\}$ are positive,
and all summands in triangle ${\{(pa,p(b+1)),(p(a+1),p(b+1)),(p(a+1),pb)\}}$ are negative.
If we show that the value of positive summands exceeds the value of negative summands, that would prove the
positiveness of the part of the sum $J_{12}$, corresponding to $m, l> p$.
Let us perform the change of variables $m=pa+m_1,\; l=pb+l_1$. Then
\begin{equation}\label{J12_1pos}
\begin{split}
&J_{12}(\{(pa, pb),(p(a+1), pb), (pa,p(b+1))\},t)= \\
&\sum^{p-1}_{\genfrac{}{}{0pt}{}{m_1,l_1=1}{m_1+l_1<p}}
 c(pa+m_1)c(pb+l_1)d(p(a+b)+m_1+l_1),
\end{split}
\end{equation}
\begin{equation}\label{J12_1neg}
\begin{split}
&J_{12}(\{(pa, p(b+1)),(p(a+1), pb), (p(a+1),p(b+1))\},t) = \\
&\sum^{p-1}_{\genfrac{}{}{0pt}{}{m_1,l_1=1}{m_1+l_1>p}}
 c(pa+m_1)c(pb+l_1)d(p(a+b)+m_1+l_1)
\end{split}
\end{equation}
Next, let us switch to variables  $s=p-m_1-l_1$, $m=m_1$ in \eqref{J12_1pos} and to variables $s=m_1+l_1-p$, $m=m_1-s$ in \eqref{J12_1neg} and consider the following relation:
\begin{equation} \label{J12_pos_to_J12_neg}
\begin{split}
&\frac{F_{12}(pa+m,p(b+1)-m-s;t)}{|F_{12}(pa+m+s,p(b+1)-m;t)|}=\\
&\frac{c_1(pa+m)c_1(p(b+1)-m-s)d_1(p(a+b+1)-s)}{c_1(pa+m+s)c_1(p(b+1)-m)d_1(p(a+b+1)+s)}\cdot e^{6(a+b+1)ps}>\\
&\frac{c_1(pa+m)c_1(p(b+1)-m-s)d_1(p(a+b+1)-s)}{c_1(pa+m+s)c_1(p(b+1)-m)d_1(p(a+b+1)+s)},
\end{split}
\end{equation}
where
\begin{equation}\label{c1_k}
c_1(k)=\dfrac{1}{k(k^2-p^2)},
\end{equation}
and
\begin{equation}\label{d1_k}
d_1(k)=\dfrac{k}{(k^2-p^2)(k^2-4p^2)}.
\end{equation}

Since $c_1(k)$ is decreasing for $k>p$ and $d_1(k)$ is decreasing for $k>2p$, the numerator of the fraction in the last line of \eqref{J12_pos_to_J12_neg} is greater than the denominator. Therefore, the fraction is greater than one, and $${J_{12}(\{(pa, pb),(p(a+1),pb), (pa,p(b+1))\};t)}$$ is greater than $${|J_{12}(\{(pa,p(b+1)),(p(a+1), pb), (p(a+1),p(b+1))\};t)|}.$$
\end{proof}

Next, let us prove the following statement:
\begin{lem}\label{J_12poslem2}
The following inequality is true:
\begin{equation}\label{J12_2pos1}
\frac{3}{10}J_{12}(\{(p,0),(2p,0),(p,p)\},t)+\sum_{a=2}^{+\infty}J_{12}(\{(pa,0),(p(a+1),0),(pa,p)\},t)>0,
\end{equation}
\begin{equation}\label{J12_2pos2}
\frac{3}{10}J_{12}(\{(0,p),(0,2p),(p,p)\},t)+\sum_{a=2}^{+\infty}J_{12}(\{(0,pa),(0,p(a+1)),(p,pa)\},t)>0.
\end{equation}
\end{lem}
\begin{proof}
Since summands in $J_{12}$ are symmetrical with respect to change ${(m,l)\rightarrow(l,m)}$, it is enough
to prove inequality \eqref{J12_2pos1} only.

According to Lemma\ref{lem_cdsigns2}, the summands in $J_{12}(\{(p,0),(2p,0),(p,p)\};t)$ are positive,
and the summands in $J_{12}(\{(pa,0),(p(a+1),0),(pa,p)\};t)$, $a\geq 2$, are negative.

Let us perform the change of variables $m=p+m_1,\; l=l_1$ in
$${J_{12}(\{(p,0),(2p,0),(p,p)\};t)}$$ and $m=pa+m_1,\; l=l_1$ in
$$J_{12}(\{(pa,0),(p(a+1),0),(pa,p)\};t),$$ $1\leq m_1\leq
p-1$, $1\leq l_1\leq p-1$. Then let us denote $m_1,l_1$ by $m,l$
and consider the following relation:
\begin{equation} \label{J12_pos2_to_J12_neg2}
\frac{|F_{12}(pa+m,l;t)|}{F_{12}(p+m,l;t)}<\frac{|c(pa+m)d(pa+m+l)|}{c(p+m)d(p+m+l)}.
\end{equation}

According to \eqref{c_k}-\eqref{d_k},

\begin{equation}\label{J12_pos2_to_J12_neg2a}
\begin{split}
&\frac{|c(pa+m)d(pa+m+l)|}{c(p+m)d(p+m+l)} = f_1(x)f_2(z),
\end{split}
\end{equation}
where $x = \frac{m}{p}$, $z = \frac{m+l}{p}$, and
\begin{equation}\label{f1}
f_1(x) = \frac{(1+x)}{(a+x)}\cdot\frac{x}{(a-1+x)}\cdot\frac{(2+x)}{(a+1+x)},
\end{equation}
\begin{equation}\label{f2}
f_2(z) = \frac{z}{(a-2+z)}\cdot\frac{a+z}{(a+1+z)}\cdot\frac{(2+z)}{(a-1+z)}\cdot\frac{(3+z)}{(a+2+z)}\cdot\frac{(1-z)}{(1+z)}.
\end{equation}

Obviously, all multipliers in \eqref{f1} are increasing functions of $x, x\in(0,1)$, and for $a\geq 3$
all multipliers but the last one in \eqref{f2} are increasing functions of $z$ and the last multiplier
is a decreasing function of $z$, $z\in(0,1)$. Therefore,
\begin{equation*}
f_1(x) < f_1(1)=\frac{6}{a(a+1)(a+2)},
\end{equation*}
\begin{equation*}
f_2(z) < \frac{12(a+1)}{(a-1)a(a+2)(a+3)}, \quad a\geq 3.
\end{equation*}

If $a=2$ then
$f_2(z)=\dfrac{(2+z)^2}{(1+z)^2}\cdot\dfrac{(1-z)}{(4+z)} < 1$ and
that is why
\begin{equation}\label{J12_pos2_to_J12_neg2b}
f_1(x)f_2(z)<\begin{cases}
              \frac{1}{4},  &  \mbox{for } a=2\\
              \frac{2}{75},  & \mbox{for } a=3\\
              \frac{1}{120},  &\mbox{for } a=4\\
               \end{cases}.
\end{equation}
Note that  $(a-1)(a+3)>a^2$ for $a\ge 3$ and therefore
\begin{equation}\label{J12_pos2_to_J12_neg2c}
f_1(x)f_2(z)<\frac{72}{(a-1)a^2(a+2)^2(a+3)}<\frac{72}{a^6}
\end{equation}
Taking  into account \eqref{J12_pos2_to_J12_neg2} -
\eqref{J12_pos2_to_J12_neg2c} we get:
\begin{equation*}
\begin{split}
&\sum_{a=2}^{+\infty}\frac{|F_{12}(pa+m,l;t)|}{F_{12}(p+m,l;t)}<
\frac{1}{4}+\frac{2}{75}+\frac{1}{120}+\\
&\sum_{a=5}^{+\infty}\frac{72}{(a-1)a^2(a+2)^2(a+3)}<0,285+\sum_{a=5}^\infty
\frac{72}{a^6}<0,285+\int_{4}^\infty
\frac{72dx}{x^6}<\frac{3}{10}.
\end{split}
\end{equation*}
Therefore, inequalities \eqref{J12_2pos1}-\eqref{J12_2pos2} are established.
\end{proof}


\subsection{Positiveness of a certain part of sum $J_{12}$ together with some negative summands from $J_{11}$}

The following two lemmas show, that the sums
${ J_{12}(\{(p,0),(2p,0),(p,p)\};t)}$ and
$J_{12}(\{(0,p),(0,2p),(p,p)\};t)$ also compensate some negative
summands from $J_{11}(t)$:

\begin{lem}\label{J12andJ11_1}
The following inequality is true:
\begin{equation}\label{J12and11_in1}
\begin{split}
&\frac{1}{5}\left[J_{12}(\{(p,0),(2p,0),(p,p)\};t)+J_{12}(\{(0,p),(0,2p),(p,p)\};t)\right]+\\
& 2\sum_{a=1}^{+\infty}J_{11}(\{(pa,p),(p(a+1),p),(pa,2p)\};t)> 0.
\end{split}
\end{equation}
\end{lem}

\begin{proof}
Because of the symmetry of the summands in $J_{12}(t)$, $$J_{12}(\{(p,0),(2p,0),(p,p)\};t)=J_{12}(\{(0,p),(0,2p),(p,p)\};t),$$ therefore inequality  \eqref{J12and11_in1} is equivalent to
\begin{equation}\label{J12and11_in11}
\frac{1}{5}J_{12}(\{(p,0),(2p,0),(p,p)\};t)+\sum_{a=1}^{+\infty}J_{11}(\{(pa,p),(p(a+1),p),(pa,2p)\};t)> 0.
\end{equation}
Let us rewrite $J_{12}(\{(p,0),(2p,0),(p,p)\};t)$ as
\begin{equation}\label{J12_orange_1}
J_{12}(\{(p,0),(2p,0),(p,p)\};t)= \sum^{p-1}_{\genfrac{}{}{0pt}{}{m,l=1}{m+l<p}}F_{12}(m,p+l;t),
\end{equation}
and $J_{11}(\{(pa,p),(p(a+1),p),(pa,2p)\};t)$ as
\begin{equation}\label{J11_blue_1}
J_{11}(\{(pa,p),(p(a+1),p),(pa,2p)\};t)= \sum^{p-1}_{\genfrac{}{}{0pt}{}{m,l=1}{m+l<p}}F_{11}(ap+m,p+l;t)
\end{equation}
and consider the ratio of the absolute values of the summands from \eqref{J11_blue_1} to the summands from \eqref{J12_orange_1}, corresponding to the same $m$ and $l$:
\begin{equation}\label{J12_orange_1toJ11_blue_1}
\begin{split}
&\frac{|F_{11}(ap+m,p+l;t)|}{F_{12}(m,p+l;t)}<\frac{|c(ap+m)d(p+l)c((a+1)p+m+l)|}{c(m)c(p+l)d(p+m+l)}=\\
&\frac{m(p-m)(p+m)}{(ap+m)((a-1)p+m)((a+1)p+m)}\cdot\frac{(p+l)^2}{(p-l)(3p+l)}\cdot\\
&\frac{(m+l)(2p+m+l)(p-m-l)(3p+m+l)}{(p+m+l)((a+1)p+m+l)(ap+m+l)((a+2)p+m+l)}<\\
&\frac{m(p-m)(p+m)}{(ap+m)((a-1)p+m)((a+1)p+m)}\cdot\frac{(p+l)^2}{(p-l)(3p+l)}\cdot\\
&\frac{(m+l)(2p+m+l)(p-l)(3p+m+l)}{(p+m+l)((a+1)p+m+l)(ap+m+l)((a+2)p+m+l)}=\\
&f_1(x)\cdot f_2(y)\cdot f_3(z),
\end{split}
\end{equation}
where $x=m/p$, $y=l/p$, $z=x+y$, $x,y,z\in(0,1)$ and
$$f_1(x)=\frac{x}{a-1+x}\cdot\frac{1+x}{a+x}\cdot\frac{1-x}{a+1+x},$$
$$f_2(y)=\frac{(1+y)^2}{3+y},$$
$$f_3(z)=\frac{z}{1+z}\cdot\frac{2+z}{a+1+z}\cdot\frac{3+z}{a+2+z}\cdot\frac{1}{a+z}.$$
Since $$f'_2(y) = \frac{(1+y)(5+y)}{(3+y)^2}>0$$ for $y\in(0,1)$, function $f_2(y)$ increases on the interval $(0,1)$, therefore,
\begin{equation}\label{f2(y)less_than_1}
f_2(y)<f_2(1)=1.
\end{equation}
When $a=1$,
$$f_1(x)=\frac{1-x}{2+x},$$
$$f_3(z)=\frac{z}{(1+z)^2}.$$
It is easy to see that for $a=1$ function $f_1(x)$ decreases, and
$f_3(z)$ increases on $(0,1)$, therefore, taking into account
\eqref{f2(y)less_than_1}, we get that
\begin{equation}
\frac{|F_{11}(p+m,p+l;t)|}{F_{12}(m,p+l;t)}<f_1(0)\cdot f_3(1)<\frac{1}{8}.
\end{equation}

When $a\geq 2$, the first two multipliers in $f_1(x)$ and the first three multipliers in $f_3(z)$ are indecreasing functions, while the last ones are decreasing functions, therefore, taking in account \eqref{f2(y)less_than_1}, we get
\begin{equation*}
\begin{split}
&\frac{|F_{11}(ap+m,p+l;t)|}{F_{12}(m,p+l;t)}<f_1(x)\cdot f_3(z)<\\
&\frac{12}{a^2(a+1)^2(a+2)(a+3)}.
\end{split}
\end{equation*}
Note that $a^2(a+2)(a+3)>(a+1)^4$ for $a\ge 3$. Therefore,
\begin{equation*}
\begin{split}
&\frac{\sum_{a=1}^{+\infty}|F_{11}(ap+m,l;t)|}{F_{12}(m,p+l;t)}<
\frac{1}{8}+\frac{1}{60}+\sum_{a=3}^{+\infty}\frac{12}{a^2(a+1)^2(a+2)(a+3)}<\\
&\frac{17}{120}+\sum_{a=3}^\infty
\frac{12}{(a+1)^6}<\frac{17}{120}+\int_{3}^\infty
\frac{12dx}{x^6}<\frac{17}{120}+\frac{4}{405} <\frac{1}{5},
\end{split}
\end{equation*}
which implies \eqref{J12and11_in11} and \eqref{J12and11_in1}.
\end{proof}

\begin{lem}\label{J12andJ11_2}
The following inequality is true:
\begin{equation}\label{J12and11_in2}
\begin{split}
&\frac{3}{50}\left[J_{12}(\{(p,0),(2p,0),(p,p)\},t)+J_{12}(\{(0,p),(0,2p),(p,p)\},t)\right]+\\
& 2\sum_{a=2}^{+\infty}J_{11}(\{(0,pa),(0,p(a+1)),(p,ap)\},t)> 0.
\end{split}
\end{equation}
\end{lem}
\begin{proof}
Because of symmetry,
$J_{12}(\{(p,0),(2p,0),(p,p)\},t)=J_{12}(\{(0,p),(0,2p),(p,p)\},t)$
inequality  \eqref{J12and11_in2} is equivalent to
\begin{equation}\label{J12and11_in21}
\frac{3}{50}J_{12}(\{(0,p),(0,2p),(p,p)\},t)+\sum_{a=2}^{+\infty}J_{11}(\{(0,pa),(0,p(a+1)),(p,ap)\},t)>
0.
\end{equation}
Let us rewrite $J_{12}(\{(0,p),(0,2p),(p,p)\},t)$ as
\begin{equation}\label{J12_orange_3}
J_{12}(\{(0,p),(0,2p),(p,p)\},t)= \sum^{p-1}_{\genfrac{}{}{0pt}{}{m,l=1}{m+l<p}}F_{12}(m,p+l;t),
\end{equation}
and $J_{11}(\{(0,pa),(0,p(a+1)),(p,ap)\},t)$ as
\begin{equation}\label{J11_blue_2}
J_{11}(\{(0,pa),(0,p(a+1)),(p,ap)\},t)= \sum^{p-1}_{\genfrac{}{}{0pt}{}{m,l=1}{m+l<p}} F_{11}(m,ap+l;t)
\end{equation}
and consider the ratio of the absolute values of the summands from \eqref{J11_blue_2} to the summands from \eqref{J12_orange_3}, corresponding to the same $m$ and $l$:
\begin{equation}\label{J12_orange_2toJ11_blue_2}
\begin{split}
\frac{|F_{11}(m,ap+l;t)|}{F_{12}(m,p+l;t)}&<\frac{|d(ap+l)c((a+1)p+m+l)|}{c(p+l)d(p+m+l)}=\\
&f_4(y)\cdot f_5(z),
\end{split}
\end{equation}
where $y=l/p$, $z=(m+l)/p$, $y,z\in(0,1)$ and
$$f_4(y)=\frac{a+y}{a+1+y}\cdot\frac{1+y}{a-1+y}\cdot\frac{y}{a-2+y}\cdot\frac{2+y}{a+2+y},$$
$$f_5(z)=\frac{z}{1+z}\cdot\frac{2+z}{a+1+z}\cdot\frac{3+z}{a+2+z}\cdot\frac{1-z}{a+z}.$$

Since $a\ge 2$, it is easy to see that $f_4(y)$ increases, so
$f_4(y)<f_4(1)$. Let us consider the numerator of $f_5(z)$:
$$\frac{d}{dz}(z(1-z)(2+z)(3+z))=-2(z+1)\left(z-\left(-1-\sqrt{5/2}\right)\right)\left(z-\left(-1+\sqrt{5/2}\right)\right).$$
So, the numerator of $f_5(z)$ reaches its maximum on the interval
$(0,1)$ at point $z~=~-1+\sqrt{5/2}$ and
$$z(1-z)(2+z)(3+z)\leq \frac{9}{4}.$$
Denominator of $f_5(z)$ is an increasing function for $z\in (0,1)$,
and it achieves minimum, equal to $a(a+1)(a+2)$, at $z=0$.

 Therefore, since $a^2(a-1)(a+2)^2(a+3)>a^6$, $a\geq 4$, and $$\sum_{a=4}^\infty
\frac{1}{a^6}<\int_3^\infty \frac{dx}{x^6}=\frac{1}{5\cdot 3^5},$$ we get
\begin{equation*}
\begin{split}
&\sum_{a=2}^{+\infty}\frac{|F_{11}(m,ap+l;t)|}{F_{12}(m,p+l;t)}<\sum_{a=2}^{+\infty}f_4(1)\cdot
f_5(z)< \sum_{a=2}^\infty \frac{27}{2a^2(a-1)(a+2)^2(a+3)}<\\
&\frac{27}{2}\left(\frac{1}{320}+\frac{1}{2700}+\sum_{a=4}^\infty
\frac{1}{a^6}\right)<\frac{48}{1000}+\frac{1}{90}<\frac{3}{50}.
\end{split}
\end{equation*}
\end{proof}

Finally, let us prove the following two lemmas:
\begin{lem}
The following inequality is true:
\begin{equation}\label{J_1211pos}
J_{12}(\{(p,p),(2p,0),(2p,p)\};t)+2J_{11}(\{(p,p),(2p,0),(2p,p)\};t) > 0.
\end{equation}
\end{lem}
\begin{proof}
Let us consider the relation of summands from $J_{12}(t)$ to the summands from $J_{11}(t)$, corresponding to the same coordinates $(2p-m,p-l)$, $m,l\in[1,p-1]$:
\begin{equation} \label{J12_pos_to_J11_neg}
\begin{split}
&\frac{F_{12}(2p-m,p-l;t)}{2|F_{11}(2p-m,p-l,t)|}>\frac{c(2p-m)c(p-l)d(3p-m-l)}{2|c(2p-m)d(p-l)c(3p-m-l)|}=\\
&\frac{(1+x)(3-x)}{2(1-x)^2}\cdot\frac{(3-x-y)^2}{(1-x-y)(5-x-y)}=:g_1(x)g_2(z),
\end{split}
\end{equation}
where $x = \frac{l}{p}$, $y = \frac{m}{p}$, $z=x+y$.
Since
\begin{equation*}
g_1(x) = -\frac{1}{2}+\frac{2}{(1-x)^2}, \quad x\in(0,1),
\end{equation*}
\begin{equation*}
g_2(z) =1+ \frac{4}{(1-z)(5-z)}, \quad z\in(0,1),
\end{equation*}
both functions are increasing on $(0,1)$, therefore, $$g_1(x)g_2(x)>g_1(0)g_2(0) = \frac{27}{10},$$ which implies inequality \eqref{J_1211pos}.
\end{proof}
\begin{lem}
The following inequality is true:
\begin{equation}\label{J12redJ11turq}
J_{12}(\{(0,p),(p,0),(p,p)\};t)+2J_{11}(\{(0,2p),(p,p),(p,2p)\};t) > 0.
\end{equation}
\end{lem}
\begin{proof}
Let us rewrite $J_{12}(\{(0,p),(p,0),(p,p)\},t)$ as
\begin{equation}\label{J12_red}
J_{12}(\{(0,p),(p,0),(p,p)\};t)= \sum^{p-1}_{\genfrac{}{}{0pt}{}{m,l=1}{m+l<p}}F_{12}(p-m,p-l;t),
\end{equation}
and $J_{11}(\{(0,2p),(p,p),(p,2p)\};t)$ as
\begin{equation}\label{J11_turq}
J_{11}(\{(0,2p),(p,p),(p,2p)\};t)= \sum^{p-1}_{\genfrac{}{}{0pt}{}{m,l=1}{m+l<p}}F_{11}(p-m,2p-l;t)
\end{equation}
and consider the relation of the absolute values of summands from \eqref{J11_turq} to the summands from \eqref{J12_red}, corresponding to the same  $m,l\in[1,p-1]$:
\begin{equation*}
\begin{split}
&\frac{2|F_{11}(p-m,2p-l;t)|}{F_{12}(p-m,p-l;t)|}<\frac{2c(p-m)d(2p-l)c(3p-m-l)}{c(p-m)c(p-l)d(2p-m-l)}=\\
&\frac{2(2-y)^2}{(3-y)(4-y)}\cdot\frac{z(1-z)}{(2-z)^2}:=g_3(y)g_4(z),
\end{split}
\end{equation*}
where $y = \frac{l}{p}$, $z = \frac{l}{p}+\frac{m}{p}$, $y,z\in(0,1)$.

It is easy to see, that $g_3(y)$ decreases on the interval $(0,1)$, so $g_3(y)<g_3(0)=\dfrac{2}{3}$.

Since $z(1-z)\leq \dfrac{1}{4}$, and $(2-z)^2>1$, $z\in(0,1)$, we finally get, that
\begin{equation*}
\begin{split}
&\frac{2|F_{11}(p-m,2p-l;t)|}{F_{12}(p-m,p-l,t)|}<\frac{2}{3}\cdot\frac{1}{4}=\frac{1}{6},
\end{split}
\end{equation*}
which implies \eqref{J12redJ11turq}.
\end{proof}

\subsection{Positiveness of the remaining part of $J_{11}$}

In this paragraph we shall prove the following statement:
\begin{lem}\label{J_11poslem}
Part of the sum $J_{11}(t)$, defined in \eqref{J11_series}, corresponding to the sets ${\{(m,l):m>p, l>2p\}}$ and $\{(m,l)\in([1,+\infty]\times[1,p])\setminus(\{(p,p),(2p,0),(2p,p)\})\}$, is positive $\forall t>0$.
\end{lem}
\begin{proof}
1) To prove the positiveness of  part of $J_{11}$, defined in \eqref{J11_series}, corresponding to sets $\{(m,l):m>p, l>2p\}$, let us consider squares
 $$\{(pa,p(a+1))\times(pb,p(b+1))\}, a, b\in \mathbb{N}, b\geq 2.$$ According to Lemma\ref{lem_cdsigns1}, all summands in triangle $$\{(pa,pb),(pa,p(b+1)),(p(a+1),pb)\}$$ are positive,
and all summands in triangle $$\{(pa,p(b+1)),(p(a+1),p(b+1)),(p(a+1),pb)\}$$ are negative. If we show, that the value of positive summands exceeds the value of negative summands, that would prove the positiveness of the part of the sum $J_{11}(t)$, corresponding to $m > p, l>2p$. Let us perform the change of variables $m=pa+m_1,\; l=pb+l_1$. Then
\begin{equation}\label{J11_1pos}
\begin{split}
&J_{11}(\{(pa, pb),(p(a+1), pb), (pa,p(b+1))\};t)= \\
&\sum^{p-1}_{\genfrac{}{}{0pt}{}{m_1,l_1=1}{m_1+l_1<p}}
 c(pa+m_1)d(pb+l_1)c(p(a+b)+m_1+l_1),
\end{split}
\end{equation}
\begin{equation}\label{J11_1neg}
\begin{split}
&J_{11}(\{(pa, p(b+1)),(p(a+1), pb), (p(a+1),p(b+1))\};t) = \\
&\sum^{p-1}_{\genfrac{}{}{0pt}{}{m_1,l_1=1}{m_1+l_1>p}}
 c(pa+m_1)d(pb+l_1)c(p(a+b)+m_1+l_1)
\end{split}
\end{equation}
Next, let us switch to variables  $s=p-m_1-l_1$, $m=m_1$ in \eqref{J11_1pos} and to
variables $s=m_1+l_1-p$, $m=m_1-s$ in \eqref{J11_1neg} and consider the relation of the
summand from \eqref{J11_1pos} to the summand from \eqref{J11_1neg} with the same $s$, $m$:
\begin{equation} \label{J11_pos_to_J11_neg}
\begin{split}
&\frac{F_{11}(pa+m,p(b+1)-m-s;t)}{|F_{11}(pa+m+s,p(b+1)-m;t)|}=\\
&\frac{c_1(pa+m)d_1(p(b+1)-m-s)c_1(p(a+b+1)-s)}{c_1(pa+m+s)d_1(p(b+1)-m)c_1(p(a+b+1)+s)}\cdot e^{6t(a+b+1)ps}>\\
&\frac{c_1(pa+m)d_1(p(b+1)-m-s)c_1(p(a+b+1)-s)}{c_1(pa+m+s)d_1(p(b+1)-m)c_1(p(a+b+1)+s)},
\end{split}
\end{equation}
where $c_1(k)$ and $d_1(k)$ were defined in \eqref{c1_k}-\eqref{d1_k}.

Since $c_1(k)$ is decreasing for $k>p$ and $d_1(k)$ is decreasing for $k>2p$, the numerator of the fraction in the last line of \eqref{J11_pos_to_J11_neg} is greater than the denominator, therefore, the fraction is greater than one and
\begin{equation}\label{J11pos_fin1}
\begin{split}
&J_{11}(\{(pa, pb),(p(a+1),pb), (pa,p(b+1))\})>\\
&\hphantom{(pa, pb),(p(a+1),pb)}|J_{11}(\{(pa,p(b+1)),(p(a+1), pb), (p(a+1),p(b+1))\})|.
\end{split}
\end{equation}

2) According to Lemma \ref{lem_cdsigns1}, all summands in $J_{11}(\{(0,p),(p,p),(p,0)\},t)$ are positive, and all summands in $J_{11}(\{(pa,p),(p(a+1),p),(p(a+1),0)\},t)$. $a\geq1$, are negative.
Therefore, in order to prove that
\begin{equation}\label{J11_pos2}
J_{11}(\{(0,p),(p,p),(p,0)\},t)+\sum_{a=2}^{+\infty}J_{11}(\{(pa,p),(p(a+1),p),(p(a+1),0)\},t)>0,
\end{equation}
it is sufficient to show that the values of positive summands are not less that the absolute values of negative summands.

Let us rewrite $J_{11}(\{(0,p),(p,p),(p,0)\};t)$ as
\begin{equation*}
J_{11}(\{(0,p),(p,p),(p,0)\};t)=\sum_{m,l=1}^{p-1}F_{11}(p-m,p-l;t)
\end{equation*}
and  $J_{11}(\{(pa,p),(p(a+1),p),(p(a+1),0)\};t)$ as
\begin{equation*}
J_{11}(\{(pa,p),(p(a+1),p),(p(a+1),0)\};t)=\sum_{m,l=1}^{p-1}F_{11}(pa-m,p-l;t)
\end{equation*}
and consider the following relation:
\begin{equation} \label{J11_pos2_to_J11_neg2}
\frac{|F_{11}(pa-m,p-l;t)|}{F_{11}(p-m,p-l;t)}<\frac{|c(pa-m)c(p(a+1)-m-l)|}{c(p-m)c(2p-m-l)} =h_1(x)h_2(z),
\end{equation}
where $x = \frac{m}{p}$, $z = \frac{m+l}{p}$, and
\begin{equation}\label{h1}
h_1(x) = \frac{x}{(a-1-x)}\cdot\frac{1-x}{a-x}\cdot\frac{2-x}{a+1-x},
\end{equation}
\begin{equation}\label{h2}
h_2(z) = \frac{1-z}{a-z}\cdot\frac{2-z}{a+1-z}\cdot\frac{3-z}{a+2-z}.
\end{equation}

Obviously, all multipliers in \eqref{h2} are decreasing functions of $z, z\in(0,1)$, all multipliers but the first one in \eqref{h1} are decreasing functions of $x$ and the first multiplier is an increasing function of $x$, $x\in(0,1)$. Therefore,
\begin{equation*}
h_1(x) < \frac{2}{a^2(a+1)},
\end{equation*}
\begin{equation*}
h_2(z) < \frac{6}{a(a+1)(a+2)},
\end{equation*}
and
\begin{equation*}
\begin{split}
&\sum_{a=2}^{+\infty}\frac{|F_{11}(pa-m,p-l;t)|}{F_{11}(p-m,p-l;t)}<\sum_{a=2}^{+\infty}\frac{12}{a^3(a+1)^2(a+2)}<\\
&\frac{1}{24}+\int_{2}^{+\infty}\frac{12dx}{x^3(x+1)^2(x+2)} <
\frac{1}{24}+\int_{2}^{+\infty}\frac{12dx}{x^6}=\frac{7}{60} < 1.
\end{split}
\end{equation*}
Therefore, inequality \eqref{J11_pos2} is established. Together with \eqref{J11pos_fin1} this completes the proof of the lemma.
\end{proof}

According to Lemmas \ref{J1series_sum1pos}-\ref{J_11poslem} and equality \eqref{J1_series},
\begin{equation*}
\begin{split}
J_1(t)&=\frac{2\pi}{p}J_{10}(t)+\frac{\pi}{2}(2J_{11}(t)+J_{12}(t))>\\
&\frac{\pi}{2}(2F_{11}(1,1;t)+F_{12}(1,1;t))=C_1\cdot e^{-6t},
\end{split}
\end{equation*}
where $C_1>0$, which completes the proof of theorem \ref{thJ_1est}.

\section{Estimate for $J_2(t)+J_4(t)$}\label{s5}

In this section we begin the proof of the following theorem:
\begin{theorem}\label{thJ_24est}
For any $t\geq 0$ function $J_2(t)+J_4(t)$ defined in \eqref{J_2}, \eqref{J_4}, satisfies the following inequality:

\begin{equation}\label{J24est}
J_2(t)+J_4(t)\geq C_2 e^{-6t},
\end{equation}
where $C_2$ is some positive constant.
\end{theorem}
\begin{proof}
In order to prove the estimate \eqref{J24est}, let us find the derivative of $J_2(t)+J_4(t)$. According to Lemma\ref{S_der},

\begin{equation}\label{J4der1}
\begin{split}
\frac{d}{dt}J_4(t)&=\int_{-\pi}^{\pi}\partial_{xx}S(t,x;\chi_{\frac{\pi}{p}}\sin p\xi)S(t,x;\chi_{\frac{\pi}{p}}\cdot(\sin p\xi+1/2\sin 2p\xi))\cdot\\
&\hphantom{\int_{-\pi}^{\pi}\partial_{xx}}S(t,x;\chi_{\frac{\pi}{p}}\cdot(\cos p\xi+\cos 2p\xi))dx-\\
&-p^2\int_{-\pi}^{\pi}S(t,x;\chi_{\frac{\pi}{p}}\sin p\xi)S(t,x;\chi_{\frac{\pi}{p}}\cdot(\sin p\xi+2\sin 2p\xi))\cdot\\
&\hphantom{\int_{-\pi}^{\pi}\partial_{xx}}S(t,x;\chi_{\frac{\pi}{p}}\cdot(\cos p\xi+\cos 2p\xi))dx-\\
&-p^2\int_{-\pi}^{\pi}S(t,x;\chi_{\frac{\pi}{p}}\sin p\xi)S(t,x;\chi_{\frac{\pi}{p}}\cdot(\sin p\xi+1/2\sin 2p\xi))\cdot\\
&\hphantom{\int_{-\pi}^{\pi}\partial_{xx}}S(t,x;\chi_{\frac{\pi}{p}}\cdot(\cos p\xi+4\cos 2p\xi))dx=\\
&\int_{-\pi}^{\pi}\partial_{xx}S(t,x;\chi_{\frac{\pi}{p}}\sin p\xi)S(t,x;\chi_{\frac{\pi}{p}}\cdot(\sin p\xi+1/2\sin 2p\xi))\cdot\\
&\hphantom{\int_{-\pi}^{\pi}\partial_{xx}}S(t,x;\chi_{\frac{\pi}{p}}\cdot(\cos p\xi+\cos 2p\xi))dx-8p^2J_4(t)+\\
&3p^2\int_{-\pi}^{\pi}S(t,x;\chi_{\frac{\pi}{p}}\sin p\xi)S(t,x;\chi_{\frac{\pi}{p}}\cdot(\sin p\xi+1/2\sin 2p\xi))\cdot\\
&\hphantom{\int_{-\pi}^{\pi}\partial_{xx}}S(t,x;\chi_{\frac{\pi}{p}}\cos p\xi)dx+\\
&3p^2\int_{-\pi}^{\pi}S^2(t,x;\chi_{\frac{\pi}{p}}\sin
p\xi)S(t,x;\chi_{\frac{\pi}{p}}\cdot(\cos p\xi+\cos 2p\xi))dx.
\end{split}
\end{equation}

Let us consider the first summand in the right hand side of \eqref{J4der1}:
\begin{equation*}
\begin{split}
&\int_{-\pi}^{\pi}\partial_{xx}S(t,x;\chi_{\frac{\pi}{p}}\sin p\xi)S(t,x;\chi_{\frac{\pi}{p}}\cdot(\sin p\xi+1/2\sin 2p\xi))\cdot\\
&\hphantom{\int_{-\pi}^{\pi}\partial_{xx}}S(t,x;\chi_{\frac{\pi}{p}}\cdot(\cos p\xi+\cos 2p\xi))dx=\\
&\int_{-\pi}^{\pi}\partial_{xx}[S(t,x;\chi_{\frac{\pi}{p}}\cdot(\sin p\xi+1/2\sin 2p\xi))S(t,x;\chi_{\frac{\pi}{p}}\cdot(\cos p\xi+\cos 2p\xi))]\cdot\\
&\hphantom{\int_{-\pi}^{\pi}\partial_{xx}}S(t,x;\chi_{\frac{\pi}{p}}\sin p\xi)dx =\\
&-p^2\int_{-\pi}^{\pi}S(t,x;\chi_{\frac{\pi}{p}}\cdot(\sin p\xi+2\sin 2p\xi))S(t,x;\chi_{\frac{\pi}{p}}\cdot(\cos p\xi+\cos 2p\xi))\cdot\\
&\hphantom{\int_{-\pi}^{\pi}\partial_{xx}}S(t,x;\chi_{\frac{\pi}{p}}\sin p\xi)dx -\\
&2p^2\int_{-\pi}^{\pi}S(t,x;\chi_{\frac{\pi}{p}}\cdot(\cos p\xi+\cos 2p\xi))S(t,x;\chi_{\frac{\pi}{p}}\cdot(\sin p\xi+2\sin 2p\xi))\cdot\\
&\hphantom{\int_{-\pi}^{\pi}\partial_{xx}}S(t,x;\chi_{\frac{\pi}{p}}\sin p\xi)dx -\\
&p^2\int_{-\pi}^{\pi}S(t,x;\chi_{\frac{\pi}{p}}\cdot(\sin p\xi+1/2\sin 2p\xi))S(t,x;\chi_{\frac{\pi}{p}}\cdot(\cos p\xi+4\cos 2p\xi))\cdot\\
&\hphantom{\int_{-\pi}^{\pi}\partial_{xx}}S(t,x;\chi_{\frac{\pi}{p}}\sin p\xi)dx=\\
&-16p^2J_4(t)+9p^2\int_{-\pi}^{\pi}S^2(t,x;\chi_{\frac{\pi}{p}}\sin p\xi)S(t,x;\chi_{\frac{\pi}{p}}\cdot(\cos p\xi+\cos 2p\xi))dx+\\
&3p^2\int_{-\pi}^{\pi}S(t,x;\chi_{\frac{\pi}{p}}\sin p\xi)S(t,x;\chi_{\frac{\pi}{p}}\cdot(\sin p\xi+1/2\sin 2p\xi))
S(t,x;\chi_{\frac{\pi}{p}}\cos p\xi)dx.\\
\end{split}
\end{equation*}

Therefore,
\begin{equation*}
\begin{split}
\frac{d}{dt}&J_4(t)=-24p^2J_4(t)+\\
&6p^2\int_{-\pi}^{\pi}S(t,x;\chi_{\frac{\pi}{p}}\sin p\xi)S(t,x;\chi_{\frac{\pi}{p}}\cdot(\sin p\xi+1/2\sin 2p\xi))
S(t,x;\chi_{\frac{\pi}{p}}\cos p\xi)dx+\\
&12p^2\int_{-\pi}^{\pi}S^2(t,x;\chi_{\frac{\pi}{p}}\sin
p\xi)S(t,x;\chi_{\frac{\pi}{p}}\cdot(\cos p\xi+\cos 2p\xi))dx.
\end{split}
\end{equation*}

Integrating by parts, we get that
\begin{equation*}
\begin{split}
&\int_{-\pi}^{\pi}S(t,x;\chi_{\frac{\pi}{p}}\sin p\xi)S(t,x;\chi_{\frac{\pi}{p}}\cdot(\sin p\xi+1/2\sin 2p\xi))\cdot\\
&\hphantom{\int_{-\pi}^{\pi}\partial_{xx}}S(t,x;\chi_{\frac{\pi}{p}}\cos p\xi)dx=\\
&\frac{1}{2p}\int_{-\pi}^{\pi}S(t,x;\chi_{\frac{\pi}{p}}\cdot(\sin p\xi+1/2\sin 2p\xi))d(S^2(t,x;\chi_{\frac{\pi}{p}}\sin p\xi))=\\
&-\frac{1}{2}\int_{-\pi}^{\pi}S^2(t,x;\chi_{\frac{\pi}{p}}\sin p\xi)S(t,x;\chi_{\frac{\pi}{p}}\cdot(\cos p\xi+\cos 2p\xi))dx.\\
\end{split}
\end{equation*}
Therefore, finally,
\begin{equation*}
\begin{split}
\frac{d}{dt}J_4(t)&=-24p^2J_4(t)+\\
&9p^2\int_{-\pi}^{\pi}S^2(t,x;\chi_{\frac{\pi}{p}}\sin
p\xi)S(t,x;\chi_{\frac{\pi}{p}}\cdot(\cos p\xi+\cos 2p\xi))dx.
\end{split}
\end{equation*}

Next, let us find the derivative of $J_2(t)$:
\begin{equation*}
\begin{split}
\frac{d}{dt}&J_2(t)=\\
&\int_{-\pi}^{\pi}\partial_{xx}S(t,x;\chi_{\frac{\pi}{p}}\cdot(1+\cos p\xi))S^2(t,x;\chi_{\frac{\pi}{p}}\cdot(\cos p\xi+\cos 2p\xi))dx+\\
&2\int_{-\pi}^{\pi}S(t,x;\chi_{\frac{\pi}{p}}\cdot(1+\cos p\xi))S(t,x;\chi_{\frac{\pi}{p}}\cdot(\cos p\xi+\cos 2p\xi))\cdot\\
&\hphantom{\int_{-\pi}^{\pi}\partial_{xx}}\partial_{xx}(S(t,x;\chi_{\frac{\pi}{p}}\cdot(\cos p\xi+\cos 2p\xi)))dx=\\
&2p^2\int_{-\pi}^{\pi}S(t,x;\chi_{\frac{\pi}{p}}\cdot(1+\cos p\xi))S^2(t,x;\chi_{\frac{\pi}{p}}\cdot(\sin p\xi+2\sin 2p\xi))dx-\\
&-4p^2\int_{-\pi}^{\pi}S(t,x;\chi_{\frac{\pi}{p}}\cdot(1+\cos p\xi))S(t,x;\chi_{\frac{\pi}{p}}\cdot(\cos p\xi+\cos 2p\xi))\cdot\\
&\hphantom{\int_{-\pi}^{\pi}\partial_{xx}}S(t,x;\chi_{\frac{\pi}{p}}\cdot(\cos p\xi+4\cos 2p\xi))dx=-16p^2J_2(t)+\\
&2p^2\int_{-\pi}^{\pi}S(t,x;\chi_{\frac{\pi}{p}}\cdot(1+\cos p\xi))S^2(t,x;\chi_{\frac{\pi}{p}}\cdot(\sin p\xi+2\sin 2p\xi))dx+\\
&12p^2\int_{-\pi}^{\pi}S(t,x;\chi_{\frac{\pi}{p}}\cdot(1+\cos p\xi))S(t,x;\chi_{\frac{\pi}{p}}\cdot(\cos p\xi+\cos 2p\xi))\cdot\\
&\hphantom{\int_{-\pi}^{\pi}\partial_{xx}}S(t,x;\chi_{\frac{\pi}{p}}\cos
p\xi)dx.
\end{split}
\end{equation*}
Therefore, $J_2(t)+J_4(t)$ satisfies the following differential equation:
\begin{equation}\label{J24eq}
\begin{split}
\frac{d}{dt}&(J_2(t)+J_4(t))+24p^2(J_2+J_4)=\\
&8p^2J_2(t)+9p^2\int_{-\pi}^{\pi}S^2(t,x;\chi_{\frac{\pi}{p}}\sin p\xi)S(t,x;\chi_{\frac{\pi}{p}}\cdot(\cos p\xi+\cos 2p\xi))dx+\\
&2p^2\int_{-\pi}^{\pi}S(t,x;\chi_{\frac{\pi}{p}}\cdot(1+\cos p\xi))S^2(t,x;\chi_{\frac{\pi}{p}}\cdot(\sin p\xi+2\sin 2p\xi))dx+\\
&12p^2\int_{-\pi}^{\pi}S(t,x;\chi_{\frac{\pi}{p}}\cdot(1+\cos p\xi))S(t,x;\chi_{\frac{\pi}{p}}\cdot(\cos p\xi+\cos 2p\xi))\cdot\\
&\hphantom{\int_{-\pi}^{\pi}\partial_{xx}}S(t,x;\chi_{\frac{\pi}{p}}\cos
p\xi)dx.
\end{split}
\end{equation}

The initial value of $J_2(t)+J_4(t)$ is equal to
$\dfrac{11\pi}{4p}$. Therefore, if we prove that the right-hand
side of \eqref{J24eq} is greater than zero, it would mean that
\begin{equation}\label{J24pos}
J_2(t)+J_4(t)\geq \frac{11\pi}{4p}e^{-24p^2t}>0.
\end{equation}

All summands in the right hand side of \eqref{J24eq}, except, possibly,
\begin{equation*}
Q(t):=\int_{-\pi}^{\pi}S^2(t,x;\chi_{\frac{\pi}{p}}\sin
p\xi)S(t,x;\chi_{\frac{\pi}{p}}\cdot(\cos p\xi+\cos 2p\xi))dx,
\end{equation*}
and
\begin{equation*}
R(t):=\int_{-\pi}^{\pi}S(t,x;\chi_{\frac{\pi}{p}}\!\cdot\!(1+\cos
p\xi))S(t,x;\chi_{\frac{\pi}{p}}\cos
p\xi)S(t,x;\chi_{\frac{\pi}{p}}\!\cdot\!(\cos p\xi+\cos 2p\xi))
\end{equation*}
are obviously positive. Let us consider these two summands.

Differentiating function $Q(t)$, we get the following relation:
\begin{equation*}
\begin{split}
\frac{d}{dt}&Q(t)=-4p^2Q(t)+\\
&3p^2\int_{-\pi}^{\pi}S^2(t,x;\chi_{\frac{\pi}{p}}\sin p\xi)S(t,x;\chi_{\frac{\pi}{p}}\cos p\xi)dx+\\
&2\int_{-\pi}^{\pi}S(t,x;\chi_{\frac{\pi}{p}}\sin
p\xi)\partial_{xx}S(t,x;\chi_{\frac{\pi}{p}}\sin
p\xi)S(t,x;\chi_{\frac{\pi}{p}}\cdot(\cos p\xi+\cos 2p\xi))dx.
\end{split}
\end{equation*}

Let us note, that
\begin{equation}\label{zeroint}
\begin{split}
&\int_{-\pi}^{\pi}S^2(t,x;\chi_{\frac{\pi}{p}}\sin p\xi)S(t,x;\chi_{\frac{\pi}{p}}\cos p\xi)dx=\\
&\frac{1}{p}\int_{-\pi}^{\pi}S^2(t,x;\chi_{\frac{\pi}{p}}\sin p\xi)d(S(t,x;\chi_{\frac{\pi}{p}}\sin p\xi))=\\
&\frac{1}{3p}\int_{-\pi}^{\pi}d(S^3(t,x;\chi_{\frac{\pi}{p}}\sin
p\xi))=0.
\end{split}
\end{equation}

Therefore,
\begin{equation*}
\begin{split}
\frac{d}{dt}&Q(t)=\\%
&-4p^2Q(t)+2p\int_{-\pi}^{\pi}S(t,x;\chi_{\frac{\pi}{p}}\sin p\xi)\partial_{x}S(t,x;\chi_{\frac{\pi}{p}}\cos p\xi)\cdot\\
&\hphantom{\int_{-\pi}^{\pi}\partial_{xx}}S(t,x;\chi_{\frac{\pi}{p}}\cdot(\cos p\xi+\cos 2p\xi))dx=\\
&-4p^2Q(t)-2p\int_{-\pi}^{\pi}\partial_{x}[S(t,x;\chi_{\frac{\pi}{p}}\sin p\xi)S(t,x;\chi_{\frac{\pi}{p}}\cdot(\cos p\xi+\cos 2p\xi))]\cdot\\
&\hphantom{\int_{-\pi}^{\pi}\partial_{xx}}S(t,x;\chi_{\frac{\pi}{p}}\cos p\xi)dx=\\
&-4p^2Q(t)-2p^2\int_{-\pi}^{\pi}S^2(t,x;\chi_{\frac{\pi}{p}}\cos p\xi)S(t,x;\chi_{\frac{\pi}{p}}\cdot(\cos p\xi+\cos 2p\xi))dx-\\
&2p^2\int_{-\pi}^{\pi}S(t,x;\chi_{\frac{\pi}{p}}\sin
p\xi)S(t,x;\chi_{\frac{\pi}{p}}\cdot(\sin p\xi+2\sin
2p\xi))S(t,x;\chi_{\frac{\pi}{p}}\cos p\xi)dx.
\end{split}
\end{equation*}
Integrating by parts and taking into account \eqref{zeroint}, we get that
\begin{equation*}
\begin{split}
&2p^2\int_{-\pi}^{\pi}S(t,x;\chi_{\frac{\pi}{p}}\sin p\xi)S(t,x;\chi_{\frac{\pi}{p}}\cdot(\sin p\xi+2\sin 2p\xi))
S(t,x;\chi_{\frac{\pi}{p}}\cos p\xi)dx=\\
&-p^2\int_{-\pi}^{\pi}S^2(t,x;\chi_{\frac{\pi}{p}}\sin p\xi)S(t,x;\chi_{\frac{\pi}{p}}\cdot(\cos p\xi+4\cos 2p\xi))dx=\\
&-4p^2\int_{-\pi}^{\pi}S^2(t,x;\chi_{\frac{\pi}{p}}\sin
p\xi)S(t,x;\chi_{\frac{\pi}{p}}\cdot(\cos p\xi+\cos
2p\xi))dx=-4p^2Q(t).
\end{split}
\end{equation*}

Finally, we get that
\begin{equation*}
\begin{split}
\frac{d}{dt}Q(t)&=-8p^2Q(t)-\\
&-2p^2\int_{-\pi}^{\pi}S^2(t,x;\chi_{\frac{\pi}{p}}\cos
p\xi)S(t,x;\chi_{\frac{\pi}{p}}\cdot(\cos p\xi+\cos 2p\xi))dx.
\end{split}
\end{equation*}

Similarly, let us find the derivative of $R(t)$:
\begin{equation*}
\begin{split}
\frac{d}{dt}R(t)&=-p^2\int_{-\pi}^{\pi}S^2(t,x;\chi_{\frac{\pi}{p}}\cos p\xi)S(t,x;\chi_{\frac{\pi}{p}}\cdot(\cos p\xi+\cos 2p\xi))dx-\\
&4p^2R(t)+3p^2\int_{-\pi}^{\pi}S^2(t,x;\chi_{\frac{\pi}{p}}\cos p\xi)S(t,x;\chi_{\frac{\pi}{p}}\cdot(1+\cos p\xi))dx+\\
&\int_{-\pi}^{\pi}S(t,x;\chi_{\frac{\pi}{p}}\cos p\xi)\partial_{xx}[S(t,x;\chi_{\frac{\pi}{p}}\cdot(1+\cos p\xi))\cdot\\
&\hphantom{\int_{-\pi}^{\pi}\partial_{xx}}S(t,x;\chi_{\frac{\pi}{p}}\cdot(\cos p\xi+\cos 2p\xi))]dx=\\
&-8p^2R(t)+6p^2\int_{-\pi}^{\pi}S^2(t,x;\chi_{\frac{\pi}{p}}\cos p\xi)S(t,x;\chi_{\frac{\pi}{p}}\cdot(1+\cos p\xi))dx-\\
&2p^2\int_{-\pi}^{\pi}S^2(t,x;\chi_{\frac{\pi}{p}}\cos p\xi)S(t,x;\chi_{\frac{\pi}{p}}\cdot(\cos p\xi+\cos 2p\xi))dx+\\
&2p^2\int_{-\pi}^{\pi}S(t,x;\chi_{\frac{\pi}{p}}\cos p\xi)S(t,x;\chi_{\frac{\pi}{p}}\sin p\xi)\cdot\\
&\hphantom{\int_{-\pi}^{\pi}\partial_{xx}}S(t,x;\chi_{\frac{\pi}{p}}\cdot(\sin p\xi+2\sin 2p\xi))dx
\end{split}
\end{equation*}

Taking into account relation \eqref{zeroint}, we finally get that
\begin{equation*}
\begin{split}
\frac{d}{dt}&R(t)=-8p^2R(t)+\\
&6p^2\int_{-\pi}^{\pi}S^2(t,x;\chi_{\frac{\pi}{p}}\cos p\xi)S(t,x;\chi_{\frac{\pi}{p}}\cdot(1+\cos p\xi))dx-\\
&2p^2\int_{-\pi}^{\pi}S^2(t,x;\chi_{\frac{\pi}{p}}\cos p\xi)S(t,x;\chi_{\frac{\pi}{p}}\cdot(\cos p\xi+\cos 2p\xi))dx-\\
&4p^2\int_{-\pi}^{\pi}S^2(t,x;\chi_{\frac{\pi}{p}}\sin
p\xi)S(t,x;\chi_{\frac{\pi}{p}}\cos 2p\xi)dx.
\end{split}
\end{equation*}

Finally, we arrive to the following differential equation:
\begin{equation}\label{ABeq}
\begin{split}
&\frac{d}{dt}(12R(t)+9Q(t))+8p^2(12R(t)+Q(t))=\\
&p^2\int_{-\pi}^{\pi}S^2(t,x;\chi_{\frac{\pi}{p}}\cos p\xi)S(t,x;\chi_{\frac{\pi}{p}}\cdot(30(1+\cos p\xi)+42(1-\cos 2p\xi)))dx-\\
&48p^2\int_{-\pi}^{\pi}S^2(t,x;\chi_{\frac{\pi}{p}}\sin
p\xi)S(t,x;\chi_{\frac{\pi}{p}}\cos 2p\xi)dx.
\end{split}
\end{equation}


Let us show, that Theorem \ref{thJ_24est} follows from the following:
\begin{theorem}\label{th_tJ_est}
Let $\tilde{J}(t)$ be the function
\begin{equation}\label{J_def}
\tilde{J}(t):=-\int_{-\pi}^{\pi}S^2(t,x;\chi_{\frac{\pi}{p}}\sin
p\xi)S(t,x;\chi_{\frac{\pi}{p}}\cos 2p\xi)dx
\end{equation}
from the right hand side of \eqref{ABeq}. Then
\begin{equation}\label{tJ_est0}
\tilde{J}(t)>\alpha\cdot e^{-6t} \quad \forall t>0,
\end{equation}
where
\begin{equation}\label{alpha}
\alpha=\frac{9\sin^2{\frac{\pi}{p}}\sin{\frac{2\pi}{p}}}{2\pi^2p^6}.
\end{equation}
\end{theorem}

The proof of Theorem \ref{th_tJ_est} is complicated and will be given in the next section.

Let us now consider equation \eqref{ABeq}.
Obviously,
$$\int_{-\pi}^{\pi}S^2(t,x;\chi_{\frac{\pi}{p}}\cos p\xi)S(t,x;\chi_{\frac{\pi}{p}}(30(1+\cos p\xi)+42(1-\cos 2p\xi)))dx>0.$$

Let us denote $12R(t)+9Q(t)$ by $g(t)$. According to
\eqref{tJ_est0}, equation \eqref{ABeq} implies that
\begin{equation*}
\frac{d}{dt}g(t)+8p^2g(t)\ge \alpha e^{-6t} + h(t),
\end{equation*}
where $\alpha$ is defined in \eqref{alpha}, and $h(t)\geq 0$ $\forall t\geq 0$.

Since
\begin{equation*}
\begin{split}
g(0) = 9\int_{-\pi/p}^{\pi/p}&\sin^2{px}(\cos px+\cos 2px)dx+\\
&12\int_{-\pi/p}^{\pi/p}(1+\cos px)\cos px(\cos px+\cos 2px)dx =\dfrac{27\pi}{2p},
\end{split}
\end{equation*}
we get, that
\begin{equation*}
g(t)=\frac{\alpha}{8p^2-6}e^{-6t} + \left(\dfrac{27\pi}{2p}-\frac{\alpha}{8p^2-6}\right)e^{-8p^2t}+\int_{0}^{t}e^{8p^2(\tau-t)}h(\tau)d\tau,
\end{equation*}
which obviously implies
\begin{equation}\label{QR_est}
12R(t)+9Q(t)>\frac{\alpha}{8p^2-6}e^{-6t}.
\end{equation}

Analogously, let us consider \eqref{J24eq}:
\begin{equation*}
\begin{split}
\frac{d}{dt}&(J_2(t)+J_4(t))+24p^2(J_2+J_4)=\\
&8p^2J_2(t)+9p^2Q(t)+12p^2R(t) +\\
&2p^2\int_{-\pi}^{\pi}S(t,x;\chi_{\frac{\pi}{p}}(1+\cos p\xi))S^2(t,x;\chi_{\frac{\pi}{p}}(\sin p\xi+2\sin 2p\xi))dx,
\end{split}
\end{equation*}
According to \eqref{QR_est},
\begin{equation*}
\frac{d}{dt}(J_2(t)+J_4(t))+24p^2(J_2+J_4)\ge
\alpha_1e^{-6t} + h_1(t),
\end{equation*}
where $\alpha_1 = \dfrac{\alpha}{8p^2-6}$ and $h_1(t)\geq 0$ $\forall t\geq 0$.

Since $(J_2+J_4)(0)=\dfrac{11\pi}{4p}$,
\begin{equation}\label{J24finest}
\begin{split}
J_2(t)+&J_4(t)=\\
&\frac{\alpha_1}{24p^2-6}e^{-6t} + \left(\dfrac{11\pi}{4p}-\frac{\alpha_1}{24p^2-6}\right)e^{-24p^2t}+\int_{0}^{t}e^{24p^2(\tau-t)}h(\tau)d\tau.
\end{split}
\end{equation}
Estimate \eqref{J24est} follows directly from \eqref{J24finest}.

\end{proof}

\section{Proof of Theorem \ref{th_tJ_est}: the first step}\label{s6}

In this section we begin to prove Theorem \ref{th_tJ_est}.

\subsection{Fourier decomposition.}

The solutions of the heat equation
\begin{equation*}
\partial_t S-\partial_{xx}S=0
\end{equation*}
with periodic boundary condition and initial conditions
$S|_{t=0}=\chi_{\frac{\pi}{p}}(x)\sin px $ and
$S|_{t=0}=\chi_{\frac{\pi}{p}}(x)\cos 2px$ can be represented as

\begin{equation}\label{S_sin}
S(t, x;\chi_{\frac{\pi}{p}}\sin p\xi)=\sum_{k=1}^{\infty}a_1(k)\sin kx e^{-k^2t};
\end{equation}

\begin{equation}\label{S_cos2}
S(t, x;-\chi_{\frac{\pi}{p}}\cos 2p\xi)=\sum_{k=1}^{\infty}b_1(k)\cos kx e^{-k^2t},
\end{equation}

where

\begin{equation}\label{a1_k}
a_1(k)=\left\{\begin{array}{rl}
        &\dfrac{2p \sin\frac{\pi k}{p}}{\pi(p^2-k^2)},  k\neq p,\\
        &\dfrac{1}{p}, k = p
            \end{array}\right.
\end{equation}
and
\begin{equation}\label{b1_k}
b_1(k)=\left\{\begin{array}{rl}
        &\dfrac{2k \sin\frac{\pi k}{p}}{\pi(4p^2-k^2)}, k\neq 2p,\\
        &-\dfrac{1}{p}, k = 2p
            \end{array}\right.
\end{equation}
are the Fourier coefficients of $\chi_{\frac{\pi}{p}}(x)\sin px$
and $(-\chi_{\frac{\pi}{p}}(x)\cos 2px)$ correspondingly.

Therefore, $\tilde{J}(t)$ is equal to
\begin{equation*}
\begin{split}
&\tilde{J}(t)=\sum_{k,l,m=1}^{\infty}a_1(k)a_1(m)b_1(l)e^{-(k^2+l^2+m^2)t}\int_{-\pi}^{\pi}\sin kx \sin mx\cos lx dx=\\
&\frac{1}{4}\sum_{k,l,m=1}^{\infty}a_1(k)a_1(m)b_1(l)e^{-(k^2+l^2+m^2)t}\int_{-\pi}^{\pi}\left(\cos(l+k-m)x+\right.\\
&\hphantom{\sum_{k,l,m=1}^{\infty}a(k)}\left.\cos(l-k+m)x-\cos(l+k+m)x-\cos(l-k-m)x\right)dx
\end{split}
\end{equation*}

Since
\begin{equation*}
\int_{-\pi}^{\pi}\cos nxdx=\left\{\begin{array}{rl}
        &0, n\neq 0,\\
        &2\pi, n = 0,
            \end{array}\right.
\end{equation*}
we get that
\begin{equation*}
\tilde{J}(t)=\frac{4 p^2}{\pi^2}{J}(t),
\end{equation*}
where
\begin{equation}\label{J_series}
{J}(t)=\sum_{k,l=1}^{\infty}\left(2A(k)A(k+l)B(l)-A(k)A(l)B(k+l)\right)e^{-(k^2+l^2+(k+l)^2)t},
\end{equation}
\begin{equation}\label{A_k}
A(k)=\left\{\begin{array}{rl}
        &\dfrac{\sin\frac{\pi k}{p}}{p^2-k^2},  k\neq p,\\
        &\dfrac{\pi}{2p^2}, k = p
            \end{array}\right.
\end{equation}

\begin{equation}\label{B_k}
B(k)=\left\{\begin{array}{rl}
        &\dfrac{k \sin\frac{\pi k}{p}}{4p^2-k^2}, k\neq 2p,\\
        &-\dfrac{\pi}{2p}, k = 2p
            \end{array}\right.
\end{equation}

Theorem \ref{th_tJ_est} follows directly from the following statement:
\begin{theorem}\label{th_J_est}
Let ${J}(t)$ be the function defined in \eqref{J_series}-\eqref{B_k}.
Then
\begin{equation}\label{tJ_est}
{J}(t)>\frac{9\sin^2{\frac{\pi}{p}}\sin{\frac{2\pi}{p}}}{4p^8}e^{-6t}\,\,
\forall t>0,
\end{equation}
\end{theorem}

The proof of Theorem \ref{th_J_est} is similar to the proof of Theorem \ref{thJ_1est}. First let us determine how the signs of the summands in  \eqref{J_series} are distributed.

We will need the following lemma, directly following from definitions \eqref{A_k} and \eqref{B_k} of functions $A(k)$ and $B(k)$:
\begin{lem}\label{AB_signs}
Let $A(k), k\in \mathbb{N}$ be the function, defined in \eqref{A_k} and $B(k), k\in \mathbb{N}$ be the function, defined in \eqref{B_k}. Then
\begin{equation*}
\begin{split}
   i&)\;  A(k)=0\quad \text {\rm for}\quad k=pl \quad\text{\rm where}\quad l\in \mathbb{N}\setminus\{1\},\\
   \hphantom{i)\;} &B(k)=0 \quad \text {\rm for}\quad k=pl \quad\text{\rm where}\quad l\in \mathbb{N}\setminus\{2\},\\
   ii&)\;  A(k)>0\quad \text{\rm for}\quad k\in (0,p)\cup (p,2p)\cup \{ \cup_{l\in \mathbb{N}}((2l+1)p,(2l+2)p)\\
   &B(k)>0\quad \text{\rm for}\quad k\in (0,p)\cup \{ \cup_{l\in \mathbb{N}}((2l+1)p,(2l+2)p)\}\\
   iii&)\;  A(k)<0\quad \text{\rm for} \quad k \in \cup_{l\in \mathbb{N}}(2lp,(2l+1)p)\\
   &B(k)<0\quad \text{\rm for}\quad k\in (p,2p)\cup (2p,3p)\cup_{l\in \mathbb{N}}(2lp,(2l+1)p)
\end{split}
\end{equation*}
\end{lem}

The next two statements follow directly from Lemma \ref{AB_signs}:

\begin{lem}\label{lem_signs1}
The signs of function $A(k)A(k+l)B(l)$ from \eqref{J_series} are distributed as follows (see Figure \ref{J21fig}):

i) In every square $$\{ (k,l)\in (pa, p(a+1))\times
(pb,p(b+1))\},$$ where $a\in \mathbb{N}, b\in \mathbb{N}, b\geq 2$
\begin{equation}\label{1reg1}
     {\rm sign}(A(k)A(k+l)B(l))=
       \begin{cases}
   -, & {\rm if}\; k+l<p(a+b+1) \\
   +, & {\rm if}\; k+l>p(a+b+1) \\
 \end{cases}
\end{equation}

ii) In every set $$\{ (k,l)\in (pa, p(a+1))\times(pb,p(b+1))\},$$ where $a\in \mathbb{N}$, $b\in\{0,1\}$
\begin{equation}\label{1reg2}
     {\rm sign}(A(k)A(k+l)B(l))=
       \begin{cases}
   +, & {\rm if}\; k+l<p(a+b+1) \\
   -, & {\rm if}\; k+l>p(a+b+1) \\
 \end{cases}
\end{equation}
iii) In every set $\{(k,l)\in(0,p)\times (pa,p(a+1))\}$, $a\in \mathbb{N}, a\geq2$
\begin{equation}\label{1reg3}
     {\rm sign}(A(k)A(k+l)B(l))=
       \begin{cases}
   +, & {\rm if}\; k+l<p(a+1) \\
   -, & {\rm if}\; k+l>p(a+1) \\
 \end{cases}
 \end{equation}
iv) In the square $\{ (k,l)\in (0, p)\times(p,2p)\}$
\begin{equation}\label{1reg4}
     {\rm sign}(A(k)A(k+l)B(l))=
       \begin{cases}
   -, & {\rm if}\; k+l<2p \\
   +, & {\rm if}\; k+l>2p \\
 \end{cases}
\end{equation}

v) For all $ (k,l)\in (0,p)\times (0,p)$
\begin{equation}\label{1reg5}
     A(k)A(k+l)B(l)>0.
     \end{equation}
\end{lem}

\begin{lem}\label{lem_signs2}
The signs of function $(-A(k)A(l)B(k+l))$ from \eqref{J_series} are distributed as follows (see Figure \ref{J22fig}):

i) In every square $$\{(k,l)\in (pa, p(a+1))\times
(pb,p(b+1))\},$$ where $a,b\in \mathbb{N}$
\begin{equation}\label{sq_ab_g2}
     {\rm sign}(-A(k)A(l)B(k+l))=
       \begin{cases}
   +, & {\rm if}\; k+l<p(a+b+1) \\
   -, & {\rm if}\; k+l>p(a+b+1) \\
 \end{cases}
\end{equation}

ii) In every set $$\{(k,l)\in (pa, p(a+1))\times
(0,p)\cup(0,p)\times (pa,p(a+1))\},$$ where $a\in \mathbb{N}, a\geq2$,
\begin{equation}\label{2.2}
     {\rm sign}(-A(k)A(l)B(k+l))=
       \begin{cases}
   -, & {\rm if}\; k+l<p(a+1) \\
   +, & {\rm if}\; k+l>p(a+1) \\
 \end{cases}
\end{equation}

iii) For all $(k,l)\in (p,2p)\times (0,p)\cup(0,p)\times(p,2p)$
\begin{equation}\label{2.4}
     -A(k)A(l)B(k+l)>0,
     \end{equation}
iv) In the square $\{(k,l)\in (0, p)\times(0,p)\}$,
\begin{equation}\label{2.2}
     {\rm sign}(-A(k)A(l)B(k+l))=
       \begin{cases}
   -, & {\rm if}\; k+l<p \\
   +, & {\rm if}\; k+l>p \\
 \end{cases}
\end{equation}
\end{lem}

\begin{figure}[h!]
\begin{multicols}{2}
\hfill
\includegraphics[scale=0.7]{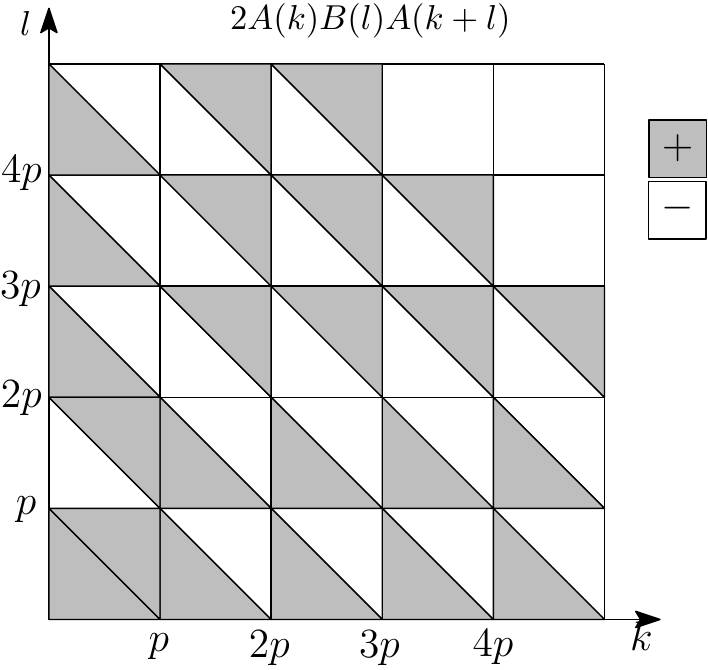}
\hfill
\caption{Signs of $A(k)A(k+l)B(l)$}
\label{J21fig}
\hfill
\includegraphics[scale=0.7]{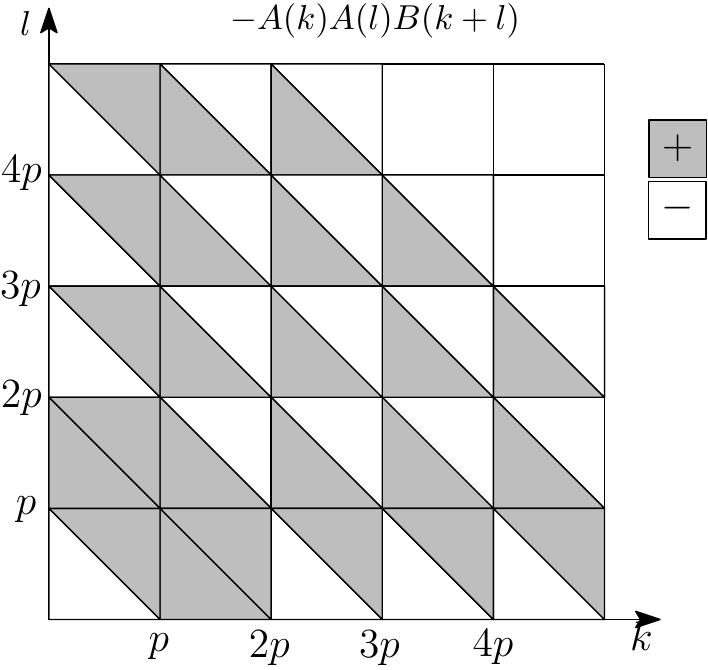}
\hfill
\caption{Signs of $(-A(k)A(l)B(k+l))$}\label{J22fig}
\end{multicols}
\end{figure}


\subsection{Beginning of Theorem \ref{th_J_est} proof.}

As in Subsection \ref{s4.4} let us denote the set of all points
from $\bN \times \bN$ belonging to the triangle with vertices at
points $(k_1,l_1),(k_2,l_2), (k_3,l_3)$  by
$\{(k_1,l_1),(k_2,l_2), (k_3,l_3)\}$ (the boundary of the triangle
is not included in this set).

For each set $F\subset \bN \times \bN$ we denote by $J(F,t)$ the
sum of the summands in \eqref{J_series} with $(k,l)\in F$.

One of our nearest aims is to prove the following statement:
\begin{lem}\label{AB_first_traingle_lem}
The following inequality is true:
\begin{equation}\label{AB_first_triangle_ineq}
\begin{split}
&J(\{(0,0),(0,p),(p,0)\},t)+\\
&J(\{(0,p),(0,2p),(p,p)\},t)+J(\{(p,0),(2p,0),(p,p)\},t)+\\
&J(\{(0,2p),(0,3p),(p,2p)\},t)+ J(\{(2p,0),(3p,0),(2p,p)\},t)+\\
&\sum_{a=3}^{+\infty}\left(J(\{(p,pa),(p,p(a+1)),(0,p(a+1))\},t)\right.\\
&\hphantom{\sum_{a=3}^{+\infty}}\left.+J(\{(pa,p),(p(a+1),p),(p(a+1),0)\},t)\right)>
\frac{9\sin^2{\frac{\pi}{p}}\sin{\frac{2\pi}{p}}}{8p^8}e^{-6t}\,\,
\forall t>0.
\end{split}
\end{equation}
\end{lem}

Then we shall show that the remaining part of the sum $J(t)$ is positive, which will complete the proof of Theorem \ref{th_J_est}.
\begin{proof}
First, let us consider the summands, corresponding to $(k,l)\in\{(0,0),(0,p),(p,0)\}$, $k=l$:
\begin{equation}\label{k_eq_l_sum}
\begin{split}
&(2A(k)A(2k)B(k)-A^2(k)B(2k))e^{-6k^2t}=\\
&\frac{9p^2k^3\sin^2\frac{\pi k}{p}\sin\frac{2\pi
k}{p}}{2(p^2-k^2)^3(p^2-4k^2)(4p^2-k^2)}\cdot e^{-6k^2t}
\end{split}
\end{equation}
Obviously, for $k\in[1,p-1]$ the expression in the right hand side
of \eqref{k_eq_l_sum} is not negative. Therefore,
\begin{equation}\label{k_eq_l_sum_ineq}
\begin{split}
&\sum_{k=1}^{p-1}(2A(k)A(2k)B(k)-A^2(k)B(2k))e^{-6k^2t}>\\
&(2A(1)A(2)B(1)-A^2(1)B(2))e^{-6t}>\frac{9\sin^2{\frac{\pi}{p}}\sin{\frac{2\pi}{p}}}{8p^8}e^{-6t}\,\,
\forall t>0.
\end{split}
\end{equation}






Next, let us group together the summands in \eqref{J_series}  corresponding to points with coordinates $(k,l)$ and $(l,k)$, $l\neq k$:
\begin{equation}\label{simm_sum}
\begin{split}
&\left(2A(k)A(k+l)B(l)-A(k)A(l)B(k+l)\right.\\
&\left.+2A(l)A(l+k)B(k)-A(l)A(k)B(l+k)\right)e^{-2(k^2+l^2+kl)t}=\\
&2(A(k)A(k+l)B(l)+A(l)A(k+l)B(k)-A(k)A(l)B(k+l))e^{-2(k^2+l^2+kl)t}=\\
&\frac{6p^2kl(k+l)(12p^2+kl-(k+l)^2)\sin\frac{\pi k}{p}\sin\frac{\pi l}{p}\sin\frac{\pi (k+l)}{p}\cdot e^{-2(k^2+l^2+kl)t}}{(p^2-k^2)(p^2-l^2)(p^2-(k+l)^2)(4p^2-k^2)(4p^2-l^2)(4p^2-(k+l)^2)}
\end{split}
\end{equation}

Let us perform the change of variables $(k,l)\to (x,y)$ by formulas $k=px, l=py$ and consider the following function:

\begin{equation}\label{Fxy}
\begin{split}
K(x,y):=\frac{6xy(x+y)(12+xy-(x+y)^2)\sin{\pi x}\sin{\pi y}\sin{\pi(x+y)}}{(1-x^2)(1-y^2)(1-(x+y)^2)(4-x^2)(4-y^2)(4-(x+y)^2)}.
\end{split}
\end{equation}
Obviously, for $(x,y)$ belonging to each from the sets
$$\{ x\in
(0,1), y\in (0,1): x+y<1\},$$
$$\{ x\in(1,2), y\in (0,1):
x+y<2\}\cup \{ x\in (0,1), y\in (1,2): x+y<2\},$$
$$\{ x\in
(2,3), y\in (0,1): x+y<3\} \cup \{ x\in (0,1), y\in (2,3):
x+y<3\}$$ both the numerator and the denominator of $K(x,y)$ are
positive.

Since for every square $(x,y)\in \{(a,a+1)\times(b,b+1)\}, a,b \in
\mathbb{N}$
\begin{equation*}
     {\rm sign}\sin{\pi x}\sin{\pi y}\sin{\pi(x+y)}=
       \begin{cases}
   +, & {\rm if}\; x+y< a+b+1 \\
   -, & {\rm if}\; x+y> a+b+1, \\
 \end{cases}
\end{equation*}
 and $12+xy-(x+y)^2$ is negative for $$(x,y)\in\{(3;+\infty)\times(0,1)\} \cup \{(0,1)\times(3;+\infty)\}$$
 such that $x+y>4$,\footnote{Indeed, function $f(x,y)=12-xy-(x+y)^2, (x,y)\in \mathbb{R}^2$ has a unique extremum
 at $(\hat{x},\hat{y})=(0,0)$, and it is maximum. Hence, maximum of $f$ on the set
 $$S=\{ (x,y)\in [3,\infty )\times[0,1]: x+y\ge 4\}\cup \{ (x,y)\in [0,1]\times[3,\infty ): x+y\ge 4\}$$
 is achieved on its boundary $\partial S$. It is easy to see that this maximum is achieved at points (3,1) and (1,3),
 and $\mbox{max}f|_S=-1$.} $K(x,y)$
 is positive for $(x,y)$ in every set $$\{(1,a);(1,a+1);(0,a+1)\}\cup
\{(a,1);(a+1,1);(a+1,0)\}, a\in \mathbb{N}, a\ge 3.$$

 Together with \eqref{k_eq_l_sum_ineq} this implies \eqref{AB_first_triangle_ineq}
\end{proof}

\section{The main step of Theorems \ref{th_tJ_est} and \ref{th_J_est} proof.}


Now let us rewrite the series $J(t)$ as
\begin{equation*}\label{J_series1}
{J}(t)=2I(t)+L(t),
\end{equation*}
where
\begin{equation}\label{I_series}
\begin{split}
&I(t)=\sum_{k,l=1}^{\infty}F_1(k,l;t),\\
&F_1(k,l;t):= A(k)A(k+l)B(l)e^{-(k^2+l^2+(k+l)^2)t},
\end{split}
\end{equation}
and
\begin{equation}\label{II_series}
\begin{split}
&L(t) = \sum_{k,l=1}^{\infty}F_2(k,l;t),\\
&F_{2}(k,l;t):=-A(k)A(l)B(k+l)e^{-(k^2+l^2+(k+l)^2)t}.
\end{split}
\end{equation}

\subsection{Positiveness of the $L(t)$ remaining part.}
Let us show, that the remaining part of the sum  $L(t)$ is
positive. Similarly to notation $J(F,t)$ introduced above, we
denote by $L(F,t)$ the sum of the summands in \eqref{II_series}
with $(k,l)\in F$ where $F\subset \bN \times \bN$.

First, we shall prove that $L(\{(m,n)\in\mathbb{N}\times\bN: m, n > p\},t)>0$:
\begin{lem}\label{II_pos1}
Let $a,b,p\in \mathbb{N}$. Then
\begin{equation}\label{II_1}
\begin{split}
& L(\{(pa, pb),(p(a+1), pb), (pa,p(b+1))\},t)+\\
& L(\{(pa, p(b+1)),(p(a+1), pb), (p(a+1),p(b+1))\},t)>0
 \end{split}
 \end{equation}
for every $t\ge 0$.
\end{lem}
\begin{proof}
According to Lemma \ref{lem_signs2}, the summands in $$L(\{(pa, pb),(p(a+1),
pb), (pa,p(b+1))\},t)$$ are positive, and the summands in $$L(\{(pa,
p(b+1)),(p(a+1), pb), (p(a+1),p(b+1))\},t)$$ are negative. If we show that
the absolute values of positive summands are not less than the absolute
values of negative summands, it would imply \eqref{II_1}.

Let us perform the change of variables $k=pa+k_1,\; l=pb+l_1$ in \eqref{II_1}. Then
\begin{equation}\label{II_1postriang}
L(\{(pa, pb),(p(a+1), pb),
(pa,p(b+1))\},t)=\sum^{p-1}_{\genfrac{}{}{0pt}{}{k_1,l_1=1}{k_1+l_1<p}}
 F_2(pa+k_1,pb+l_1;t)
\end{equation}
\begin{equation}\label{II_1negtriang}
L(\{(pa,p(b+1)),(p(a+1), pb),
(p(a+1),p(b+1))\},t)=\sum^{p-1}_{\genfrac{}{}{0pt}{}{k_1,l_1=1}{k_1+l_1>p}}
F_2(pa+k_1,pb+m_1;t)
\end{equation}
Now let us perform  the change of variables  $s=p-k_1-l_1$,
$k=k_1$ in \eqref{II_1postriang} and the change of variables
$s=k_1+l_1-p$, $k=k_1-s$ in \eqref{II_1negtriang} and consider the
following relation:
\begin{equation} \label{II_pos_to_II_neg}
\begin{split}
&\frac{F_2(pa+k,p(b+1)-k-s;t)}{|F_2(pa+k+s,p(b+1)-k;t)|}=\\
&\frac{A_1(pa+k)A_1(p(b+1)-k-s)B_1(p(a+b+1)-s)}{|A_1(pa+k+s)A_1(p(b+1)-k)B_1(p(a+b+1)+s)|}\cdot e^{6(a+b+1)ps}>\\
&\frac{A_1(pa+k)A_1(p(b+1)-k-s)B_1(p(a+b+1)-s)}{|A_1(pa+k+s)A_1(p(b+1)-k)B_1(p(a+b+1)+s)|},
\end{split}
\end{equation}
where
\begin{equation}\label{A_1}
A_1(k):=\frac{1}{k^2-p^2}
\end{equation}
\begin{equation}\label{B_1}
B_1(k):=\frac{k}{k^2-4p^2}
\end{equation}

Since $A_1(k)$ is decreasing for $k>p$ and $B_1(k)$ is decreasing for $k>2p$,
the numerator of the fraction in the last line of \eqref{II_pos_to_II_neg} is
greater than the denominator, therefore, the fraction is greater than one and
$L(\{(pa, pb),(p(a+1),pb), (pa,p(b+1))\},t)$ is greater than
$|L(\{(pa,p(b+1)),(p(a+1), pb), (p(a+1),p(b+1))\},t)|$.
\end{proof}
\begin{lem}\label{L2lem}
The following inequality is true:
\begin{equation}\label{L2}
\begin{split}
& L(\{(0, p),(p, 0), (p,p)\},t)+\\
& \sum_{a=3}^{+\infty}L(\{(pa, 0),(p(a+1), 0), (pa,p)\},t)+\\
& \sum_{a=3}^{+\infty}L(\{(0,pa),(0,p(a+1)), (p,pa)\},t)>0.
 \end{split}
 \end{equation}
\end{lem}
\begin{proof}
Because of symmetry, inequality \eqref{L2} is equivalent to the following:
\begin{equation*}
\begin{split}
& L(\{(0, p),(p, 0), (p,p)\},t)+\\
& 2\sum_{a=3}^{+\infty}L(\{(pa, 0),(p(a+1), 0), (pa,p)\},t)>0.
\end{split}
\end{equation*}
Let us rewrite sum $L(\{(0, p),(p, 0), (p,p)\},t)$ as
\begin{equation}\label{Lpos1}
L(\{(0, p),(p, 0), (p,p)\},t) = \sum_{\substack{m,l=1
\\m+l<p}}^{p-1}F_2(p-m,l+m;t)
\end{equation}
and $L(\{(pa, 0),(p(a+1), 0), (pa,p)\},t)$ as
\begin{equation}\label{Lneg1}
L(\{(pa, 0),(p(a+1), 0), (pa,p)\},t) = \sum_{\substack{m,l=1
\\m+l<p}}^{p-1}F_2(pa+m,l;t).
\end{equation}
According to Lemma \ref{lem_signs2}, all the summands in \eqref{Lpos1} are positive,
and all the summands in \eqref{Lneg1} are negative. Let us consider the absolute value
of the ratio of the summands from \eqref{Lneg1} to the summands from \eqref{Lpos1} corresponding to the same $m,l$:
\begin{equation}\label{Lneq2}
\begin{split}
&\frac{|F_2(pa+m,l;t)|}{F_{2}(p-m,l+m;t)}\leq\frac{|A(pa+m)A(l)B(pa+m+l)|}{A(p-m)A(l+m)B(p+l)}=\\
&\frac{3p-l}{(p+l)^2}\cdot\frac{m(2p-m)}{(p(a-1)+m)(p(a+1)+m)}
\cdot\frac{(p-m-l)(p+m+l)(pa+m+l)}{(p(a-2)+m+l)(p(a+2)+m+l)}=\\
&f_1(x)f_2(y)f_3(z),
\end{split}
\end{equation}
where $x=m/p$, $y=l/p$, $z=x+y$, $x,y,z\in(0,1)$, and
$$f_1(x)=\frac{x(2-x)}{(a-1+x)(a+1+x)},$$
$$f_2(y)=\frac{3-y}{(1+y)^2},$$
$$f_3(z)=\frac{(1-z^2)(a+z)}{(a+2+z)(a-2+z)}.$$
It is easy to see, that functions $f_2(y)$ and $f_3(z)$ decrease, therefore,
\begin{equation*}
\begin{split}
&f_1(x)f_2(y)f_3(z)<f_1(x)f_2(0)f_3(x)=\\
&3\cdot\frac{x(2-x)(1-x^2)(a+x)}{(a+2+x)(a-2+x)(a-1+x)(a+1+x)}.
\end{split}
\end{equation*}
Let us consider function $g(x)=x(2-x)(1-x^2)$. Since
$g'(x)=2(2x-1)(x^2-x-1)$, it reaches its maximum on the interval
$(0,1)$ at $x=1/2$, $g(1/2)=9/16$.


Taking into account \eqref{Lneq2} and all other information
written below \eqref{Lneq2} we get:
\begin{equation*}
\begin{split}
&2\sum_{a=3}^{\infty}\frac{|F_2(pa+m,l;t)|}{F_{2}(p-m,l+m;t)}<
2\sum_{a=3}^{\infty}3\cdot\frac{x(2-x)(1-x^2)(a+x)}{(a+2+x)(a-2+x)(a-1+x)(a+1+x)}<\\
&\sum_{a=3}^{\infty}\frac{27a}{8(a^2-4)(a^2-1)}<\frac{27}{8}\cdot\frac{3}{5\cdot 8}+\frac{27}{8}\sum_{a=4}^{\infty}\frac{a}{(a^2-4)^2}<\\
&\frac{81}{320}+\frac{27}{8}\int_{3}^{+\infty}\frac{xdx}{(x^2-4)^2}=\frac{81}{320}+\frac{27}{80}=\frac{189}{320}<1,
\end{split}
\end{equation*}
which completes the proof of \eqref{L2}.
 \end{proof}
\subsection{Positiveness of some part of $I(t)$ }

Let us introduce set $H\subset \mathbb N\times\mathbb N$, defined as follows:
\begin{equation}\label{setH}
\begin{split}
H=&\{(0,3p),(p,3p),(p,2p)\}\cup\{(p,2p),(2p,2p)(2p,p)\}\cup\\
&\left[\,\bigcup_{a=1}^{\infty}\{(pa,2p),(pa,3p),(p(a+1),2p)\}\right]
\end{split}
\end{equation}
In this section we prove the positiveness of the sum
$I(\mathbb{N}\times\mathbb{N}\setminus{H},t)$.
\begin{lem}\label{I_partpos1}
Let $b$ be a natural number, $b\geq 3$. Then the following inequality holds:
\begin{equation}\label{I_pos}
\begin{split}
&I(\{(0, pb),(0, p(b+1)), (p,pb)\},t)+ \\
&\sum_{a=2}^{+\infty}I(\{(pa, pb),(pa, p(b+1)), (p(a+1),pb)\},t)>0, \,\, \forall t>0.
\end{split}
\end{equation}
\end{lem}
\begin{proof}
According to Lemma \ref{lem_signs1}, the summands in $I(\{(0, pb),(0, p(b+1)), (p,pb)\},t)$ are positive and the summands in $I(\{(pa, pb),(pa, p(b+1)), (p(a+1),pb),t\})$, $b\geq 3$, $a\geq 2$, are negative. Let us
rewrite these sums as
\begin{equation}\label{I_pink}
I(\{(0, pb),(0, p(b+1)), (p,pb)\},t) = \sum_{\substack{k,l=1
\\k+l<p}}^{p-1}F_1(k,bp+l;t),
\end{equation}
and
\begin{equation}\label{I_green}
I(\{(pa, pb),(pa, p(b+1)), (p(a+1),pb),t\}) = \sum_{\substack{k,l=1\\
k+l<p}}^{p-1}F_1(pa+k,bp+l;t),
\end{equation}
and consider the absolute value of the ratio of negative summands from \eqref{I_green} to positive summands from \eqref{I_pink}
corresponding to the same $(k,l)$:
\begin{equation}\label{ratio_green_to_pink}
\begin{split}
\frac{|F_1(pa+k,pb+l;t)|}{F_1(k,bp+l;t)}<&\frac{|A(pa+k)A(p(a+b)+k+l)|}{A(k)A(pb+k+l)}=\\
&f_1(x)\cdot f_2(z),
\end{split}
\end{equation}
where $x = k/p$, $z = (k+l)/p$, $x,z \in(0,1)$ and
\begin{equation*}
f_1(x)=\frac{(x+1)(1-x)}{(a-1+x)(a+1+x)},
\end{equation*}
\begin{equation*}
f_2(z)=\frac{b-1+z}{a+b-1+z}\cdot\frac{b+1+z}{a+b+1+z}.
\end{equation*}
Since
$$f_1'(x)=-\frac{2a(1+ax+x^2)}{(a+1+x)^2(a-1+x)^2}<0, \,\, x\in(0,1),$$
function $f_1(x)$ decreases on $(0,1)$, and $f_1(x)<f_1(0)$.

It is easy to see, that $f_2(z)$ increases when $z\in(0,1)$,
therefore, taking into account \eqref{ratio_green_to_pink} and
information below this inequality, we get that
\begin{equation*}
\frac{|F_1(pa+k,pb+l;t)|}{F_1(k,bp+l;t)}<f_1(0)\cdot f_2(1) =
\frac{b(b+2)}{(a-1)(a+1)(a+b)(a+b+2)},
\end{equation*}
and
\begin{equation*}
\begin{split}
&\sum_{a=2}^{+\infty}\frac{|F_1(pa+k,pb+l;t)|}{F_1(k,bp+l;t)}<\sum_{a=2}^{+\infty}\frac{b(b+2)}{(a-1)(a+1)(a+b)(a+b+2)}<\\
&\sum_{a=2}^{+\infty}\frac{1}{(a-1)(a+1)}= \sum_{a=2}^{+\infty}\frac{1}{2}\left(\frac{1}{a-1}-\frac{1}{a+1}\right)=\\
&\lim_{n\to\infty}\sum_{a=2}^{n}\frac{1}{2}\left(1-\frac{1}{3}+\frac{1}{2}-\frac{1}{4}+...+
\frac{1}{n-2}-\frac{1}{n}+\frac{1}{n-1}-\frac{1}{n+1}\right)=\\
&\lim_{n\to\infty}\sum_{a=2}^{n}\frac{1}{2}\left(1+\frac{1}{2}-\frac{1}{n}-\frac{1}{n+1}\right)=\frac{3}{4}<1
\end{split}
\end{equation*}
\end{proof}
\begin{lem}\label{I_partpos}
The following inequalities hold:
\begin{equation}\label{I_pos}
 \frac{1}{3}I(\{(0, p),(p, 0), (p,p)\},t)+ I(\{(p, p),(2p, 0), (2p,p)\},t)>0
\end{equation}
\begin{equation}\label{II_pos}
 \frac{1}{15}I(\{(0, p),(p, 0), (p,p)\},t)+I(\{(2p,p),(3p,0),(3p,p)\},t)>0
 \end{equation}
\end{lem}
\begin{proof}
According to Lemma \ref{lem_signs1}, the summands in $I(\{(0,
p),(p, 0), (p,p)\},t)$ are positive and the summands in $I(\{(pa,
p),(p(a+1), 0), (p(a+1),p)\},t)$, $a=1, 2$, are negative. Let us
rewrite these sums as
\begin{equation*}
I(\{(0, p),(p, 0), (p,p)\},t) = \sum_{\substack{k,l=1
\\k+l<p}}^{p-1}F_1(p-k,p-l;t),
\end{equation*}
and
\begin{equation*}
I(\{(pa, p),(p(a+1), 0), (p(a+1),p)\},t) = \sum_{\substack{k,l=1\\
k+l<p}}^{p-1}F_1(p(a+1)-k,p-l;t),
\end{equation*}
and consider the ratio of a negative summand to a positive summand
corresponding to the same $(k,l)$:
\begin{equation}\label{Iratio}
\begin{split}
&\frac{|F_1(p(a+1)-k,p-l;t)|}{F_1(p-k,p-l;t)}=\\
&\frac{|A((a+1)p-k)A((a+2)p-k-l)|}{A(p-k)A(2p-k-l)}\cdot e^{-2ap((3+a)p-2k-l)t}<\\
&f_1(x)\cdot f_2(x+y),
\end{split}
\end{equation}
where $x = k/p$, $y = l/p$, $x,y \in(0,1)$, $x+y<1$ and
\begin{equation*}
f_1(x)=\frac{x(2-x)}{(a-x)((a+2)-x)},
\end{equation*}
\begin{equation*}
f_2(x+y)=\frac{1-x-y}{a+1-x-y}\cdot\frac{3-x-y}{a+3-x-y}.
\end{equation*}
It is easy to see, that function $f_2(x+y)$ decreases as a
function of $y$ on $(0,1)$, therefore
\begin{equation*}
f_2(x+y)<f_2(x+0) = \frac{1-x}{a+1-x}\cdot\frac{3-x}{a+3-x},
\end{equation*}
and
\begin{equation}\label{rel_a}
f_1(x)f_2(x+y)<f_1(x)f_2(x)=\frac{x(2-x)(1-x)(3-x)}{(a-x)(a+2-x)(a+1-x)(a+3-x)}.
\end{equation}
When $a=1$,
\begin{equation}\label{rel_1}
f_1(x)f_2(x)=\frac{x}{4-x}<\frac{1}{3},
\end{equation}
and for $a=2$
\begin{equation}\label{rel_2}
f_1(x)f_2(x)=\frac{x}{4-x}\cdot\frac{1-x}{5-x}<\frac{1}{3}\cdot\frac{1}{5}=\frac{1}{15}.
\end{equation}
Relations \eqref{Iratio}, \eqref{rel_1}, and \eqref{rel_2} imply,
\eqref{I_pos}, \eqref{II_pos}.
\end{proof}
\begin{lem}\label{I_partpos2}
The following inequality holds:
\begin{equation}\label{I_red_green}
I(\{(p, p),(2p, p), (p,2p)\},t)+ \sum_{a=4}^{+\infty}I(\{(p, pa),(2p, pa), (p,p(a+1))\},t)>0
\end{equation}
\end{lem}
\begin{proof}
According to Lemma \ref{lem_signs1}, the summands in $I(\{(p, p),(2p, p), (p,2p)\},t)$ are positive and the summands in $I(\{(p, pa),(2p, pa), (p,p(a+1))\},t)$, $a\geq 4$, are negative. Let us
rewrite these sums as
\begin{equation}\label{I_red}
I(\{(p, p),(2p, p), (p,2p)\},t) = \sum_{\substack{k,l=1
\\k+l<p}}^{p-1}F_1(p+k,p+l;t),
\end{equation}
and
\begin{equation}\label{I_dgreen}
I(\{(p, pa),(2p, pa), (p,p(a+1))\},t) = \sum_{\substack{k,l=1\\
k+l<p}}^{p-1}F_1(p+k,pa+l;t),
\end{equation}
and consider the absolute value of the ratio of negative summands from \eqref{I_dgreen} to positive summands from \eqref{I_red}
corresponding to the same $(k,l)$:
\begin{equation}\label{ratio_dgreen_to_red}
\begin{split}
\frac{|F_1(p+k,pa+l;t)|}{F_1(p+k,p+l;t)}<&\frac{|B(pa+l)A(p(a+1)+k+l)|}{B(p+l)A(2p+k+l)}=\\
&f_1(x)\cdot f_2(z),
\end{split}
\end{equation}
where $x = l/p$, $z = (k+l)/p$, $x,z \in(0,1)$ and
\begin{equation*}
f_1(x)=\frac{(a+x)}{(1+x)}\cdot\frac{1-x}{a-2+x}\cdot\frac{3+x}{a+2+x},
\end{equation*}
\begin{equation*}
f_2(z)=\frac{1+z}{a+z}\cdot\frac{3+z}{a+2+z}.
\end{equation*}
It is easy to see, that for $a\geq4$ the first two multiples in
$f_1(x)$ are decreasing functions, and the last one is increasing
function as well as all multiples in $f_2(z)$, when $x,z\in(0,1)$.
Therefore,
$$
   f_1(x)<\frac{4a}{(a-2)(a+3)},\; x\in (0,1);\quad
   f_2(z)<\frac{8}{(a+1)(a+3)},\; z\in (0,1)
$$
and
\begin{equation*}
\frac{|F_1(p+k,pa+l;t)|}{F_1(p+k,p+l;t)}<f_1(x)\cdot
f_2(z)<\frac{32a}{(a-2)(a+1)(a+3)^2}.
\end{equation*}
Note that $(a-2)(a+1)(a+3)^2=a^4+5a^3+a^2-21a-18>a^4$ for $a\ge 4$
and therefore
\begin{equation*}
\begin{split}
&\sum_{a=4}^{\infty}\frac{|F_1(p+k,pa+l;t)|}{F_1(p+k,p+l;t)}<\sum_{a=4}^{\infty}\frac{32a}{(a-2)(a+1)(a+3)^2}<\\
&32\left(\frac{4}{2\cdot 5\cdot 49}+\frac{5}{3\cdot 6\cdot
64}+\frac{6}{4\cdot 7\cdot 81}+\frac{7}{5\cdot 8\cdot
100}+\frac{8}{6\cdot 9\cdot 121}+\sum_{a=9}^\infty
\frac{1}{a^3}\right)=\\
&\frac{64}{5\cdot 49}+\frac{5}{3\cdot 6\cdot2}+\frac{48}{7\cdot 81}+\frac{7}{5\cdot
25}+\frac{32\cdot 4}{33\cdot 99}+\sum_{a=9}^\infty
\frac{32}{a^3}<\\
&\frac{4\cdot 16}{5\cdot48}+\frac{1}{6}+\frac{49}{7\cdot77}+\frac{7}{119}+\frac{4}{96}+\sum_{a=9}^\infty
\frac{32}{a^3}=\\
&\frac{1}{3}+\frac{1}{6}+\frac{1}{11}+\frac{1}{17}+\frac{1}{24}+\int_8^\infty
\frac{32dx}{x^3}<1
\end{split}
\end{equation*}
which completes the proof of \eqref{I_red_green}
\end{proof}
\begin{lem}\label{I_partpos3}
The following inequality holds:
\begin{equation}\label{I_blue_yellow}
\begin{split}
&I(\{(0, 2p),(p, p), (p,2p)\},t)+ \sum_{a=3}^{+\infty}I(\{(p(a-1), 2p),(pa, p), (pa,2p)\},t)+\\
&I(\{(p, 3p),(p, 4p), (2p,3p)\},t)>0
\end{split}
\end{equation}
\end{lem}
\begin{proof}
According to Lemma \ref{lem_signs1}, the summands in $I(\{(0, 2p),(p, p), (p,2p)\},t)$ are positive and the summands in $I(\{(p(a-1), 2p),(pa, p), (pa,2p)\},t)$, $a\geq 3$, and $I(\{(p, 3p),(p, 4p), (2p,3p)\},t)$ are negative. Let us first consider $$I(\{(0, 2p),(p, p), (p,2p)\},t)$$ and $$I(\{(p(a-1), 2p),(pa, p), (pa,2p)\},t)$$ and rewrite these sums as
\begin{equation}\label{I_blue}
I(\{(0, 2p),(p, p), (p,2p)\},t) = \sum_{\substack{k,l=1
\\k+l<p}}^{p-1}F_1(p-k,2p-l;t),
\end{equation}
and
\begin{equation}\label{I_yellow}
I(\{(p(a-1), 2p),(pa, p), (pa,2p)\},t) = \sum_{\substack{k,l=1\\
k+l<p}}^{p-1}F_1(ap-k,2p-l;t).
\end{equation}
The absolute value of the ratio of negative summands from \eqref{I_yellow} to positive summands from \eqref{I_blue}
corresponding to the same $(k,l)$ can be estimated as follows:
\begin{equation}\label{ratio_yellow_to_blue}
\begin{split}
\frac{|F_1(ap-k,2p-l;t)|}{F_1(p-k,2p-l;t)}<&\frac{|A(ap-k)A(p(a+2)-k-l)|}{A(p-k)A(3p-k-l)}=\\
&f_1(x)\cdot f_2(z),
\end{split}
\end{equation}
where $x = k/p$, $z = (k+l)/p$, $x,z \in(0,1)$ and
\begin{equation*}
f_1(x)=\frac{x(2-x)}{(a-1-x)(a+1-x)},
\end{equation*}
\begin{equation*}
f_2(z)=\frac{2-z}{a+1-z}\cdot\frac{4-z}{a+3-z}.
\end{equation*}
It is easy to see, that for $a\geq3$ all multiples in $f_2(z)$ are
decreasing functions on $(0,1)$, therefore $f_2(z)<f_2(0)$. The
numerator $x(2-x)$ of $f_1(x)$ increases for ${x\in (0,1)}$, and its
denominator $(a-1-x)(a+1-x)$ decreases for $x\in (0,1)$ (since
$a\ge 3$). Therefore, $f_1(x)$ increases for $x\in (0,1)$, and
$f(x)<f_1(1)$. So, in virtue of~\eqref{ratio_yellow_to_blue},
\begin{equation*}
\frac{|F_1(ap-k,2p-l;t)|}{F_1(p-k,2p-l;t)}<f_1(x)\cdot
f_2(z)<f_1(1)f_2(0)=\frac{8}{a(a-2)(a+1)(a+3)}.
\end{equation*}
Since $a(a-2)(a+1)(a+3)=a^4+2a^3-5a^2-6a>a^4,$
\begin{equation}\label{ratio_yellow_to_blue_fin}
\begin{split}
&\sum_{a=3}^{\infty}\frac{|F_1(ap-k,2p-l;t)|}{F_1(p-k,2p-l;t)}<
\sum_{a=3}^{\infty}\frac{8}{a(a-2)(a+1)(a+3)}<\\
&8\left(\frac{1}{3\cdot 1\cdot 4\cdot
6\cdot}+\sum_{a=4}^\infty\frac{1}{a^4}\right)<\frac{1}{9}+\int_3^\infty
\frac{8dx}{x^4}<\frac{17}{81}.
\end{split}
\end{equation}
Now let us consider  $I(\{(0, 2p),(p, p), (p,2p)\},t)$ and
$I(\{(p, 3p),(p, 4p), (2p,3p)\},t)$ and rewrite the first sum as
\eqref{I_blue} and the second one as follows:
\begin{equation}\label{I_yellow1}
I(\{(p, 3p),(p, 4p), (2p,3p)\},t) = \sum_{\substack{k,l=1\\
k+l<p}}^{p-1}F_1(p+k,3p+l;t).
\end{equation}
The absolute value of the ratio of negative summands from
\eqref{I_yellow1} to positive summands from \eqref{I_blue}
corresponding to the same $(k,l)$ can be estimated in the
following way:
\begin{equation}\label{ratio_yellow_to_blue1}
\begin{split}
\frac{|F_1(p+k,3p+l;t)|}{F_1(p-k,2p-l;t)}<&\frac{|A(p+k)B(3p+l)A(4p+k+l)|}{A(p-k)B(2p-l)A(3p-k-l)}=\\
&f_1(x)\cdot f_2(y)\cdot f_3(z),
\end{split}
\end{equation}
where $x = k/p$, $y = l/p$, $z=x+y$, $x,y,z \in(0,1)$ and
\begin{equation*}
f_1(x)=\frac{2-x}{2+x},
\end{equation*}
\begin{equation*}
f_2(y)=\frac{4-y}{2-y}\cdot\frac{y}{1+y}\cdot\frac{3+y}{5+y},
\end{equation*}
\begin{equation*}
f_3(z)=\frac{(2-z)(4-z)}{(3+z)(5+z)}.
\end{equation*}
It is easy to see, that functions $f_1(x)$, $f_3(z)$ decrease, $x,z\in(0,1)$, and function $f_2(y)$
increases, $y\in(0,1)$. Therefore,
\begin{equation*}
\frac{|F_1(p+k,3p+l;t)|}{F_1(p-k,2p-l;t)}<f_1(x)\cdot f_2(y)\cdot f_3(z)<f_1(0)f_2(1)f_3(0)=\frac{8}{15},
\end{equation*}
which, together with \eqref{ratio_yellow_to_blue_fin}, completes the proof of the lemma.
\end{proof}

Lemmas \ref{I_partpos1}-\ref{I_partpos3} imply
\begin{cor}\label{cor0}
The following inequality is true:
\begin{equation}\label{LplusI}
I(\mathbb{N}\times\mathbb{N}\setminus{H},t)>0,
\end{equation}
where the set $H$ is defined in \eqref{setH}.
\end{cor}
\subsection{Positiveness of the remaining part of $I(t)$ together with certain positive summands from $L(t)$}

The next three lemmas prove, that the positive sum $$L(\{(p,p),(2p,0), (2p,p)\},t)+L(\{(p,p),(0,2p), (p,2p)\},t)$$
compensates for the remaining negative sum $I(H,t)$, where $H$ is the set defined in \eqref{setH}.
\begin{lem}\label{I_plus_L_pos1}
The following inequality holds:
\begin{equation}\label{IL1}
\begin{split}
&\frac{1}{2}\left(L(\{(p,p),(2p,0), (2p,p)\},t)+L(\{(p,p),(0,2p), (p,2p)\},t)\right)+ \\
&2\sum_{a=1}^{+\infty}I(\{(pa, 2p),(pa, 3p), (p(a+1),2p)\},t)>0.
\end{split}
\end{equation}
\end{lem}
\begin{proof}
Because of the symmetry,
$$
L(\{(p,p),(2p,0), (2p,p)\},t)=L(\{(p,p),(0,2p), (p,2p)\},t),
$$
therefore, inequality \eqref{IL1} is equivalent to the following:
\begin{equation}\label{IL11}
\frac{1}{2}\cdot L(\{(p,p),(0,2p),
(p,2p)\},t)+\sum_{a=1}^{+\infty}I(\{(pa, 2p),(pa, 3p),
(p(a+1),2p)\},t)>0
\end{equation}
Let us rewrite these sums as
\begin{equation}\label{L_red1}
L(\{(p,p),(0,2p), (p,2p)\},t) = \sum_{\substack{k,l=1
\\k+l<p}}^{p-1}F_2(p-k,p+k+l;t),
\end{equation}
and
\begin{equation}\label{I_grey1}
I(\{(pa, 2p),(pa, 3p), (p(a+1),2p)\},t) = \sum_{\substack{k,l=1\\
k+l<p}}^{p-1}F_1(ap+k,2p+l;t).
\end{equation}
The absolute value of the ratio of negative summands from \eqref{I_grey1} to positive summands from \eqref{L_red1}
corresponding to the same $(k,l)$ can be estimated as follows:
\begin{equation}\label{ratio_grey_to_red}
\begin{split}
\frac{|F_1(ap+k,2p+l;t)|}{F_2(p-k,p+k+l;t)}<&\frac{|A(ap+k)A(p(a+2)+k+l)|}{-A(p-k)A(p+k+l)}=\\
&f_1(x)\cdot f_2(z),
\end{split}
\end{equation}
where $x = k/p$, $z = (k+l)/p$, $x,z \in(0,1)$ and
\begin{equation*}
f_1(x)=\frac{x}{a-1+x}\cdot\frac{2-x}{a+1+x},
\end{equation*}
\begin{equation*}
f_2(z)=\frac{z}{a+1+z}\cdot\frac{2+z}{a+3+z}.
\end{equation*}
It is easy to see, that for $a\geq1$ all multiples in $f_2(z)$ are increasing functions on $(0,1)$,
therefore $f_2(z)<f_2(1)$. Next, since the first multiple in $f_1(x)$ is a non-decreasing function,
and the second is a decreasing function,
$$f_1(x)<\frac{1}{a}\cdot\frac{2}{a+1},$$
therefore,
\begin{equation*}
\frac{|F_1(ap+k,2p+l;t)|}{F_2(p-k,p+k+l;t)}<f_1(x)\cdot f_2(z)<\frac{6}{a(a+1)(a+2)(a+4)},
\end{equation*}
and since $a(a+1)(a+2)(a+4)>a^4$,
\begin{equation*}
\begin{split}
&\sum_{a=1}^{\infty}\frac{|F_1(ap+k,2p+l;t)|}{F_2(p-k,p+k+l;t)}<
\sum_{a=1}^{\infty}\frac{6}{a(a+1)(a+2)(a+4)}<\\
&\frac{6}{1\cdot 2\cdot 3\cdot 5\cdot}+\frac{6}{2\cdot 3\cdot
4\cdot
6\cdot}+\sum_{a=3}^\infty\frac{6}{a^4}<\frac{29}{120}+\int_2^\infty
\frac{6dx}{x^4}<\frac{1}{2}
\end{split}
\end{equation*}
which completes the proof of \eqref{IL11} and \eqref{IL1}.
\end{proof}
\begin{lem}\label{I_plus_L_pos2}
The following inequality holds:
\begin{equation}\label{IL2}
\begin{split}
&\frac{1}{5}\left(L(\{(p,p),(2p,0), (2p,p)\},t)+L(\{(p,p),(0,2p), (p,2p)\},t)\right)+ \\
&2I(\{(0, 3p),(p, 3p), (p,2p)\},t)>0.
\end{split}
\end{equation}
\end{lem}
\begin{proof}
Because of the symmetry,
$$
L(\{(p,p),(2p,0), (2p,p)\},t)=L(\{(p,p),(0,2p), (p,2p)\},t),
$$
inequality \eqref{IL2} is equivalent to the following one:
\begin{equation}\label{IL21}
\frac{1}{5}L(\{(p,p),(0,2p), (p,2p)\},t)+I(\{(0, 3p),(p, 3p), (p,2p)\},t)>0
\end{equation}
Let us rewrite these sums as
\begin{equation}\label{L_red2}
L(\{(p,p),(0,2p), (p,2p)\},t) = \sum_{\substack{k,l=1
\\k+l<p}}^{p-1}F_2(p-k,2p-l;t),
\end{equation}
and
\begin{equation}\label{I_grey2}
I(\{(0, 3p),(p, 3p), (p,2p)\},t) = \sum_{\substack{k,l=1\\
k+l<p}}^{p-1}F_1(p-k,3p-l;t),
\end{equation}
and consider the ratio of negative summands from \eqref{I_grey2}
to positive summands from \eqref{L_red2} corresponding to the same
$(k,l)$:
\begin{equation}\label{ratio_grey_to_red2}
\begin{split}
\frac{|F_1(p-k,3p-l;t)|}{F_2(p-k,2p-l;t)}<&\frac{|B(3p-l)A(4p-k-l)|}{-A(2p-l)B(3p-k-l)}=\\
&f_1(x)\cdot f_2(z),
\end{split}
\end{equation}
where $x = l/p$, $z = (k+l)/p$, $x,z \in(0,1)$ and
\begin{equation*}
f_1(x)=\frac{(3-x)^2}{5-x},
\end{equation*}
\begin{equation*}
f_2(z)=\frac{1-z}{(3-z)^2}.
\end{equation*}
Since
$$f'_1(x)=-\frac{(3-x)(7-x)}{(5-x)^2}<0, \,\,x\in(0,1),$$
$$f'_2(z)=-\frac{1+z}{(3-x)^3}<0, \,\,z\in(0,1),$$
both functions decrease on $(0,1)$, therefore,
\begin{equation*}
\frac{|F_1(p-k,3p-l;t)|}{F_2(p-k,2p-l;t)}<f_1(x)f_2(z)<f_1(0)f_2(0)=\frac{1}{5},
\end{equation*}
which completes the proof of \eqref{IL21} and \eqref{IL2}.
\end{proof}
\begin{lem}\label{I_plus_L_pos3}
The following inequality holds:
\begin{equation}\label{IL3}
\begin{split}
&\frac{1}{9}(L(\{(p,p),(2p,0), (2p,p)\},t)+L(\{(p,p),(0,2p), (p,2p)\},t))+ \\
&2I(\{(p, 2p),(2p, 2p), (2p,p)\},t)>0.
\end{split}
\end{equation}
\end{lem}
\begin{proof}
Because of the symmetry, $$L(\{(p,p),(2p,0), (2p,p)\},t)=L(\{(p,p),(0,2p), (p,2p)\},t)$$, therefore, inequality \eqref{IL3} is equivalent to the following:
\begin{equation}\label{IL31}
\frac{1}{9}L(\{(p,p),(2p,0), (2p,p)\},t)+I(\{(p, 2p),(2p, 2p), (2p,p)\},t)>0
\end{equation}
Let us rewrite these sums as
\begin{equation}\label{L_red3}
L(\{(p,p),(2p,0), (2p,p)\},t) = \sum_{\substack{k,l=1
\\k+l<p}}^{p-1}F_2(2p-k,p-l;t),
\end{equation}
and
\begin{equation}\label{I_grey3}
I(\{(p, 2p),(2p, 2p), (2p,p)\},t) = \sum_{\substack{k,l=1\\
k+l<p}}^{p-1}F_1(2p-k,2p-l;t),
\end{equation}
and consider the ratio of negative summands from \eqref{I_grey3} to positive summands from \eqref{L_red3}
corresponding to the same $(k,l)$:
\begin{equation}\label{ratio_grey_to_red3}
\begin{split}
\frac{|F_1(2p-k,2p-l;t)|}{F_2(2p-k,p-l;t)}<&\frac{|B(2p-l)A(4p-k-l)|}{-A(p-l)B(3p-k-l)}=\\
&f_1(x)\cdot f_2(z),
\end{split}
\end{equation}
where $x = l/p$, $z = (k+l)/p$, $x,z \in(0,1)$ and
\begin{equation*}
f_1(x)=\frac{(2-x)^2}{4-x},
\end{equation*}
\begin{equation*}
f_2(z)=\frac{1-z}{(3-z)^2}.
\end{equation*}
Since
$$f'_1(x)=-\frac{(6-x)(2-x)}{(4-x)^2}<0, \,\,x\in(0,1),$$
$$f'_2(z)=-\frac{1+z}{(3-x)^3}<0, \,\,z\in(0,1),$$
both functions decrease on $(0,1)$, therefore,
\begin{equation*}
\frac{|F_1(2p-k,2p-l;t)|}{F_2(2p-k,p-l;t)}<f_1(x)f_2(z)<f_1(0)f_2(0)=\frac{1}{9},
\end{equation*}
which completes the proof of \eqref{IL31}.
\end{proof}
Lemmas \ref{I_plus_L_pos1}-\ref{I_plus_L_pos1} imply
\begin{cor}\label{cor1}
The following inequality is true:
\begin{equation*}
I(H,t)+L(\{(p,p),(2p,0), (2p,p)\},t)+L(\{(p,p),(0,2p), (p,2p)\},t)>0.
\end{equation*}
where the set $H$ is defined in \eqref{setH}.
\end{cor}
Lemma \ref{AB_first_traingle_lem} together with Corollaries \ref{cor0} and \ref{cor1} completes the proof of Theorems \ref{th_J_est}, \ref{th_tJ_est} and \ref{thJ_24est}.

In their turn, Theorems \ref{thJ_1est}, \ref{thJ_24est} together with inequality \eqref{J_3_est} constitute the proof of Theorems \ref{main_est_th} and \ref{stabiliz_theorem}.

\end{document}